%% file: hfi.tex
\documentclass [11pt]{amsart}
\usepackage {amsmath, amsthm, amssymb, amscd, mathrsfs, graphicx, extarrows, color, url, enumerate, epsf, hyperref}
\usepackage{tikz}
\usetikzlibrary{positioning}
\usepackage{subfig}

\usepackage[all,knot,arc]{xy}
\usepackage[text={6in,8.5in},centering,letterpaper,dvips]{geometry}

\newtheorem {theorem}{Theorem}[section]
\newtheorem {lemma}[theorem]{Lemma}
\newtheorem {proposition}[theorem]{Proposition}
\newtheorem {corollary}[theorem]{Corollary}
\newtheorem {conjecture}[theorem]{Conjecture}

\theoremstyle{definition}
\newtheorem {definition}[theorem]{Definition}
\theoremstyle{remark}
\newtheorem {remark}[theorem]{Remark}

\newcommand\Z{\mathbb{Z}}
\newcommand\R{\mathbb{R}}
\newcommand\Q{\mathbb{Q}}

\newcommand\X{\mathbb{X}}

\def\H{\mathcal{H}}

\newcommand \bJ {\bar{J}}
\newcommand \bh {\mkern3mu \overline{\mkern-3mu H \mkern-1mu} \mkern1mu}
\newcommand \bH {\overline{\H}}
\newcommand\T{\mathbb{T}}
\newcommand\Ta{\mathbb{T}_\alpha}
\newcommand\Tb{\mathbb{T}_\beta}
\newcommand\Tg{\mathbb{T}_\gamma}

\def\Sym{\mathrm{Sym}}

\def\del {{\partial}}

\def\spinc {{\operatorname{spin^c}}}
\def\Spinc {{\operatorname{Spin^c}}}
\def\J {\mathcal{J}}

\def\fin\qedhere
\def\pr {{\text{pr}}}
\DeclareMathOperator{\id}{id}

\def\from {{\leftarrow}}
\def\M {\mathcal {M}}
\def\Mh {\widehat{\M}}
\def\s{\mathfrak s}
\def\t{\mathbf t}
\def\x{\mathbf x}
\def\y{\mathbf y}

\def\Tower{\mathcal{T}^+}
\def\Ring {\mathcal R}

\def\dhorz{\del_{\operatorname{horz}}}
\def\dvert{\del_{\operatorname{vert}}}
\newcommand\alphas{\boldsymbol\alpha}
\newcommand\betas{\boldsymbol\beta}
\newcommand\gammas{\boldsymbol\gamma}
\newcommand\deltas{\boldsymbol\delta}

\def\AI {\mathit{AI}}

\def\CI{\mathit{CI}}

\def\gr{\mathrm{gr}}

\def\coker{{\operatorname{Coker}}}
\def\Id  {{\operatorname{Id}}}

\def\fin\qedhere
\def\from {{\leftarrow}}
\def\CFinfty {\CF^\infty}
\def\HFinfty {\HF^\infty}
\def\T{\mathcal{T}}
\def\ccdot {\! \cdot \!}
\def\pin{\operatorname{Pin}(2)}

\def\Cone{\mathit{Cone}}

\def\sarkar{\varsigma}
\def\Epin {E \kern-1pt \pin}
\def\Bpin{B \! \pin}

\newcommand{\bunderline}[1]{\underline{#1\mkern-2mu}\mkern2mu }
\newcommand{\sunderline}[1]{\underline{#1\mkern-3mu}\mkern3mu }
\def\du {\bar{d}}
\def\dl {\bunderline{d}}
\def\Vu {\overline{V}}
\def\Vl {\sunderline{V}}

\def\CF {\mathit{CF}}
\def\HF {\mathit{HF}}
\newcommand\HFhat{\widehat{\HF}}
\newcommand\CFhat{\widehat{\CF}}
\newcommand\HFp {\HF^+}
\newcommand \CFp {\CF^+}
\newcommand \CFm {\CF^-}
\newcommand \HFm {\HF^-}
\newcommand \CFinf {\CF^{\infty}}
\newcommand \CFi {\CF^{\infty}}
\newcommand \HFinf {\HF^{\infty}}
\newcommand \CFo {\CF^{\circ}}
\newcommand \HFo {\HF^{\circ}}
\def\HFred{\HF_{\operatorname{red}}}

\def\CFI {\mathit{CFI}}
\def\HFI {\mathit{HFI}}
\newcommand\HFIhat{\widehat{\HFI}}

\newcommand\HFIp {\HFI^+}
\newcommand \CFIp {\CFI^+}
\newcommand \CFIm {\CFI^-}
\newcommand \HFIm {\HFI^-}
\newcommand \CFIinf {\CFI^{\infty}}
\newcommand \HFIinf {\HFI^{\infty}}
\newcommand \CFIo {\CFI^{\circ}}
\newcommand \HFIo {\HFI^{\circ}}
\def\HFIred{\HFI_{\operatorname{red}}}

\def\CFKm{\mathit{CFK}^-}
\def\CFKi{\mathit{CFK}^{\infty}}
\def\HFKhat{\widehat{\mathit{HFK}}}

\def\CFKinfty{\CFKi}
\def\inv{\iota}
\def\delinv{\partial^{\inv}}

\def\co{\colon\thinspace}

\def\mm {\mathfrak{m}}

\def\Arf {\operatorname{Arf}}
\def\aa {\mathfrak{a}}
\def\bb{\mathfrak{b}}
\def\swf{\mathit{SWF}}
\def\borel{\operatorname{borel}}
\def\rp{\mathbb{RP}}
\def\Filt{\mathcal{F}}
\def\CFKI{\mathit{CFKI}}

\def\Ch{\mathscr{C}}
\newcommand \bFilt {\mkern3mu \overline{\mkern-3mu \Filt \mkern-1mu} \mkern1mu}
\def\hF {\widehat{F}}
\def\xu {\bar{x}}
\def\xl {\sunderline{x}}

\begin{document}

\title{Involutive Heegaard Floer homology}

\author[Kristen Hendricks]{Kristen Hendricks}
\author[Ciprian Manolescu]{Ciprian Manolescu}
\thanks {CM was partially supported by NSF grant DMS-1402914.}

\address {Department of Mathematics, UCLA, 520 Portola Plaza\\ 
Los Angeles, CA 90095}
\email {hendricks@math.ucla.edu}
\email {cm@math.ucla.edu}

\begin {abstract}
Using the conjugation symmetry on Heegaard Floer complexes, we define a three-manifold invariant called involutive Heegaard Floer homology, which is meant to correspond to $\Z_4$-equivariant Seiberg-Witten Floer homology. Further, we obtain two new invariants of homology cobordism, $\dl$ and $\du$, and two invariants of smooth knot concordance, $\Vl_0$ and $\Vu_0$. We also develop a formula for the involutive Heegaard Floer homology of large integral surgeries on knots.  We give explicit calculations in the case of L-space knots and thin knots. In particular, we show that $\Vl_0$ detects the non-sliceness of the figure-eight knot. Other applications include constraints on which large surgeries on alternating knots can be homology cobordant to other large surgeries on alternating knots.
\end {abstract}

\maketitle

\section {Introduction}

In \cite{Triangulations}, the second author resolved the remaining cases of the triangulation conjecture, by showing that there are manifolds of every dimension $n\geq 5$ that cannot be triangulated. The proof involves the construction of a $\pin$-equivariant version of Seiberg-Witten Floer homology. ($\pin$ is the group consisting of two copies of the complex unit circle with a map $j$ interchanging them such that $ij = -ji$ and $j^2=-1$.) From the module structure of this homology one extracts three non-additive maps 
$$\alpha, \beta, \gamma \co \Theta_{\Z}^3 \to \Z,$$
where $\Theta_\Z^3$ denotes the three-dimensional homology cobordism group. The maps $\alpha, \beta, \gamma$ are analogous to the Fr{\o}yshov-type correction terms arising from monopole or Heegaard Floer homology \cite{FroyshovSW, KMOS, AbsGraded}, but have the additional property that their reduction mod $2$ is equal to the Rokhlin invariant. Furthermore, we have $\beta(-Y)=-\beta(Y)$ for any homology sphere $Y$. This implies the non-existence of elements of order $2$ in $\Theta_\Z^3$ with odd Rokhlin invariant, which in turn disproves the triangulation conjecture---in view of the previous work of Galewski-Stern and Matumoto \cite{GS, Matumoto}.

The construction of $\pin$-equivariant version of Seiberg-Witten Floer homology in \cite{Triangulations} uses finite dimensional approximation, following \cite{Spectrum}, and it is only applicable to rational homology spheres. Doing calculations with this method is rather difficult, and at the moment only accessible when one has an explicit description of the Seiberg-Witten Floer complex, e.g. for Seifert fibrations \cite{Stoffregen}. An alternative construction was given by Lin in \cite{Lin}; this refines the Kronheimer-Mrowka definition of monopole Floer homology from \cite{KMBook}, and works for arbitrary $3$-manifolds. Recently, Lin established an exact triangle for his theory, which  allowed him to compute it for many examples \cite{LinExact}. However, some of Lin's computations make use of the isomorphism between monopole and Heegaard Floer homology \cite{KLT1, CGH2}.
 
Indeed, among Floer theories for three-manifolds, the one most amenable to computations is Heegaard Floer homology. This was introduced by Ozsv\'ath and Szab\'o in the early 2000's \cite{HolDisk, HolDiskTwo, HolDiskFour}. The definition starts with a pointed Heegaard diagram $\H=(\Sigma, \alphas, \betas, z)$ representing a three-manifold $Y$. One then takes the Lagrangian Floer homology of two tori in the symmetric product of the Heegaard surface $\Sigma$. There are four flavors of this construction, denoted $\HFhat$, $\HFp$, $\HFm$, and $\HFinfty$; we will use $\HFo$ to denote any of them, with $\circ \in \{\widehat{\phantom{a}}, +, -, \infty\}$.  Heegaard Floer homology was shown to be isomorphic to monopole Floer homology \cite{KLT1, CGH2}. One reason why Heegaard Floer homology is computationally tractable is because of the surgery formulas \cite{Knots, IntSurg, RatSurg, LinkSurg} which relate it to a similar invariant for knots, knot Floer homology \cite{Knots, RasmussenThesis}. In view of this, it would be desirable to construct a Heegaard Floer analog of $\pin$-equivariant Seiberg-Witten Floer homology. 

In the present paper we define a Heegaard Floer analog of $\Z_4$-equivariant Seiberg-Witten Floer homology, which we call {\em involutive Heegaard Floer homology}. Here, $\Z_4$ is the subgroup of $\pin$ generated by the element $j$. The $\Z_4$-equivariant theory does not have the full power of a $\pin$-equivariant one; in particular, one cannot use it to give another disproof of the triangulation conjecture, because the resulting homology cobordism invariants do not capture the Rokhlin invariant. Nevertheless, we will see that the information in involutive Heegaard Floer homology goes beyond that in ordinary Heegaard Floer homology. Moreover, we develop a formula for the involutive Heegaard Floer homology of large surgeries on knots, and this leads to many explicit calculations.

Both Seiberg-Witten and Heegaard Floer homology decompose as direct sums, indexed by the $\spinc$ structures on the $3$-manifold. In Seiberg-Witten theory, the element $j \in \pin$ gives a symmetry of the equations that takes a $\spinc$ structure to its conjugate, $\s \to \bar \s$. In Heegaard Floer theory, there is a similar conjugation symmetry, given by switching the orientation of the Heegaard surface, as well as swapping the $\alpha$ and the $\beta$ curves: 
$$ (\Sigma, \alphas, \betas, z) \to (-\Sigma, \betas, \alphas, z).$$
As noted in \cite[Theorem 2.4]{HolDiskTwo}, this induces isomorphisms
$$ \J \co \HFo(Y, \s) \xrightarrow{\phantom{o} \cong \phantom {oi}} \HFo(Y, \bar \s)$$
for any $\spinc$ structure $\s$ on $Y$. We have $\J^2=\id$, so $\J$ is an involution on $\HFo(Y).$ This involution was used in various arguments in the Heegaard Floer literature; see for example \cite{LiscaStipsicz2, LiscaOwens}.

We define involutive Heegaard Floer homology by making use of the construction of $\J$ {\em at the chain level}. Specifically, we have a Heegaard Floer chain group $\CFo(Y)$ and a map
$$ \inv \co \CFo(Y) \to \CFo(Y)$$
that induces the map $\J=\inv_*$ on homology. We define the $\circ$ flavor of involutive Heegaard Floer homology, $\HFIo(Y)$, to be the homology of the mapping cone:
$$
\CFo(Y) \xrightarrow{\phantom{o} Q (1+\inv) \phantom{o}} Q \ccdot \CFo(Y) [-1]. 
$$
Here $Q$ is just a formal variable, with $Q^2=0$, and $[-1]$ denotes a shift in grading. If we work with coefficients in $\Z_2$, then  Heegaard Floer groups come equipped with $\Z_2[U]$-module structures, and we get a $\Z_2[Q, U]/(Q^2)$-module structure on $\HFIo(Y)$. 

\begin{theorem}
\label{thm:invariance}
For any flavor $\circ \in \{\widehat{\phantom{a}}, +, -, \infty\}$, the isomorphism class of the involutive Heegaard Floer homology $\HFIo(Y)$, as a $\Z_2[Q, U]/(Q^2)$-module, is a three-manifold invariant.
\end{theorem}

The ring $\Ring=\Z_2[Q, U]/(Q^2)$ is the cohomology ring of the classifying space $B\Z_4$, with $\Z_2$ coefficients. If we have a space with a $\pin$ action, then one can obtain its $\Z_4$-equivariant homology (as an $\Ring$-module) from its $S^1$-equivariant homology by constructing a mapping cone and then taking homology, just as we constructed $\HFIp$ from $\HFp$. Since $\HFp$ is supposed to correspond to $S^1$-equivariant Seiberg-Witten Floer homology, we see that $\HFIp$ should correspond to $\Z_4$-equivariant  Seiberg-Witten Floer homology. To construct the $\pin$-equivariant theory, one would have to complete the mapping cone to an infinite complex that involves not just the map $\inv$, but also the chain homotopy relating $\inv^2$ to the identity, and higher homotopies. To define these higher homotopies one would need to prove that Heegaard Floer homology is natural ``to infinite order,'' whereas currently this is established only to first order, by the work of Juh\'asz and Thurston \cite{Naturality}. We refer to Section~\ref{sec:motivation} below for more explanations.

The module $\HFIp(Y)$ decomposes as a direct sum indexed by the orbits of $\spinc$ structures under the conjugation action. The most interesting case is when we have a $\spinc$ structure $\s$ that comes from a spin structure, i.e., $\s = \bar \s$. We then obtain a group $\HFIp(Y, \s)$.

Furthermore, if we have a four-dimensional spin cobordism $(W, \t)$ from $(Y, \s)$ to $(Y', \s')$, we construct maps
$$ F^+_{W, \t, \aa} \co \HFIp(Y, \s) \to \HFIp(Y', \s').$$
A priori, these depend on some additional data $\aa$, which includes a choice of Heegaard diagrams for $Y$ and $Y'$ and a handle decomposition of the cobordism $W$. Although we expect the maps to not depend essentially on $\aa$, proving this would require results about higher order naturality that are not available by current techniques.

Recall that Heegaard Floer homology (for torsion $\spinc$ structures $\s$) can be equipped with an absolute grading with values in $\Q$; cf. \cite{AbsGraded}. When $Y$ is a rational homology sphere, the minimal grading of the infinite $U$-tower in $\HFp(Y, \s)$ gives the Ozsv\'ath-Szab\'o correction term $d(Y, \s) \in \Q$. When $\s$ is spin, the involutive Heegaard Floer homology $\HFIp(Y, \s)$ has two infinite $U$-towers, and by imitating \cite{AbsGraded} we obtain two new correction terms
$$\dl(Y, \s),\  \du(Y, \s) \in \Q$$
such that
$$ \dl(Y, \s)\leq d(Y, \s) \leq \du(Y, \s).$$

We also have a Fr{\o}yshov-type inequality for spin four-manifolds with boundary, analogous to \cite[Theorem 9.6]{AbsGraded}:
\begin{theorem} \label{thm:froyshov} 
Let $Y$ be a rational homology three-sphere, and $\s$ be a spin structure on $Y$. Then if $X$ is a smooth negative-definite four manifold $X$ with boundary $Y$, and $\t$ is a spin structure on $X$ such that $\t|_Y=\s$, then
$$\operatorname{rk}(H^2(X;\Z)) \leq 4\dl(Y, \s).$$
\end{theorem}

If $Y$ is a $\Z_2$-homology sphere, then it admits a unique spin structure $\s$, and we can simply write $\dl(Y)$ and $\du(Y)$ for the corresponding invariants. Recall that two three-manifolds $Y_0$ and $Y_1$ are called homology cobordant (resp. $\Z_2$-homology cobordant or $\Q$-homology cobordant) if there exists a smooth, compact, oriented cobordism $W$ from $Y_0$ to $Y_1$ such that $H_*(W, Y_i; \Z)=0$ (resp. $H_*(W, Y_i; \Z_2)=0$ or $H_*(W, Y_i; \Q)=0$) for $i=0,1$. The homology cobordism group $\Theta^3_\Z$ is generated by oriented integer homology spheres, modulo the equivalence relation given by homology cobordism. Similarly, the $\Z_2$-homology cobordism group $\Theta^3_{\Z_2}$ is generated by oriented $\Z_2$-homology spheres, modulo $\Z_2$-homology cobordism.

\begin{theorem}
\label{thm:ds}
The correction terms $\dl, \du$ are invariants of $\Z_2$-homology cobordism, i.e., they descend to (non-additive) maps
$$\dl, \du \co \Theta^3_{\Z_2} \to \Q$$
Further, when $Y$ is an integer homology sphere then $\dl$ and $\du$ take even integer values, and so give maps
$$ \dl, \du \co \Theta^3_{\Z} \to 2\Z.$$  
\end{theorem}

In some cases, for example when $Y$ is an L-space (i.e., $\HFhat(Y, \s) = \Z_2$ for every $\s \in \Spinc(Y)$), it turns out that $\dl(Y)=\du(Y)=d(Y)$. On the other hand, for the Brieskorn sphere $\Sigma(2,3,7)$ we have
 $$ \dl(\Sigma(2,3,7))=-2, \ \ \du(\Sigma(2,3,7))= d(\Sigma(2,3,7))=0.$$
Thus, whereas the usual correction term $d$ cannot tell that $\Sigma(2,3,7)$ is not homology null-cobordant, the invariant $\dl$ can. Of course, this can also be seen by other methods, e.g. using the Rokhlin invariant, which is $1$ for $\Sigma(2,3,7)$. More interesting is the following corollary:

\begin{corollary}
\label{cor:Lcor}
The L-spaces that are $\Z_2$-homology spheres generate a proper subgroup of $\Theta^3_{\Z_2}$. For example, $\Sigma(2,3,7)$ is not $\Z_2$-homology cobordant to any L-space. 
\end{corollary}

Observe that L-spaces that are $\Z_2$-homology spheres include, for example, all double branched covers over alternating knots in $S^3$; cf. \cite{BrDCov}. We also remark that Corollary~\ref{cor:Lcor} can be obtained using $\pin$-equivariant Seiberg-Witten Floer homology, by showing that $\alpha=\beta=\gamma$ for L-spaces; this is a consequence of the Gysin sequence that relates the $\pin$- and $S^1$-equivariant theories. (See \cite[Proposition 3.10]{Lin} for a version of this.)

In contrast to Corollary~\ref{cor:Lcor}, the Brieskorn sphere $\Sigma(2,3,7)$ does bound a rational homology ball, and hence is $\Q$-homology cobordant to $S^3$; cf. \cite{FSExample}. This forces the $d$ invariant to be zero.

In our theory, the calculation of $\dl$ and $\du$ for $\Sigma(2,3,7)$ is done using an adaptation of the large surgery formula from \cite{Knots, RasmussenThesis} to the involutive setting. More generally, this adaptation allows us to calculate $\HFIp$ of a large integral surgery on a knot $K \subset S^3$ in terms of the knot Floer complex $\CFKi(S^3,K)$ and the analogue of the map $\inv$ on $\CFKi(S^3,K)$, which we denote by $\inv_K$. 

Let us denote by $S^3_p(K)$ the result of surgery on a knot $K \subset S^3$, with framing $p \in \Z$. Recall that the usual large surgery formula (cf. \cite[Theorem 4.4]{Knots} or \cite[Theorem 2.3]{IntSurg}) identifies $\HFp(S^3_p(K), [s])$, for $p \gg 0$, with the homology of a quotient complex $A_s^+$ of $\CFKi(S^3, K)$. Here, $s$ is an integer, $[s]$ is its mod $p$ reduction, and we use the standard identification of $\spinc$ structures on $S^3_p(K)$ with the elements of $\Z/p\Z$. In particular, for $s=0$ we have a spin structure, and the map $\inv_K$ induces a similar map $\inv_0$ on $A_0^+$. Let $\AI_0^+$ be the mapping cone
$$
A_0^+ \xrightarrow{\phantom{o} Q (1+\inv_0) \phantom{o}} Q \ccdot A_0^+ [-1]. 
$$

We now state the involutive large surgery formula. 

\begin{theorem}
\label{thm:Large}
Let $K \subset S^3$ be a knot, and let $g(K)$ be its Seifert genus. Then, for all integers $p \geq g(K),$ we have an isomorphism of relatively graded $\Ring$-modules
$$ \HFIp(S^3_p(K), [0]) \cong H_*(\AI_0^+).$$ 
\end{theorem}

Note that, in Heegaard Floer theory, large surgeries are considered those with coefficient $p \geq 2g(K)-1$. If we are only interested in the spin structure $s=0$, then the weaker inequality $p \geq g(K)$ suffices. Therefore, in this paper, a ``large'' surgery will mean one with coefficient $p \geq g(K)$. 

In order to compute $H_*(\AI_0^+)$, we need to understand the conjugation symmetry $\inv_K$ on the knot Floer complex. This can be determined explicitly for two important families of knots: 
\begin{itemize}
\item {\em L-space knots} (and their mirrors), those knots that admit a surgery that is an L-space; cf. \cite{OSLens}. These include all torus knots, Berge knots, and $(-2,3,2k+1)$ pretzel knots; 
\item {\em Floer homologically thin knots} (which we simply call {\em thin}), those for which the knot Floer homology is supported in a single diagonal; cf. \cite{RasmussenThesis, RasSurvey, MOQuasi}. These include all alternating knots \cite{AltKnots} and, more generally, all quasi-alternating knots \cite{MOQuasi}.
\end{itemize}

The key observation is that $\inv_K^2$ is equal to the map studied by Sarkar in \cite{SarkarMoving}, which corresponds to moving the basepoints around the knot. In the two cases above, knowing $\inv_K^2$ and the behavior of $\inv_K$ with respect to gradings suffices to determine $\inv_K$ up to chain homotopy. From knowledge of $\inv_K$ one can calculate $\HFIp$ for large surgeries on those knots (many of which are hyperbolic).

In fact, it should be noted that $\inv_K$ is in principle computable for all knots in $S^3$, using grid diagrams and the maps on grid complexes \cite{MOS, MOST}. Thus, $\HFIp$ is algorithmically computable for all large surgeries on knots. Although in this paper we limit ourselves to large surgeries, we expect that $\HFIp$ satisfies involutive analogues of the surgery exact triangle from \cite{HolDiskTwo}, of the general knot surgery formulas from \cite{IntSurg, RatSurg}, and perhaps of the link surgery formula from \cite{LinkSurg}. Thus, it may be possible to show that $\HFIp$ is algorithmically computable for all three-manifolds, along the lines of \cite{MOT}.

Going back to correction terms, recall that, for large $p$, the Ozsv\'ath-Szab\'o correction term of $S^3_p(K)$ in the spin structure $[0]$ is given by
\begin{equation}
\label{eq:dspk}
 d(S^3_p(K), [0])= d(L(p,1),0) - 2V_0(K) = \frac{p-1}{4} - 2V_0(K),
 \end{equation}
where $V_0(K) \in \Z$ is an invariant of (smooth) knot concordance, coming from the knot Floer complex of $K$; cf. \cite{RasmussenThesis, RasmussenGT, RatSurg, Peters, NiWu}. Similarly, using $\inv_K$ we obtain new concordance invariants $ \Vl_0(K)$ and $\Vu_0(K)$, and we have the following result.

\begin{theorem}
\label{thm:dLarge}
Let $K \subset S^3$ be a knot of Seifert genus $g(K)$. Then, for each integer $p \geq g(K)$, we have
$$  \dl(S^3_p(K), [0]) = \frac{p-1}{4} -2\Vl_0(K), \ \ \ \ \du(S^3_p(K), [0])=\frac{p-1}{4} -2\Vu_0(K).$$
\end{theorem}

The calculation of $\Vl_0$ and $\Vu_0$ for L-space knots and mirrors of L-space knots can be found in Section~\ref{sec:Lspace}, and that for thin knots in Section~\ref{sec:thin}. Let us state the result for alternating knots:
\begin{theorem}
\label{thm:Alt}
Let $K$ be an alternating knot, with signature $\sigma$ and Arf invariant $\Arf \in \{0,1\}$. The values of  the triple $(\Vl_0, V_0, \Vu_0)$ for $K$ are given in the following tables. 

If $\sigma \leq 0$, then
\begin{center}
\begin{tabular}{| c || c | c|}
\hline 
$\sigma$ & $\Arf=0$ & $\Arf=1$ \\
\hline \hline
$-8k$ & $(2k, 2k, 2k)$ & $(2k+1, 2k, 2k)$ \\
\hline
$-8k-2$ & $(2k+1, 2k+1, 2k)$ & $(2k+1, 2k+1, 2k+1)$ \\
\hline
$-8k-4$ & $(2k+2, 2k+1, 2k+1)$ & $(2k+1, 2k+1, 2k+1)$ \\
\hline
$-8k-6$ & $(2k+2, 2k+2, 2k+2)$ & $(2k+2, 2k+2, 2k+1)$ \\
\hline
\end{tabular}
\end{center}

If $\sigma > 0$, then
\begin{center}
\begin{tabular}{| c || c | c|}
\hline 
$\sigma$ & $\Arf=0$ & $\Arf=1$ \\
\hline \hline
$8k$ & $(0, 0, -2k)$ & $(0, 0, -2k)$ \\
\hline
$8k+2$ & $(0, 0, -2k)$ & $(0, 0, -2k-1)$ \\
\hline
$8k+4$ & $(0, 0, -2k-1)$ & $(0, 0, -2k-1)$ \\
\hline
$8k+6$ & $(0, 0, -2k-2)$ & $(0, 0, -2k-1)$ \\
\hline
\end{tabular}
\end{center}
\end{theorem}

For example, the figure-eight knot $4_1$ (with the $+1$ surgery being $\Sigma(2,3,7)$) has $\sigma=0$ and $\Arf=1$. Therefore,
$$\Vl_0(4_1)=1, \ \ \  V_0(4_1)=\Vu_0(4_1)=0.$$
 
The figure-eight knot is not slice: classically, one can prove this by checking the Fox-Milnor condition on the Alexander polynomial \cite{FoxMilnor}, or (as above) by noting that $+1$ surgery on $4_1$ has non-trivial Rokhlin invariant. However, the non-sliceness of $4_1$ cannot be detected by most of the modern concordance invariants coming from Floer or Khovanov homology: $\tau$ from \cite{4BallGenus, RasmussenThesis}, $s$ from \cite{RasmussenMilnor}, $\delta$ from \cite{MOwens}, $d(S^3_1(K))=-2V_0(K)$ from \cite{AltKnots, RasmussenGT, Peters}, $\nu$ from \cite{RatSurg}, $\nu^+$ from \cite{HomWu}, $\varepsilon$ from \cite{HomEps}, and $\Upsilon_K(t)$ from \cite{UpsilonT} all vanish on amphichiral knots such as $4_1$. By contrast, our concordance invariant $\Vl_0$ does detect that $4_1$ is not slice.

Moreover, combining Theorems~\ref{thm:dLarge} and \ref{thm:Alt} we obtain various constraints on which large surgeries on an alternating knot can be homology cobordant to large surgeries on another alternating knot. For example, we have

\begin{corollary}
\label{cor:altHC}
Let $K$ and $K'$ be alternating knots such that $\sigma(K) \equiv 4\cdot \Arf(K) + 4 \pmod 8$. If $S^3_{p}(K)$ and $S^3_{p}(K')$ are $\Z_2$-homology cobordant for some odd $p \geq \max(g(K), g(K'))$, then $\sigma(K) =\sigma(K')$.
\end{corollary}

In the same spirit, if we combine Theorem~\ref{thm:froyshov} with Theorem~\ref{thm:Alt}, we obtain constraints on the intersection forms of spin four-manifolds with boundary a large surgery on an alternating knot.

The paper is organized as follows. In Section~\ref{sec:def} we define involutive Heegaard Floer homology and prove its invariance (Theorem~\ref{thm:invariance}). In Section~\ref{sec:motivation}
we explain in more detail why $\HFIp$ should correspond to $\Z_4$-equivariant Seiberg-Witten Floer homology. In Section~\ref{sec:prop} we establish a few properties of the involutive Heegaard Floer groups, and define the cobordism maps. In Section~\ref{sec:ds} we define the new correction terms $\dl, \du$ and prove Theorems~\ref{thm:froyshov} and \ref{thm:ds}. Section~\ref{sec:surgery} contains the proof of the involutive analog of the large surgery formula, Theorem~\ref{thm:Large}; we also prove there Theorem~\ref{thm:dLarge} and Corollary~\ref{cor:Lcor}, and show that $\dl$ and $\du$ are not homomorphisms. In Section~\ref{sec:Lspace} we apply the involutive large surgery formula to compute $\HFIp$ for large surgeries on (mirrors of) L-space knots. Large surgeries on thin knots are discussed in Section~\ref{sec:thin}, where we prove Theorem~\ref{thm:Alt} and Corollary~\ref{cor:altHC}.

\medskip
\noindent \textbf{Acknowledgements.} We thank Jennifer Hom, Andr{\'a}s Juh\'asz, Tye Lidman, Francesco Lin,  Robert Lipshitz and Sucharit Sarkar for helpful conversations. We are also grateful to the referees for many helpful suggestions.

\section{Definition}
\label{sec:def}
The goal of this section is to define involutive Heegaard Floer homology. We assume that the reader is familiar with regular Heegaard Floer homology, as in \cite{HolDisk, HolDiskTwo, HolDiskFour}. However, we start by reviewing a few concepts in order to fix notation, and to emphasize naturality issues.

\subsection{Heegaard Floer homology} Fix a closed, connected, oriented three-manifold $Y$. Denote by $\Spinc(Y)$ the space of $\spinc$ structures on $Y$, and pick some $\s \in \Spinc(Y)$. Heegaard Floer homology is computed from a \textit{pointed Heegaard diagram} for $Y$. A pointed Heegaard diagram is a set of data $H = (\Sigma, {\alphas}, {\betas}, z)$
where: 
\begin{itemize}
\item $\Sigma \subset Y$ is an embedded, oriented surface of genus $g$, that splits the three-manifold $Y$ into two handlebodies $U_0$ and $U_1$;  
\item ${\alphas} = \{\alpha_1,\dots,\alpha_g\}$ is a set of nonintersecting simple closed curves on $\Sigma$ which bound disks in $U_0$, and in fact span the kernel of $H_1(\Sigma; \Z) \to H_1(U_0; \Z)$; 
\item ${\betas}=\{\beta_1,\dots,\beta_g\}$ is a similar set of curves for $U_1$ instead of $U_0$, such that $\alpha_i \cap \beta_j$ is transverse for any $i,j$; 
\item $z\in \Sigma$ is a basepoint that does not lie on any of the alpha or beta curves.
\end{itemize}

 The Heegaard Floer groups are variations of Lagrangian Floer cohomology for the two tori 
 $$ \Ta = \alpha_1 \times \dots \times \alpha_g, \ \ \Tb = \beta_1 \times \dots \times \beta_g$$
inside the symmetric product $\Sym^g(\Sigma)$. There is a natural map $\s_z \co \Ta \cap \Tb \to \Spinc(Y)$, and we will focus on those intersection points $\x$ such that $\s_z(\x)=\s$. To define the Floer groups,  we need to impose an admissibility condition on $H$, depending on $\s$; cf. \cite[Section 4.2.2]{HolDisk}. We also need to choose a suitable (generic) family $J$ of almost complex structures on $\Sym^g(\Sigma)$. We will write $\H$ for the data $(H, J)$, which we call a {\em Heegaard pair}.

Given such a pair $\H$, the Heegaard Floer chain complex $\CFinfty(\H, \s)$ is freely generated over $\Z_2$ by pairs $[\x, i]$ with $\x \in \Ta \cap \Tb$ and $i\in \Z$, such that $\s_z(\x)=\s$. The differential is given by
$$ \del[\x, i] = \sum_{\{\y \in \Ta \cap \Tb \mid \s_z(\y)=\s\}} \sum_{\{ \phi \in \pi_2(\x, \y) \mid \mu(\phi)=1  \}} \# \Mh(\phi)\cdot [\y, i-n_z(\phi)].$$
Here, $\pi_2(\x, \y)$ is the space of homotopy classes of Whitney disks from $\x$ to $\y$, $\mu(\phi)$ is the Maslov index, $\Mh(\phi)$ is the moduli space of $J$-holomorphic disks in the class $\phi$ (modulo the action of $\R$), and $n_z(\phi)$ is the algebraic intersection number of $\phi$ with the divisor $\{z\} \times \Sym^{g-1}(\Sigma)$. There is an action of $\Z_2[U, U^{-1}]$ on $\CFinfty$, where $U$ acts by $U \cdot [\x, i]=[\x, i-1]$ and decreases relative grading by $2$. The other complexes $\CFp, \CFm$ and $\CFhat$ are obtained from $\CFinfty$ by considering only pairs $[\x, i]$ with $i \geq 0, i < 0$ and $i=0$. All three complexes have an induced $\Z_2[U]$-action, which is trivial in the case of $\CFhat$. 

We will write $\CFo(\H, \s)$ for any of the four flavors of the Heegaard Floer chain complex, and $\HFo(\H, \s)$ for the homology groups.

\begin{theorem}[Ozsv\'ath-Szab\'o \cite{HolDisk}]
\label{thm:OSinv}
The isomorphism class of $\HFo(\H, \s)$ (as a $\Z_2[U]$-module) is an invariant of the three-manifold $Y$. 
\end{theorem}

A stronger result was obtained by Juh\'asz and Thurston, who proved naturality of the invariant:
\begin{theorem}[Juh\'asz-Thurston \cite{Naturality}]
\label{thm:JT}
If we fix $Y$ and the basepoint $z \in Y$, then the $\Z_2[U]$-modules $\HFo(\H, \s)$ form a transitive system. That is to say, for any two Heegaard pairs $\H=(H, J)$ and $\H'=(H', J')$ we have a distinguished isomorphism
$$ \Psi(\H, \H') \co \HFo(\H, \s) \to \HFo(\H', \s),$$
such that for all $\H, \H', \H''$ we have:
\begin{enumerate}[(i)]
\item $\Psi(\H, \H)=\id_{\HFo(\H, \s)}$;
\item $\Psi(\H', \H'') \circ \Psi(\H, \H') = \Psi(\H, \H'').$ 
\end{enumerate}
\end{theorem}

Given a transitive system, we can get a single module $\HFo(Y, z, \s)$ as the inverse limit of this system. We can identify any $\HFo(\H, \s)$ with $\HFo(Y, z, \s)$ in a canonical way. We usually drop $z$ from the notation and write $\HFo(Y, \s)$. We can also consider the direct sum over all $\spinc$ structures:
$$ \HFo(Y) := \bigoplus_{\s \in \Spinc(Y)} \HFo(Y, \s).$$

Although Juh\'asz and Thurston phrased their theorem in terms of homology, their methods actually give a result at the chain level:
\begin{proposition}
\label{prop:JT}
If we fix $(Y, z, \s)$, then the chain groups $\CFo(\H, \s)$ form a transitive system in the homotopy category of chain complexes of $\Z_2[U]$-modules. In other words, for every two Heegaard pairs $\H$ and $\H'$  we have a chain homotopy equivalence 
$$ \Phi(\H, \H') \co \CFo(\H, \s) \to \CFo(\H', \s),$$
satisfying the analogs of conditions (i) and (ii) from the statement of Theorem~\ref{thm:JT}, with equality replaced by chain homotopy. In fact, the maps $\Psi(\H, \H')$ are those induced on homology by $\Phi(\H, \H')$.
\end{proposition}

\begin{remark} The map $\Phi(\H,\H')$ is only unique up to chain homotopy. In the remainder of the paper, whenever we write $\Phi(\H, \H')$, we mean a representative of this chain homotopy class of maps. \end{remark}

\begin{proof}
The proofs of Theorems~\ref{thm:OSinv} and \ref{thm:JT} involve showing that any two Heegaard pairs are related by a sequence of the following moves: changing the almost complex structures $J$ (with the diagram $H$ being fixed); isotopies of the alpha and beta curves in $H$; handleslides of the curves in $H$; stabilizations and destabilizations of $H$; diffeomorphisms of $H$ induced by an ambient isotopy of $\Sigma$ in $Y$.\footnote{In fact, a diffeomorphism induced by an ambient isotopy can be obtained as the composition of some stabilizations and destabilizations. However, it is convenient to consider it as a separate move.} These moves induce chain homotopy equivalences between the respective Floer chain complexes. Indeed, this was shown in \cite{HolDisk} for most of the moves. The exceptions are handleslides, for which the argument in \cite[Section 9.2]{HolDisk} only shows that they induce quasi-isomorphisms. It is proved there that if $F$ and $G$ are the maps associated to a handleslide and its inverse, then $G \circ F$ is homotopic to the composition of the triangle map coming from a small isotopy with the nearest point map. However, one can further show that, for a small isotopy, the triangle map is chain homotopic to the nearest point map; see \cite[Proposition 11.3]{LipshitzCyl} and \cite[proof of Theorem 6.6]{OzsvathStipsicz}. This implies that $G \circ F$ is chain homotopic to the identity. The same goes for $F \circ G$, so we can conclude that handleslides actually induce chain homotopy equivalences. 

Given two Heegaard pairs $\H$ and $\H'$, we define the maps $\Phi(\H, \H')$ by choosing a sequence of moves relating $\H$ and $\H'$. We claim that different choices of moves yield chain homotopic maps. First, note that interpolating between two (families of) almost complex structures in two different ways produces chain homotopic maps, by the usual continuation arguments in Floer theory. (We are using here that the space of compatible almost complex structures is contractible.)

With regard to the moves on Heegaard diagrams, to show that they give rise to chain homotopic maps, in view of Theorem 2.39 in \cite{Naturality}, it suffices to prove that $\CFo(\H, \s)$ is a strong Heegaard invariant in the sense of \cite[Definition 2.33]{Naturality}, in the homotopy category of chain complexes of $\Z_2[U]$-modules. There are four conditions to be checked: functoriality, commutativity, continuity, and handleswap invariance. All of these are checked  in Sections 9.2 and 9.3 of \cite{Naturality}, in the context of proving the weaker statement that $\HFo(\H, \s)$ is a strong Heegaard invariant in the category of $\Z_2[U]$-modules. However, the proofs there actually work at the chain level, with the maps being considered up to chain homotopy. 

Given that the maps $\Phi(\H, \H')$ are well-defined up to chain homotopy, conditions (i) and (ii) are almost automatic. Indeed, for (i), when $\H=\H'$, we can consider the empty sequence of moves, so that $\Phi(\H, \H)$ is the identity. For (ii), we consider a sequence of Heegaard moves from $\H$ to $\H'$, and another from $\H'$ to $\H''$, and by composing them we get a sequence from $\H$ to $\H''$.
\end{proof}

We will sometimes write $\CFo(Y, \s)$ for $\CFo(\H, \s)$. This is justified by Proposition~\ref{prop:JT}, which says that the chain groups $\CFo(\H, \s)$ for different $\H$ are chain homotopy equivalent (although, of course, they are not usually isomorphic).

\subsection{The involution} \label{sec:invo} Now let us discuss the conjugation action on Heegaard Floer homology. Given a pointed Heegaard diagram $H=(\Sigma, \alphas, \betas, z)$, we define the conjugate diagram $\bh$ by 
$$ \bh = (-\Sigma, {\betas}, {\alphas}, z),$$
where $-\Sigma$ means $\Sigma$ with the orientation reversed. A family $J$ of almost complex structures on $\Sym^g(\Sigma)$ gives a conjugate family $\bJ$ on $\Sym^g(-\Sigma)$. If $\H=(H, J)$ is a Heegaard pair, we write $\bH$ for the conjugate pair $(\bh, \bJ)$.

Intersection points in $\Ta \cap \Tb$ for $H$ are in one-to-one correspondence with those for $\bh$, and this correspondence takes a $\spinc$ structure $\s$ to its conjugate $\bar \s$. Moreover, $J$-holomorphic disks with boundaries on $(\Ta, \Tb)$ are in one-to-one correspondence with $\bJ$-holomorphic disks with boundaries on $(\Tb, \Ta)$. Thus, as observed in \cite[Theorem 2.4]{HolDiskTwo}, we get a canonical isomorphism between Heegaard Floer chain complexes:
\begin{align*}
\eta \co \CFo(\H, \s) \xrightarrow{\phantom{u} \cong \phantom{u}} \CFo(\bH, \bar\s).
\end{align*}

Moreover, $\H$ and $\bH$ represent the same based three-manifold $(Y, z)$. According to Proposition~\ref{prop:JT}, we have a chain homotopy equivalence
\begin{align*}
\Phi(\bH, \H) \co \CFo(\bH, \bar \s) \xrightarrow{\phantom{u} \sim \phantom{u}} \CFo(\H, \bar \s).
\end{align*}
We denote by $\inv$ the composition of these two maps:
$$ \inv = \Phi(\bH, \H) \circ \eta \co \CFo(\H, \s) \to \CFo(\H, \bar \s).$$

\begin{lemma}
\label{lem:inv}
The map $\inv^2 \co \CFo(\H, \s) \to \CFo(\H, \s)$ is chain homotopic to the identity.
\end{lemma}

\begin{proof}
We have $ \inv^2 =   \Phi(\bH, \H) \circ \eta \circ \Phi(\bH, \H) \circ \eta.$ Note that $\eta^2=1$, so the composition
\begin{equation}
\label{eq:etaPhi}
 \eta \circ \Phi(\bH, \H) \circ \eta \co \CFo(\H, \s) \to \CFo(\bH, \s)
 \end{equation}
is the conjugation of $\Phi(\bH, \H)$ by $\eta$.  Recall that $\Phi(\bH, \H)$ is the composition of maps associated to moves between the respective Heegaard pairs. When we conjugate any such map by $\eta$, we get the map associated to the corresponding move between the conjugate Heegaard pairs. (This uses the identification between $J$- and $\bJ$-holomorphic triangles.) In view of Proposition~\ref{prop:JT}, the map \eqref{eq:etaPhi} is chain homotopic to $\Phi(\H, \bH)$. Therefore,
$$ \inv^2 \sim \Phi(\bH, \H) \circ \Phi(\H, \bH) \sim \id_{\CFo(\H, \s)},$$
when in the last step we used the properties of a transitive system. \end{proof}

Lemma~\ref{lem:inv} implies that $\inv$ induces an involution 
$$\J =\inv_* \co \HFo(Y, \s)  \xrightarrow{\phantom{u} \cong \phantom{u}} \HFo(Y, \bar \s)$$ on Heegaard Floer homology. This was
already observed in \cite[Theorem 2.4]{HolDiskTwo}.
 
\begin{remark}
If we view Heegaard splittings as coming from self-indexing Morse functions on the three-manifold $Y$, then the equivalence $ \Phi(\bH, \H)$ is induced by moving from a Morse function $h$ to $-h$. 
 \end{remark}

\subsection{Involutive Heegaard Floer homology} \label{sec:CFI} Let $[\Spinc(Y)]$ denote the space of orbits of $\spinc$ structures on $Y$, under the conjugation action. An orbit $\varpi \in [\Spinc(Y)]$ is either of the form $\{\s\}$ with $\s = \bar \s$, or of the form $\{\s, \bar \s\}$ with $\s \neq \bar \s$. The former case corresponds to $\spinc$ structures that come from spin structures.

\begin{remark}
A $\spinc$ structure $\s$ with $\s = \bar \s$ admits $2^{b_1(Y)}$ lifts to a spin structure; see \cite[p. 124]{Lin}. By a slight abuse of terminology, when $\s = \bar \s$ we will refer to $\s$ as a spin structure without fixing a specific lift.
\end{remark}

Given $z \in Y$, a Heegaard pair $\H$ for $(Y, z)$, and an orbit $\varpi \in [\Spinc(Y)]$, set
$$\CFo(\H, \varpi) = \bigoplus_{\s \in \varpi} \CFo(\H, \s).$$

We define the {\em involutive Heegaard Floer complex} $\CFIo(\H, \varpi)$ to be the mapping cone complex
\begin{align} \label{eq:involutive0}
 \CFo(\H, \varpi) \xrightarrow{1+\inv} \CFo(\H, \varpi).
\end{align}

Given a complex $C_*$, we use $C[n]_*$ to denote the same complex with the grading shifted by $n$: $C[n]_k = C_{k+n}.$ Thus, as an abelian group, the cone complex above is
$$ \CFo(\H, \varpi)[-1] \oplus \CFo(\H, \varpi),$$
with the first factor being the domain of $1+\inv$ and the second the target.

To get more structure on this complex, it is helpful to introduce a formal variable $Q$ of degree $-1$ with $Q^2=0$, and write \eqref{eq:involutive0} as
\begin{align} \label{eq:involutive}
 \CFo(\H, \varpi) \xrightarrow{Q (1+\inv)} Q \ccdot \CFo(\H, \varpi)[-1].
\end{align}
Note that in the target the shift $[-1]$ cancels out the shift due to the variable $Q$, so in fact $Q \ccdot \CFo(\H, \varpi)[-1]$ is isomorphic to $\CFo(\H, \varpi)$ as a graded module.

We can re-write \eqref{eq:involutive} as
$$\bigl( \CFo(\H, \varpi)[-1] \otimes \Z_2[Q]/(Q^2), \partial + Q(1+\inv) \bigr ),$$ where $\partial$ is the ordinary Heegaard Floer differential. We write 
$$\del^\inv =  \partial + Q(1+\inv)$$
for the differential on $\CFIo(\H, \varpi)$. Observe that, by construction, $\CFIo(\H, \varpi)$ is a complex of modules over the ring 
$$\Ring = \Z_2[Q,U]/(Q^2),$$ with $Q$ and $U$ decreasing the grading by $1$ and $2$, respectively. 

\begin{proposition}
\label{prop:CFI}
The quasi-isomorphism class of the complex $\CFIo(\H, \varpi)$ (over $\Ring$) is an invariant of the pair $(Y, \varpi)$.
\end{proposition}

\begin{proof}
Note that in the defintion of $\inv$ we used the map $\Phi(\bH, \H)$, which was constructed from a sequence of Heegaard moves relating $\bH$ to $\H$. Thus, a priori, $\CFIo(\H, \varpi)$ depends not only on $\H$, but also on that sequence of moves. However, Proposition~\ref{prop:JT} guarantees that  $\Phi(\bH, \H)$ is well-defined up to chain homotopy. Therefore, so is $\inv$. Since the mapping cones of chain homotopy maps are homotopic, we conclude that changing the sequence of moves only changes $\CFIo(\H, \varpi)$ by a homotopy equivalence.

Next, fix the basepoint $z \in Y$ and suppose that we have a different Heegaard pair $\H'$ for $(Y, z)$. Let $\inv'$ be the corresponding map from $\CFo(\H', \varpi)$ to $\CFo(\H',  \varpi)$, which is the composition of $\Phi(\bH', \H')$ with an involution $\eta'$. Consider the diagram
\begin{equation}
\label{eq:xy}
\xymatrixcolsep{4pc}\xymatrix{
 \CFo(\H, \varpi) \ar[d]_{\Phi(\H, \H')} \ar[r]^{ Q(1+\inv) \phantom{ulu} } & Q \ccdot  \CFo(\H, \varpi)[-1] \ar[d]^{\Phi(\H, \H')}\\
 \CFo(\H', \varpi) \ar[r]^{ Q (1+\inv') \phantom{ulu}  } & Q \ccdot  \CFo(\H', \varpi)[-1].
}
\end{equation}
We claim that this diagram commutes up to chain homotopy. This is equivalent to showing that
\begin{equation}
\label{eq:zeroe}
 \Phi(\H, \H') \circ \Phi(\bH, \H) \circ \eta \sim \Phi(\bH', \H') \circ \eta' \circ \Phi(\H, \H').
 \end{equation}
We have $\eta' \circ \Phi(\H, \H') \sim \Phi(\bH, \bH') \circ \eta$. (This is similar to the discussion of conjugation by $\eta$ in the proof of Lemma~\ref{lem:inv}.) Further, by Proposition~\ref{prop:JT} we have $$\Phi(\bH', \H') \circ \Phi(\bH, \bH') \sim \Phi(\bH, \H') \sim  \Phi(\H, \H') \circ \Phi(\bH, \H),$$ 
so \eqref{eq:zeroe} follows.

Let $$\Upsilon(\H, \H') \co \CFo(\H, \varpi) \to Q \ccdot  \CFo(\H', \varpi)[-1]$$
be the chain homotopy that makes the diagram \eqref{eq:xy} commute. Let us add $\Upsilon(\H, \H')$ to that diagram as a diagonal map from the upper left to the lower right corner. Together with the two vertical maps, this gives a chain map $\Phi^I(\H, \H')$ between the two rows, i.e., between $\CFIo(\H, \varpi)$ and $\CFIo(\H', \varpi)$. To see that $\Phi^I(\H, \H')$ is a quasi-isomorphism, note that any mapping cone comes equipped with a natural two-step filtration; in the diagram \eqref{eq:xy}, we set the filtration level to be $1$ for the left column and $0$ for the right column. The chain map we constructed respects the filtration, and it induces a quasi-isomorphism on the associated graded. (This is because on the associated graded we only see the vertical maps $\Phi(\H, \H')$, which are chain homotopy equivalences.) A filtered chain map that induces a quasi-isomorphism on the associated graded must be a quasi-isomorphism itself.

We have shown that the quasi-isomorphism type of $\CFIo(\H, \varpi)$ is an invariant of the triple $(Y, z, \varpi)$. It remains to prove independence of the basepoint $z$. If we have another basepoint $z' \in Y$, pick a diffeomorphism $\phi\co Y \to Y$ such that $\phi(z)=z'$. Given a Heegaard pair $\H$ for $(Y, z)$ and a sequence of moves from $\bH$ to $\H$, we can apply the diffeomorphism $\phi$ to obtain a Heegaard pair $\phi(\H)$ for $(Y, z')$, as well as a sequence of moves from $\overline{\phi(\H)}=\phi(\bH)$ to $\phi(\H)$. The resulting modules $\CFIo(\H, \varpi)$ and $\CFIo(\phi(\H), \varpi)$ are clearly isomorphic.
\end{proof}

We define the involutive Heegaard Floer homology $\HFIo(Y, \varpi)$ to be the homology of the complex $\CFIo(\H,\varpi)$. Summing over all $\spinc$ orbits, we set
$$ \HFIo(Y) := \bigoplus_{\varpi \in [\Spinc(Y)]} \HFIo(Y, \varpi).$$

Proposition~\ref{prop:CFI} shows that the isomorphism class of $\HFIo(Y, \varpi)$, as an $\Ring$-module, is an invariant of $(Y, \varpi)$. This implies Theorem~\ref{thm:invariance}.

\subsection{Remarks on naturality}
\label{sec:remarks}
With $(Y, z)$ fixed, let us construct an oriented one-dimensional simplicial complex (an oriented multigraph) $K^1$ as follows. We take the vertices of $K^1$ to be all the admissible Heegaard pairs $\H=(H, J)$. Then, for each standard move (change in $J$, isotopy, handleslide, stabilization), we draw an edge from the initial to the final Heegaard pair. Thus, to each vertex of $K^1$ we associate a complex $\CFo(\H, \s)$, and to each edge $e$ a chain homotopy equivalence $\Phi(e)$ between the respective complexes. Theorem~\ref{thm:OSinv} can then be viewed as a consequence of  the connectedness of $K^1$. 

Furthermore, Proposition~\ref{prop:JT} implies that whenever we have three vertices $v_0, v_1, v_2$ in $K^2$, along with edges $e_{01}, e_{12}, e_{02}$ (with $e_{ij}$ going from $v_i$ to $v_j$), the equivalence $\Phi(e_{02})$ is chain homotopic to $\Phi(e_{12}) \circ \Phi(e_{01})$ or, since we work with mod $2$ coefficients,
$$ \Phi(e_{02}) + \Phi(e_{12}) \circ \Phi(e_{01}) \sim 0.$$
 The chain homotopy depends on some choices (of handleswaps, of $2$-parameter families of isotopic curves or almost complex structures, etc.) For each such choice, let us attach a $2$-simplex  to the loop in $K^1$ formed by the union of the three edges. This produces a simplicial complex $K^2$, where to each $2$-simplex $f$ we have associated a chain homotopy $\Upsilon(f)$. Theorem~\ref{thm:JT} can then be viewed as a consequence of the fact that $K^2$ is simply connected. (This should be compared with the discussion in Appendix A in \cite{Naturality}, which contains an elementary proof of simple connectivity for a $2$-complex of handleslides.)

While Theorem~\ref{thm:JT} and Proposition~\ref{prop:JT} give a naturality result ``of order one," one could ask for more. Whenever we have the $2$-skeleton of a tetrahedron in $K^2$, with vertices $v_0, v_1, v_2, v_3$, edges $e_{ij}$ oriented from $v_i$ to $v_j$ for $i < j$, and faces $f_{ijk}$ for $i < j < k$, there is an associated sum of compositions 
$$ \Upsilon(f_{013}) +\Upsilon(f_{023}) + \Upsilon(f_{123}) \circ \Phi(e_{01}) + \Phi(e_{23}) \circ \Upsilon(f_{012}).$$
We conjecture that this sum is chain homotopic to zero. One should be able to attach a $3$-simplex  to the existing skeleton for each natural choice of such a ``higher chain homotopy,'' and obtain a $2$-connected simplicial complex $K^3$. This would be a naturality result of order two. Moreover, one should be able to continue this process and hope for $(n-1)$-connected complexes $K^n$ for all $n$, such that their union (the geometric realization of a simplicial set) is weakly contractible. This would be a naturality result of infinite order.

Our expectation is based on the situation in Seiberg-Witten theory \cite{KMBook, Spectrum}, where monopole Floer homology is defined starting from a contractible set of choices (metrics, perturbations, base connections). In contrast, the analogous result in Heegaard Floer theory seems much harder to obtain. A Heegaard diagram corresponds to a (self-indexing) Morse function on $Y$, which gives a gradient vector field on $Y$. To prove naturality of order one, Juh\'asz and Thurston had to study singularities of $2$-parameter families of gradients. To prove naturality of infinite order, one would need to understand singularities of $n$-parameter families of gradients, for all $n$.

Although they are beyond the scope of the current work, let us now mention two ways in which Proposition~\ref{prop:CFI} could be strengthened. Both of them would involve proving a kind of naturality result of order two.

The first improvement would be to replace invariance up to quasi-isomorphism (in Proposition~\ref{prop:CFI}) with invariance up to chain homotopy equivalence. This would require constructing a homotopy inverse to the quasi-isomorphism $\Phi^I(\H, \H')$. A good candidate is $\Phi^I(\H', \H)$. To prove that $\Phi^I(\H, \H') \circ \Phi^I(\H', \H) \sim \id$, one needs a commutation result between maps of the form $\Phi(\H, \H')$ and chain homotopies of the form $\Upsilon(\H, \H')$. In principle, this is an instance of order two naturality. However, in the case at hand, a proof may be more tractable by choosing $\Phi(\H', \H)$ to follow the same Heegaard moves as $\Phi(\H, \H')$, but in reverse, and by choosing $\Upsilon(\H', \H)$ to be a suitable reverse of $\Upsilon(\H, \H')$.

The second possible strengthening is to show naturality of order one for $\CFIo(\H, \varpi)$. In the spirit of Proposition~\ref{prop:JT}, this would mean that the complexes $\CFIo(\H, \varpi)$ for different $\H$ form a transitive system in the homotopy category. This would give a construction of $\HFIo(Y, \varpi)$ as a well-defined $\Ring$-module. A proof would require the following: As we move from $\H$ to itself by a sequence of Heegaard moves, we get a map $\Phi'(\H,\H)$, which is chain homotopic to the identity by a homotopy $\Upsilon(\H, \H)$ given by Proposition~\ref{prop:JT}. There is a similar conjugate homotopy $\Upsilon(\bH, \bH)$, and we need to prove a commutation result between these $\Upsilon$ maps and the equivalences $\Phi(\H, \bH)$. This is again an instance of order two naturality. In fact, since we used naturality of order one for $\CFo$ to prove Proposition~\ref{prop:CFI} (which is invariance of the isomorphism class for $\CFIo$, i.e., naturality of order zero), it is not surprising that proving naturality of order one for $\CFIo$ requires naturality of order two for $\CFo$.

\section{Comparison to Seiberg-Witten theory}
\label{sec:motivation}

By work of Kutluhan-Lee-Taubes \cite{KLT1}, or Collin-Ghiggini-Honda and Taubes \cite{CGH2, Taubes}, Heegaard Floer homology is now known to be isomorphic to monopole (Seiberg-Witten) Floer homology as defined by Kronheimer-Mrowka \cite{KMBook}. An alternative construction of Seiberg-Witten Floer homology for rational homology spheres was given by the second author in \cite{Spectrum}, and work by Lidman and the second author establishes the equivalence of this to the Kronheimer-Mrowka theory \cite{LidmanManolescu}. Under these equivalences, the variant $\HFhat(Y)$ of Heegaard Floer homology corresponds to the ordinary (non-equivariant) homology of the suspension spectrum $\swf(Y)$ defined in \cite{Spectrum}, and the variant $\HFp(Y)$ corresponds to the $S^1$-equivariant homology of $\swf(Y)$.  

The spectrum $\swf(Y)$ comes equipped with a $\pin$ action extending the $S^1$ action, and its $\pin$-equivariant homology was studied in \cite{Triangulations}. Since $\pin$ is an extension of $\Z_2$ by $S^1$, we expect that the Heegaard Floer analogue of $\pin$-equivariant Seiberg-Witten Floer homology should be a $\Z_2$-equivariant version of $\HFp(Y)$. Further, the extra $\Z_2$-action in Seiberg-Witten theory is the conjugation symmetry, which corresponds to $\inv$ on $\HFp(Y)$.

In order to understand what a potential $\Z_2$-equivariant version of $\HFp(Y)$ (or ``$\pin$ Heegaard Floer homology'') would look like, recall that a model for constructing $\Z_2$-equivariant Floer homology (in a different setting) was given by Seidel and Smith \cite{SeidelSmith}. Their starting point was $\Z_2$-equivariant Morse theory, in which one does Morse theory on the homotopy quotient 
$$M_{\borel} := M \times_{\Z_2} E\Z_2.$$ This space is a fiber bundle over $B\Z_2=\rp^{\infty}$ with fiber $M$. By equipping the base $\rp^{\infty}$ with a standard metric and Morse function (with one critical point in each nonnegative degree), and equipping the fibers with a suitable family of metrics and functions, one can show that $H_*(M_{\borel}; \Z_2) = H_*^{\Z_2}(M; \Z_2)$ is the homology of a complex 
$$ (C_{*, \borel}(M), \del_{\borel}),$$ 
as follows. As a vector space, the complex $C_{*, \borel}(M)$ is freely generated by pairs $[x, j]$, where $x$ is a generator of the usual Morse complex $C_*(M)$ and $j$ is a nonnegative integer. The degree of $[x, j]$ is $(\operatorname{deg } x) + j$. Further, there is a $H^*(B\Z_2; \Z_2) = \Z_2[Q]$ action on $C_{*, \borel}(M)$, given by $Q \cdot [x, j] = [x,j-1].$ The differential $\del_{\borel}$ decomposes as 
\begin{equation}
\label{eq:borel}
 \del_{\borel} = \del + Q \cdot (1 + \inv) + Q^2 \cdot H + \dots
 \end{equation}
where $\del$ is the ordinary Morse differential for $M$, and the higher terms are Morse continuation maps parametrized by gradient flow lines on the base $\rp^{\infty}$. In particular, $\inv \co C_*(M) \to C_*(M)$ is a chain map that induces the involution on $H_*(M)$, and the next term $H$ is a chain homotopy between $\inv^2$ and the identity. 

In view of this, a potential $\Z_2$-equivariant version of $\HFp(Y)$ (or ``$\pin$ Heegaard Floer homology'') should be the homology of a complex whose generators are $[\x, i, j]$ for $\x \in \Ta \cap \Tb$ and $i, j \in \Z_{\geq 0}$, with a differential similar to \eqref{eq:borel}. This would come with $U$ and $Q$ actions given by $U \cdot [\x, i, j] = [\x, i-1, j]$ and $Q \cdot [\x, i, j] = [\x, i, j-1]$. However, the $U$ action would generally not commute with the differential. Rather, given that this should mimic $\pin$-equivariant Seiberg-Witten Floer homology and therefore be a module over $H^*(\Bpin; \Z_2)=\Z_2[Q, V]/(Q^3)$ as in \cite{Triangulations}, we only expect the $U^2=V$ action to commute with the differential, and we further expect the $Q$ action to satisfy $Q^3=0$ on homology.

In any case, a definition of $\pin$ Heegaard Floer homology along these lines is difficult for the following reasons. The $Q^2$ term in \eqref{eq:borel} involves the chain homotopy $H$ between $\inv^2$ and the identity. Such a homotopy exists because of the usual naturality (of order one) in Heegaard Floer homology, but proving that a complex involving $H$ is well-defined up to chain homotopy would involve showing that any two choices of $H$ are related by a higher homotopy, and this would mean proving naturality of order two. (Compare Section~\ref{sec:remarks}.) Moreover, naturality of order two would be involved in even defining the $Q^3$ term in the differential. To define the whole differential would require a proof of naturality up to arbitrary orders in Heegaard Floer homology, and this is not accessible by current methods.

These limitations have led us to settle for a truncation of the complex \eqref{eq:borel}, in which we set $Q^2=0$. The result is the involutive Heegaard Floer homology defined in Section~\ref{sec:def}. This is a module over the ring $\Ring=\Z_2[Q,U]/(Q^2)$, which can be identified with $H^*(B\Z_4; \Z_2)$. In fact, in the Introduction we claimed that $\HFIp(Y)$ should correspond to $\Z_4$-equivariant Seiberg-Witten Floer homology, where $\Z_4 \subset \pin$ is the subgroup generated by $j$. This is justified by the following result.

\begin{proposition}
\label{prop:SWF}
Let $X$ be a paracompact Hausdorff space with a $\pin$ action, and fix a model for the classifying bundle $\Epin$.  Let $C_*^{S^1}(X; \Z_2)$ be the singular chain complex of the $S^1$ homotopy quotient $X \times_{S^1} \Epin$, with an involution $\inv$ coming from the action of $j \in \pin$ on $X \times \Epin$.  Then, we have an isomorphism of $\Ring$-modules
\begin{equation}
\label{eq:z4}
 H_*^{\Z_4}(X; \Z_2) \cong H_*(\Cone(C_*^{S^1}(X; \Z_2) \xrightarrow{Q (1+\inv) } Q \ccdot C_*^{S^1}(X; \Z_2)[-1])).
 \end{equation}
\end{proposition}

\begin{proof}
Let 
$$ Y:= (X \times_{S^1} \Epin) \times_{\Z_2} S^1$$ 
be the mapping torus of $X \times_{S^1} \Epin$ associated to the involution $j$. The right hand side of \eqref{eq:z4} is the homology of $Y$. Since $ (X \times_{S^1} \Epin) / {\Z_2} \cong X \times_{\pin} \Epin$, we can view $Y$ as a principal bundle
\begin{equation}
\label{eq:bundle1}
S^1 \hookrightarrow Y \twoheadrightarrow X \times_{\pin} \Epin.
\end{equation}
Principal circle bundles over a base $B$ are classified by homotopy classes of maps $B \to BS^1$, that is, by elements in $H^2(B; \Z)$ (their Euler classes). In our case the classifying map is the composition 
$$X \times_{\pin} \Epin \to \Bpin \to BS^1,$$ 
where the latter map comes from the quotient $\pin/ \Z_4 \cong S^1$.

On the other hand, since $\pin/ \Z_4 \cong S^1$, we also have a fiber bundle (which is not a principal bundle)
\begin{equation}
\label{eq:bundle2}
S^1 \hookrightarrow  X \times_{\Z_4} \Epin \twoheadrightarrow  X \times_{\pin} \Epin.
\end{equation}

General circle bundles over a base $B$ are classified by homotopy classes of maps $B \to BO(2)$, and they can be made into principal bundles if and only if they are orientable, that is, when the classifying map factors through $BSO(2) \cong BS^1$. Each circle bundle $S^1 \hookrightarrow T \twoheadrightarrow B$ has an associated second Steifel-Whitney class $w_2(T) \in H^2(B; \Z_2)$. Cap product with $w_2(T)$ gives one of the maps in the Gysin sequence on homology
$$  \dots \to H_*(T; \Z_2) \to H_*(B; \Z_2) \xrightarrow{\frown w_2(T)} H_{*-2}(B; \Z_2) \to \cdots$$

Recall that the Gysin sequence is constructed starting from  the short exact sequence of a pair,
$$ 0 \to C_*(T; \Z_2) \to  C_*(M; \Z_2) \to C_*(M, T; \Z_2) \to 0,$$
where $M$ is the disk bundle associated to $T$, i.e., the mapping cylinder for the projection $p \co T \to B$. Moreover, we have a Thom quasi-isomorphism $C_*(M, T) \to C_{*-2}(B), \ x \to p_*(x) \frown  \zeta,$ where $\zeta \in C^*(M, T)$ is a representative of the Thom class; since we work over a field, the quasi-isomorphism is in fact a chain homotopy equivalence. We also have another equivalence $p_*: C_*(M)  \to C_*(B)$. Putting everything together, we get equivalences:
$$  C_*(T; \Z_2)\xrightarrow{\sim} \Cone(C_*(M, T; \Z_2) \to C_*(M; \Z_2) ) \xrightarrow{\sim} \Cone (C_*(B; \Z_2) \xrightarrow{\frown w} C_{*-2}(B; \Z_2)),$$
where $w$ is any cocycle representing $w_2(T)$. These equivalences are canonical up to chain homotopy. Therefore, if we have two circle bundles $T, T'$ over $B$ with the same second Stiefel-Whitney class, then the homologies $H_*(T; \Z_2)$ and $H_*(T'; \Z_2)$ are canonically isomorphic.

In our setting, the circle bundles \eqref{eq:bundle1} and \eqref{eq:bundle2} have the same second Stiefel-Whitney class, namely the pull-back to $H^*(X \times_{\pin} \Epin; \Z_2)$ of the generator $Q^2 \in H^2(\Bpin; \Z_2) \cong \Z_2$. This implies that 
\begin{equation}
\label{eq:isoxy}
H_*^{\Z_4}(X; \Z_2) \cong H_*(Y; \Z_2).
\end{equation}

It remains to show that this isomorphism preserves the $\Ring$-module structures. To see this, consider the Gysin sequence on cohomology, at the chain level, which gives an isomorphism similar to \eqref{eq:isoxy}:
$$
H^*_{\Z_4}(X; \Z_2) \cong H^*(Y; \Z_2).
$$
In the particular case $X = pt$, the isomorphism gives
$$ H^*(B\Z_4; \Z_2) \cong H^*((\Epin/S^1) \times_{\Z_2} S^1; \Z_2) \cong \Ring.$$
For general $X$, the projection $X \to pt$ yields a commutative diagram
\begin{equation}
\label{eq:cmd}
\xymatrixcolsep{4pc}\xymatrix{
H^*_{\Z_4}(X; \Z_2) \ar[d] \ar[r]^{\cong} & H^*(Y; \Z_2) \ar[d] \\
H^*_{\Z_4}(pt;\Z_2)  \ar[r]^-{\cong} & H^*((\Epin/S^1) \times_{\Z_2} S^1; \Z_2).}
\end{equation}

The $\Ring$-module actions on $H_*^{\Z_4}(X; \Z_2)$ and $H_*(Y; \Z_2)$ are given by cap products with cohomology elements pulled back under the vertical maps in \eqref{eq:cmd}. The fact that the diagram commutes implies that the isomorphism \eqref{eq:isoxy} commutes with the $\Ring$-actions.\end{proof}

In view of Proposition~\ref{prop:SWF}, we make the following

\begin{conjecture}
For any rational homology sphere $Y$ and spin structure $\s$ on $Y$, we have an $\Ring$-module isomorphism 
$$ \HFIp_*(Y, \s) \cong H_*^{\Z_4}(\swf(Y, \s); \Z_2),$$
where $\swf(Y, \s)$ is the Seiberg-Witten Floer spectrum from \cite{Spectrum}.
\end{conjecture}

\section{Properties}
\label{sec:prop}

In this section we establish a few properties of the involutive Heegaard Floer homology groups $\HFIo(Y, \varpi)$ defined in Section~\ref{sec:def}. In this section, we will frequently have cause to refer to maps between involutive Heegaard Floer homology groups $\HFI$. Since we do not know that involutive Heegaard Floer homology is a natural object, this always means that the map is defined for a particular set of Heegaard data $\H$ for $Y$ (and may be a different map for a different set of Heegaard data).

\subsection{Basic facts}

First, observe that since $\CFIo = (\CFo[-1] \otimes (\Z_2[Q]/(Q^2))$ as a vector space, the (relative and absolute) gradings on $\CFo$ induce gradings on $\CFIo$. Let us write $c_1(\varpi)$ for the first Chern class of any representative of $\varpi$, and let $\mathfrak{d}(\varpi) = \operatorname{gcd}\{ \langle c_1(\varpi), \xi \rangle \mid \xi \in H_2(Y; \Z) \}.$ Then, $\CFIo(Y, \varpi)$ has an absolute $\Z_2$ grading and a relative $\Z/{\mathfrak{d}(\varpi)}\Z$ grading. Further, when $c_1(\varpi)$ is torsion, there is an absolute $\Q$-grading lifting the $\Z/{\mathfrak{d}(\varpi)}\Z$ grading. (Compare \cite{HolDisk}, \cite{AbsGraded}.)

We have the following exact sequences, in analogy with the corresponding long exact sequences for ordinary Heegaard Floer homology.

\begin{proposition} \label{propn:longexact} The involutive Heegaard Floer groups have long exact sequences
\begin{align*}
\cdots \rightarrow \HFIhat(Y,\varpi) \rightarrow \HFIp(Y,\varpi) \xrightarrow{\cdot U} \HFIp(Y, \varpi) \rightarrow \cdots \\
\cdots \rightarrow \HFIm(Y,\varpi) \xrightarrow{i} \HFIinf(Y, \varpi) \xrightarrow{\pi} \HFIp(Y, \varpi) \rightarrow \cdots 
\end{align*}
\noindent where $i$ and $\pi$ denote the maps induced by inclusion and projection.
\end{proposition}

\begin{proof} We prove the second assertion; the first is similar. From the definitions of the ordinary Heegaard Floer chain complexes, there is a short exact sequence $0 \rightarrow \CFm(Y,\s) \rightarrow \CFinf(Y,\s) \rightarrow \CFp(Y, \s)\rightarrow 0$. We have the following commutative diagram.
\begin{center}
\begin{tikzpicture}[node distance = 4 cm, auto]
\node(11){$0$};
\node(12)[right = .5 cm of 11]{$Q\ccdot \CFm(Y,\s)[-1]$};
\node(13)[right of =12]{$Q \ccdot\CFinf(Y, \s)[-1]$};
\node(14)[right of =13]{$Q\ccdot \CFp(Y, \s)[-1]$};
\node(15)[right = .5 cm of 14]{$0$};
\path[->](11)edge(12);
\path[->](12)edge node [above]{$i$}(13);
\path[->](13)edge node [above]{$\pi$}(14);
\path[->](14)edge (15);
\node(1)[above = 1.15 cm of 11]{$0$};
\node(2)[above = 1 cm of 12]{$\CFm(Y,\s)$};
\node(3)[above = 1 cm of 13]{$\CFinf(Y, \s)$};
\node(4)[above = 1 cm of 14]{$\CFp(Y, \s)$};
\node(5)[above = 1.2 cm of 15]{$0$};
\path[->](1)edge(2);
\path[->](2)edge node [above]{$i$}(3);
\path[->](3)edge node [above]{$\pi$}(4);
\path[->](4)edge(5);
\path[->](2)edge node[right]{$\cdot Q(1+\inv)$}(12);
\path[->](3)edge node[right]{$\cdot Q(1+\inv)$}(13);
\path[->](4)edge node[right]{$\cdot Q(1+\inv)$}(14);
\end{tikzpicture}
\end{center}

Commutativity of this diagram induces a short exact sequence of chain maps between the mapping cones:
$$0 \rightarrow \CFIm(Y,\varpi) \rightarrow \CFIinf(Y, \varpi) \rightarrow \CFIp(Y, \varpi)\rightarrow 0$$
which in turn gives rise to the desired long exact sequence.
\end{proof}

In addition to the involutive Heegaard Floer complexes $\CFIo(Y, \varpi)$, we can also consider their duals, the cochain complexes $\CFI_{\circ}(Y, \varpi)$. In ordinary Heegaard Floer homology, the dual complex to $\CFo(Y, \varpi)$ is denoted $\CF_{\circ}(Y, \varpi)$, and then $\CFI_{\circ}(Y, \varpi)$ can be viewed as the mapping cone of $Q(1+\inv^\sharp)$ on $\CF_{\circ}(Y, \varpi)$, where $\inv^\sharp$ is the dual map to $\iota$ on cochain complexes. 

By analogy with the reduced Heegaard Floer groups $\HFred^+(Y, \s)$ from \cite[Definition 4.7]{HolDisk}, we define the {\em reduced involutive Heegaard Floer homology} as
\begin{equation}
\label{eq:HFIred}
 \HFIred^+(Y, \varpi) := \HFIp(Y, \varpi)/ \operatorname{Im}(U^n) \text{ for } n \gg 0.
 \end{equation}
Equivalently, $\HFIred^+(Y, \varpi)$ is the cokernel of the map $\pi \co \HFIinf(Y, \varpi) \to \HFIp(Y, \varpi)$ from the second exact sequence in Proposition~\ref{propn:longexact}.

\subsection{Orientation reversal} 
We would like to consider the behavior of the involutive invariant under orientation reversal. For context, let us quickly recall the situation for ordinary Heegaard Floer homology. Recall that there is a natural bijection $\Spinc(Y) \simeq \Spinc(-Y)$ by regarding a nowhere-vanishing vector field over $Y$ as a nowhere-vanishing vector field over $-Y$. Ozsv{\'a}th and Szab{\'o} \cite[Section 5]{HolDiskFour} define a pairing map
$$\CFinf(Y,\s) \times \CFinf(-Y,\s) \rightarrow \Z $$
via 
\begin{align*}
\langle[{\bf x},i], [{\bf y},j]\rangle = \begin{cases} 1 \text{ if } {\bf x}={\bf y} \text{ and } i+j+1=0 \\
0 \text{ otherwise} \end{cases}
\end{align*}
which satisfies $\langle \del_{Y} \alpha, \beta \rangle = \langle \alpha, \del_{-Y} \beta \rangle$ and $\langle U\alpha, \beta \rangle = \langle \alpha, U \beta \rangle$. For torsion $\spinc$ structures, this induces isomorphisms of chain complexes between $\CFinf(Y,\s)$ and the cochain complex $CF_{\infty}(-Y,\s)$ and between $\CFp(Y,\s)$ and the cochain complex $\CF_{-}(-Y,\s)$, in both cases via $[{\bf x},i] \mapsto [{\bf x}, -i-1]^*$. 

In order to define an analogous pairing for involutive Heegaard Floer homology on the chain level, we must take some care with the choices made in producing appropriate chain maps $\inv$ for $Y$ and $-Y$. As in Section \ref{sec:def}, let $\H = (\Sigma, {\alphas}, {\betas}, z)$ be a Heegaard diagram for $Y$, and let $\bH = (-\Sigma, {\betas}, {\alphas}, z)$. Assume we have chosen a sequence of moves between $\bH$ and $\H$ inducing a chain homotopy equivalence $\Phi(\bH,\H) \co \CFinf(\bH,\bar{\s}) \rightarrow \CFinf(\H,\bar{\s})$, so that the chain map $\inv_Y$ is the composition
$$\inv_Y \co \CFinf(\H,\s) \xrightarrow{\eta} \CFinf(\bH,\bar{\s}) \xrightarrow{\Phi(\bH,\H)} \CFinf(\H, \bar{\s})$$
where $\eta$, as before, is the canonical isomorphism between $\CFinf(\H,\s)$ and $\CFinf(\bH,\bar{\s})$. Now, let $-\H = (\Sigma, {\betas}, {\alphas}, z)$ be our choice of Heegaard diagram for $-Y$, and let the sequence of moves connecting $-\bH = (-\Sigma, {\alphas}, {\betas}, z)$ to $-\H$ be the same moves as in the previous sequence, but in the opposite order. Then we have a map
$$\inv_{-Y} \co \CFinf(-\H,\s) \xrightarrow{\eta} \CFinf(-\bH,\bar{\s}) \xrightarrow{\Phi(-\bH,-\H)} \CFinf(-\H, \bar{\s}).$$

With this in mind, we define a pairing $$\langle,\rangle^{\inv} \co \CFIinf(\H,\varpi) \times \CFIinf(-\H,\varpi) \rightarrow \mathbb Z$$ by requiring that $\langle [{\bf x},i], [{\bf y},j] \rangle^{\inv} = \langle Q[{\bf x},i], Q[{\bf y},j] \rangle^{\inv} =0$ and
\begin{align*}
\langle [{\bf x},i], Q[{\bf y},j] \rangle^{\inv} &= \langle Q[{\bf x},i], [{\bf y}, j]\rangle^{\inv} = \begin{cases} 1 \text{ if } {\bf x}={\bf y} \text{ and } i+j+1=0 \\
0 \text{ otherwise.} \end{cases}
\end{align*}
The point of our care in constructing $\inv_Y$ and $\inv_{-Y}$ is the following lemma.

\begin{lemma} \label{lemma:iotapair} The pairing $\langle,\rangle^{\inv}$ satisfies
$$\langle Q(\inv_Y[{\bf x}, i]), [{\bf y},j] \rangle^{\inv} = \langle [{\bf x}, i], Q(\inv_{-Y}[{\bf y},j]) \rangle^{\inv}.$$
\end{lemma}

\begin{proof} Each map of chain complexes induced by one of the moves in the sequence we have chosen to connect $\bH$ and $\H$ counts either rigid pseudo-holomorphic disks of appropriate index or pseudo-holomorphic triangles; interchanging ${\alphas}$ and ${\betas}$ and running the sequence in the opposite order has the effect that the direction of each disk or triangle is reversed. Suppose that $\inv_Y([{\bf x},i])$, written as a sum of generators of $\CFinfty(\H, \bar{\s})$ such that each appears exactly once, contains $[{\bf y},-j-1]$. Then there is a chain of pseudo-holomorphic disks and triangles in the maps comprising $\Phi(\bH,\H)$ from ${\bf x}$ to ${\bf y}$. This implies that there is a chain of holomorphic disks and triangles in the maps comprising $\Phi(-\bH,-\H)$ from ${\bf y}$ to ${\bf x}$, and, by considering intersections with the basepoint, that $\inv_{-Y}([{\bf y},j])$, again written as a sum of intersection points, must contain $[{\bf x}, -i-1]$. \end{proof}

We can now prove the following.

\begin{lemma} \label{lemma:identities} Under the pairing $\langle,\rangle^{\inv}$, we have the identities
\begin{align*}
\langle \alpha, \delinv_{-Y} \beta \rangle^{\inv} &= \langle \delinv_Y \alpha, \beta \rangle^{\inv} \\
\langle \alpha, U\beta \rangle^{\inv} &= \langle U\alpha, \beta \rangle^{\inv}.
\end{align*}\end{lemma}

\begin{proof} First, suppose that $\alpha =[{\bf x},i]$ and $\beta =Q[{\bf y},j]$. Let $\del_{Y}$ denote the usual Heegaard Floer differential for $Y$, so that $\delinv_Y[{\bf x},i] = \del_Y[{\bf x},i] + Q(1+\inv_Y)[{\bf x},i]$. Similarly, $\delinv_{-Y}(Q[{\bf y},j])=Q\del_{-Y}[{\bf y},j]$. Thus, we have
\begin{align*}
\langle \delinv_Y \alpha, \beta \rangle^{\inv} &= \langle \del_Y[{\bf x},i] + Q(1+\inv_Y)[{\bf x},i], Q[{\bf y}, j]\rangle^{\inv} \\
&=\langle \del_Y[{\bf x},i], Q[{\bf y}, j]\rangle^{\inv} \\
&=\langle [{\bf x},i], Q\del_{-Y}[{\bf y}, j] \rangle^{\inv} \\
&=\langle \alpha, \delinv_{-Y} \beta \rangle^{\inv}
\end{align*}
where the third equality comes from the fact that $J$-holomorphic disks from ${\bf x}$ to ${\bf y}$ counted by $\del_{Y}$ are in one-to-one correspondence with $-J$ holomorphic disks from ${\bf y}$ to ${\bf x}$ counted by $\del_{-Y}$. The case in which $\alpha =Q[{\bf x},i]$ and $\beta=[{\bf y},j]$ can now be obtained from this case by reversing the roles of $Y$ and $-Y$.

Next, suppose that $\alpha = [{\bf x},i]$, $\beta = [{\bf y},j]$. Then $\delinv_{Y}(\alpha) = \del_Y[{\bf x},i]+Q(1+\inv_Y)[{\bf x},i]$ and $\delinv_{-Y}(\beta)=\del_{-Y}[{\bf y},j]+Q(1+\inv_{-Y})[{\bf y},j]$. We have
\begin{align*}
\langle \delinv_Y \alpha, \beta \rangle &= \langle \del_Y[{\bf x},i] + Q(1+\inv_Y)[{\bf x},i], [{\bf y}, j]\rangle^{\inv} \\
&=\langle Q(1+\inv_Y)[{\bf x},i], [{\bf y}, j]\rangle^{\inv} \\
&=\langle Q[{\bf x},i], [{\bf y}, j]\rangle^{\inv} + \langle Q [\inv_Y({\bf x}), i], [{\bf y}, j] \rangle^{\inv}\\
&=\langle [{\bf x},i], Q[{\bf y}, j]\rangle^{\inv} + \langle [{\bf x},i], Q(\inv_{-Y}[{\bf y},j])\rangle^{\inv}\\
&=\langle [{\bf x},i], Q(1+\inv_{-Y})[{\bf y}, j] \rangle^{\inv} \\
&=\langle [{\bf x}, i], \del_{-Y}[{\bf y},j]+Q(1+\inv_{-Y})[{\bf y},j] \rangle^{\inv} \\
&=\langle \alpha, \delinv_{-Y} \beta \rangle^{\inv}.
\end{align*}
The fourth, most important, equality comes from Lemma \ref{lemma:iotapair}. Finally, in the case that $\alpha = Q[{\bf x},i]$ and $\beta=Q[{\bf y},j]$, both sides of the claimed identity are trivial.

The second statement is straightforward since $U[{\bf x},i]=[{\bf x},i-1]$.
\end{proof} 

By Lemma \ref{lemma:identities} the pairing we have defined on $\CFIinf$ descends to give pairings
\begin{align*}
\langle , \rangle^{\inv} & \co \HFIinf(Y,\varpi) \times \HFIinf(-Y,\varpi) \rightarrow \Z, \\
\langle , \rangle^{\inv} & \co \HFIp(Y,\varpi) \times \HFIm(-Y,\varpi) \rightarrow \Z.
\end{align*}
In particular, for $c_1(\varpi)$ torsion, this gives isomorphisms between $\CFIinf(Y,\varpi)$ and the cochain complex $CFI_{\infty}(-Y,\varpi)$, and between $\CFIp(Y,\varpi)$ and the cochain complex $\CFI_{-}(-Y,\varpi)$, in both cases via $[{\bf x},i] \mapsto [Q{\bf x}, -i-1]^*$ and $[Q{\bf x},i] \mapsto [{\bf x},-i-1]^*$. 

At this point we recall that the analogous isomorphism on $\CFinf$ in ordinary Heegaard Floer homology, which goes by $[{\bf x},i] \mapsto [{\bf x}, -i-1]^*$, takes $\CFinf_r(Y,\s)$ to $\CF_{\infty}^{-r-2}(-Y,\s)$ \cite[Proposition 7.11]{HolDiskFour}. Because the absolute grading on $\CFIo$ is induced by the absolute grading on $\CFo$, we immediately have the following.

\begin{proposition} \label{propn:reversalisomorphisms} If $\s$ is a torsion $\spinc$ structure on $Y$ and $\varpi$ is its orbit under conjugation, there are isomorphisms
\begin{align*}
\CFIinf_r(Y, \varpi) \rightarrow \CFI_{\infty}^{-r-1}(-Y,\varpi) \\
\CFIp_r(Y, \varpi) \rightarrow \CFI_{-}^{-r-1}(-Y,\varpi)
\end{align*}
which induce isomorphisms on homology
\begin{align*}
\mathcal D^{I,\infty} \co \HFIinf_r(Y, \varpi) \rightarrow \HFI_{\infty}^{-r-1}(-Y,\varpi) \\
\mathcal D^{I,+} \co \HFIp_r(Y, \varpi) \rightarrow \HFI_{-}^{-r-1}(-Y,\varpi).
\end{align*}
\end{proposition}

\begin{proof} Let $r$ be the absolute grading of $[{\bf x},i]$ as an element of $\CFinf(Y,\s)$. Then the absolute grading of $[{\bf x},i]$ as an element of $\CFIinf(Y)$ is $r+1$ and the absolute grading of its image $Q[{\bf x},-i-1]^*$ as an element of $\CFIinf(-Y)$ is $-r-2 = -(r+1)-1$. Similarly, the absolute grading of $Q[{\bf x},i]$ as an element of $\CFIinf(-Y)$ is $r$ and the absolute grading of its image $[{\bf x},-i-1]^*$ is $-r-2+1 = -r-1$.\end{proof} 

Using the universal coefficients theorem for cohomology, Proposition~\ref{propn:reversalisomorphisms} allows us to calculate $\HFIm(-Y, \varpi)$ from knowledge of $\HFIp(Y, \varpi)$. Further, once we know $\HFIm(-Y, \varpi)$, we can get $\HFIp(Y, \varpi)$ by analyzing the second exact triangle from Proposition~\ref{propn:longexact}.

\subsection{Relation to the usual Heegaard Floer groups}
\begin{proposition} \label{propn:nonspin}
Let $\varpi \in [\Spinc(Y)]$ be an orbit of the form $\{\s, \bar \s\}$ with $\s \neq \bar \s$. Then, we have an isomorphism of graded $\Ring$-modules
$$ \HFIo(Y, \varpi) \cong \HFo(Y, \s)[-1] \oplus \HFo(Y, \s),$$
with trivial multiplication by $Q$. 
\end{proposition}

\begin{proof}
Recall that for the map $\inv \co \CFo(Y, \s) \rightarrow \CFo(Y, \bar{\s})$ is the composition of a chain isomorphism and a chain homotopy equivalence, and is therefore itself a chain homotopy equivalence, with homotopy inverse $\inv \co \CFo(Y, \bar{\s}) \rightarrow \CFo(Y, \s)$. 

For notational simplicity, consider the general situation of two $\Z_2[U]$-complexes $(A, \del_A)$ and $(B, \del_B)$ with $U$-equivariant chain homotopy equivalences $f\co A \rightarrow B$ and $g\co B \rightarrow A$ which are homotopy inverses. Consider the chain complex $C=((A\oplus B)[-1])\otimes \Z_2[Q]/(Q^2)$ with differential $\del$ given by
\begin{align*}
&a \mapsto \del_A a + Qa + Qf(a) 
&&b \mapsto \del_B b + Qb + Qg(b) \\
&Qa \mapsto Q \del_A a 
&&Qb \mapsto Q \del_B b
\end{align*}
Let $W \co A \to A$ be the homotopy between $gf$ and the identity. We do a change of basis to $C$, replacing $A$ with $A'$ consisting of elements $a' = a+f(a)+QW(a)$ and $QB$ by $QB'$ with elements $Qb' = Qb+Qg(b)$. Because $f$, $g$, and $W$ are $U$-equivariant, the change of basis map is $U$-equivariant as well. Furthermore, it is a chain isomorphism; explicitly we have
\begin{align*}
\del(a') &= \del(a+f(a)+QW(a)) \\
		&= \del_A(a) + Qa + Qf(a) + \del_B f(a) + Qf(a) + Qgf(a) + Q\del_A W(a) \\
		&= \del_A(a) + f\del_A(a) + Q(a+gf(a) + \del_A W(a))\\
		&= \del_A(a) + f\del_A(a) +QW(\del_A(a))\\
		&= (\del(a))'\\
\del(Qb') &= \del(Q(b+g(b)))\\
		 &= Q\del_B b + Q\del_B g(b)\\
		 &= (\del(Qb))' 
\end{align*}
Furthermore, with respect to this change of basis the elements in $QA$ and $B$ have the following differentials
\begin{align*}
\del(Qa) &= Q\del_A(a)\\
\del(b) &= \del_B(b) + Qb +Qg(b)\\
		&= \del_B(b)+Qb'
\end{align*}
We see that $C$ splits as a direct sum of the complexes $(A',\del|_{A'})$, $(QA,\del_A)$, and $(B\oplus QB', \del|_{B \oplus QB'})$. The last summand is acyclic, so we see that the homology of $C$ is 
$$ H_{*-1}(A') \oplus H_*(QA).$$
However, the map $(A,\del_A)\rightarrow (A', \del|_{A'})$ is a chain isomorphism, so we conclude that the homology of $C$ is isomorphic to
$$ H_{*-1}(A) \oplus H_*(A).$$ 
The $Q$ action is trivial because if $a'$ is a cycle, then $\del_A a + \del_Bf(a) + Q \del_A W(a)=0$, so in fact all three of these summands are zero. It follows that $Qa'=Q(a+f(a))=\del a,$ so $Q[a']=0$ in homology. 

Applying this result to our case with $A=\CFo(Y,\s)$, $B=\CFo(Y,\bar{\s})$, $f=\inv$ and $g=\inv$  yields the desired conclusion. \end{proof}

In light of the previous proposition, we will focus on $\HFIo(Y, \varpi)$ for $\varpi$ consisting of a single element $\s$ with $\s = \bar \s$ (that is, $\s$ is spin). In this case we simply write $\CFIo(Y, \s)$ for $\CFIo(Y, \varpi)$, and $\HFIo(Y, \s)$ for $\HFIo(Y, \varpi)$.

\begin{proposition} \label{propn:exact} Let $\s$ be a spin structure on $Y$. Then, there is an exact triangle of $U$-equivariant maps relating $\HFIo$ to $\HFo$:
\begin{equation}
\label{pic:exact}
\begin{tikzpicture}[baseline=(current  bounding  box.center)]
\node(1)at(0,0){$\HFIo(Y,\s)$};
\node(2)at (-2,1){$\HFo(Y, \s)$};
\node(3)at (2,1){$Q \ccdot \HFo(Y,\s)[-1]$};
\path[->](2)edge node[above]{$Q(1+ \inv_*)$}(3);
\path[->](3)edge (1);
\path[->](1)edge(2);
\end{tikzpicture}
\end{equation}
Here, the map $\HFIo(Y, \s) \to \HFo(Y, \s)$ decreases grading by $1$, and the other two maps are grading-preserving.
\end{proposition}

\begin{proof} This follows directly from the definitions: $\CFIo(Y,\varpi)$ is a mapping cone of chain complexes, and the homology of any mapping cone fits into such an exact triangle (cf., for example, \cite[Proposition 1.5.2]{Weibel}). 
\end{proof}

\begin{corollary} 
\label{cor:hatHFI}
Let $\s$ be a spin structure on $Y$. Then $\HFIhat(Y,\s)$ is isomorphic to the homology of the mapping cone $$\HFhat(Y,\s) \xrightarrow{Q(1+\inv_*)} Q \ccdot \HFhat(Y,\s)[-1]$$
as $\mathcal R$-modules (with trivial $U$ action). 
\end{corollary}

\begin{proof} Because $\HFhat(Y,\s)$ and $\HFIhat(Y,\s)$ are $\Z_2$-vector spaces, the exact triangle of Proposition \ref{propn:exact} splits, and we have
$$0 \rightarrow \operatorname{ker}(Q(1+\inv_*)) \rightarrow \HFIhat(Y,\s) \rightarrow \operatorname{coker}(Q(1+\inv_*))[-1] \rightarrow 0.$$
We see that $\HFIhat(Y,\s) = \operatorname{ker}(Q(1+\inv_*))[-1] \oplus \operatorname{coker}(Q(1+\inv_*))[-1]$, with the $Q$ action agreeing with the map $\operatorname{ker}(Q(1+\inv_*)) \xrightarrow{\cdot Q} \operatorname{coker}(Q(1+\inv_*))$. \end{proof}

Note that Corollary~\ref{cor:hatHFI} implies that the Euler characteristic of $\HFIhat(Y, \s)$ is always zero.

\medskip

Interestingly, the analog of Corollary~\ref{cor:hatHFI} does not hold for the other versions. For example,  in Section~\ref{subsec:lht} we will compute  $\HFIp(-\Sigma(2,3,7))$, and we will see that 
it is not isomorphic to the mapping cone of $1 + \inv_*$ on $\HFp$. Thus, to understand $\HFIp$ we need to study $\inv$ at the chain level, before taking homology.

\subsection{L-spaces} \label{sec:Lspaces}
In Heegaard Floer theory, a three-manifold $Y$ is called an {\em L-space} if for all $\spinc$ structures $\s$ on $Y$, we have $\HFhat(Y, \s) \cong \Z_2$ (in some grading) or, equivalently, $\HFp(Y, \s) \cong \Tower := \Z_2[U, U^{-1}]/\Z_2[U]$. An L-space is necessarily a rational homology sphere.

\begin{corollary} \label{corollary:lspace} Let $Y$ be an L-space, and $\s$ a spin structure on $Y$. Then 
$$\HFIp(Y, \s) \cong \HFp(Y, \s)[-1] \otimes_{\mathbb \Z_2[U]} \mathcal R.$$ \end{corollary}

\begin{proof} If $Y$ is an L-space, there is exactly one homology class in each grading $\HFp_r(Y,\s)$. Therefore, since $\inv$ is grading-preserving and $U$-equivariant, $\inv_*$ is either the zero map or the identity. Since $\inv_*^2=\Id$, we see that $\inv_*$ is the identity map on $\HFp(Y,\s)$. Therefore the map $(1+\inv_*)\co\HFp(Y,\s) \rightarrow Q\ccdot \HFp(Y,\s)[-1]$ must be zero. From the exact triangle (\ref{pic:exact}) we see that $\HFIp(Y, \s)$ is an extension of $\HFp(Y,\s)[-1]$ by $Q\ccdot \HFp(Y,\s)[-1]$.  Since these two towers are supported in different degrees mod $2\Z$, we conclude that the extension is trivial, implying the result.\end{proof}

\subsection{Cobordism maps}
\label{sec:cobordisms}
Suppose we have a connected, oriented four-dimensional cobordism $W$ between connected three-manifolds $Y$ and $Y'$. In \cite{HolDiskFour}, Ozsv\'ath and Szab\'o construct maps
$$ F^{\circ}_{W, \s} \co \HFo(Y, \s|_Y) \to \HFo(Y', \s|_{Y'}),$$
where $\s$ is a $\spinc$ structure on $W$. Strictly speaking, the maps $F^{\circ}_{W, \s}$ also depend on the choice of a path $\gamma$ from the basepoint on $Y$ to the basepoint on $Y'$; see \cite{JuhaszSutured} and \cite{Zemke} for a discussion of this in the hat case. However, we drop $\gamma$ from the notation for simplicity.

There is a conjugation symmetry acting on $\spinc$ structures on $W$. If $\varpi$  is an equivalence class under this symmetry (consisting of either one or two elements), we define
$$ F^{\circ}_{W, \varpi} = \sum_{\s \in \varpi} F^{\circ}_{W, \s} \co \HFo(Y, \varpi|_Y) \to \HFo(Y', \varpi|_{Y'}).$$

The purpose of this subsection is to construct similar maps in involutive Heegaard Floer homology.

\begin{proposition}
\label{prop:mapsHFI}
With $W$ and $\varpi$ as above, there exist maps
$$ F^{I, \circ}_{W, \varpi, \aa} \co \HFIo(Y, \varpi|_Y) \to \HFIo(Y', \varpi|_{Y'})$$
depending on some additional data $\aa$, such that the exact triangles of the form \eqref{pic:exact} fit into commutative diagrams
\begin{equation}
\label{eq:commd}
\xymatrix{
\dots \ar[r] &\HFo(Y, \varpi|_Y) \ar[d]_{F^{\circ}_{W, \varpi}} \ar[r]^-{Q(1+\inv_*)} &Q\ccdot \HFo(Y, \varpi|_Y)[-1]  \ar[d]_{F^{\circ}_{W, \varpi}} \ar[r] &\HFIo(Y, \varpi|_Y)  \ar[d]_{F^{I, \circ}_{W, \varpi, \aa}} \ar[r] &\dots \\
\dots \ar[r] &\HFo(Y', \varpi|_{Y'})  \ar[r]^-{Q(1+\inv_*)} & Q\ccdot \HFo(Y', \varpi|_{Y'})[-1]  \ar[r] & \HFIo(Y', \varpi|_{Y'})  \ar[r] &\dots
}
\end{equation}
\end{proposition}

\begin{proof}
To begin with, we choose the following:  
\begin{enumerate}[(i)]
\item A decomposition of $W$ as
$$W = W_1 \cup_{Y_1} W_2 \cup_{Y_3} \dots \cup_{Y_{n-1}} W_n,$$
where $W_i$ is cobordism between three-manifolds $Y_{i-1}$ and $Y_i$, for each $i=1, \dots, n$, such that $Y_0=Y$ and $Y_n = Y'$. We require that exactly one of the $W_i$, say $W_j$, consists of the addition of two-handles. (There can be several two-handles in $W_j$, or even none. Thus, $W_j$ is given by surgery on a possibly empty framed link $\mathbb{L} \subset Y_{j-1}$.) Further, for $i < j$, the cobordism $W_i$ consists of a single one-handle addition, and for $i > j$, it consists of a single three-handle addition;
\item A basepoint $z$ on $Y$ such that attaching the handles in each $W_i$ is always done away from $z$; this gives a path $\gamma$ from $z \in Y$ to a basepoint $z'\in Y'$.
\item For each $i < j$ (so that $W_i$ is a one-handle), a choice of Heegaard pairs $\H_{i-1}$ for $Y_{i-1}$ and $\H'_i$ for $Y_i$, such that $\H'_i$ is obtained from $\H_{i-1}$ by a connected sum with a standard diagram for $S^1 \times S^2$, as in \cite[Section 4.3]{HolDiskFour};
\item For the value $j$ such that $W_j$ consists of two-handles, a choice of a bouquet for the framed link $\mathbb{L}$, as well as a Heegaard triple (together with almost complex structures) subordinate to that bouquet, as in \cite[Section 4.1]{HolDiskFour}. When restricted to $Y_{j-1}$ and $Y_j$ this gives Heegaard pairs $\H_{j-1}$ and $\H'_j$, respectively;
\item For each $i > j$ (so that $W_i$ is a three-handle), a choice of Heegaard pairs $\H_{i-1}$ for $Y_{i-1}$ and $\H'_i$ for $Y_i$, such that $\H_{i-1}$ is obtained from $\H'_i$ by a connected sum with a standard diagram for $S^1 \times S^2$, as in \cite[Section 4.3]{HolDiskFour};
\item For each $i=1, \dots, n-1,$ a sequence of Heegaard moves relating the Heegaard pairs $\H'_i$ to $\H_i$ for $Y_i$. These give rise to chain homotopy equivalences $\Phi(\H'_i, \H_i)$ from $\CFo(\H'_i, \varpi|_{Y_i})$ to $\CFo(\H_i, \varpi|_{Y_i}).$
\end{enumerate}

Note that a decomposition of $W$ as in (i) above can be obtained from a self-indexing Morse function on $W$. Moreover, the data (i)-(vi) is what was needed to define the cobordism maps between ordinary Heegaard Floer complexes in \cite{HolDiskFour}.

To define cobordism maps between involutive Heegaard Floer complexes, we will use some additional choices $\aa_i$ and $\bb_i$ to construct chain maps
$$ f_i = F^{I, \circ}_{W_i, \varpi|_{W_i}, \aa_i} \co \CFIo(\H_{i-1}, \varpi|_{Y_i}) \to  \CFIo( \H'_{i}, \varpi|_{Y_{i}})$$
for $i=1, \dots, n,$ and
$$ g_i = \Phi^I(\H'_i, \H_i; \bb_i) \co \CFIo(\H'_i, \varpi|_{Y_i}) \to \CFIo(\H_i, \varpi|_{Y_i})$$
for $i=1, \dots, n-1$. Note that before defining $f_i$ and $g_i$, we will first need to construct their domains and targets. (Indeed, recall that the involutive Heegaard Floer chain complexes depend on sequences of moves that relate a diagram to its conjugate.) Once $f_i$ and $g_i$ are constructed, we will set
$$ f^{I,\circ}_{W, \varpi, \aa} = f_n \circ g_{n-1} \circ f_{n-1} \circ \dots \circ f_2 \circ g_1 \circ f_1$$
and then let $F^{I,\circ}_{W, \varpi, \aa}$ be the map induced by $ f^{I,\circ}_{W, \varpi, \aa}$ on homology. The total data $\aa$ will consist of (i)-(vi) above, together with the choices $\aa_i$ and $\bb_i$ at each step.

For each $i=1, \dots, n$, we define $f_i$ as follows. We choose a sequence of Heegaard moves from $\bH_{i-1}$ to $\H_{i-1}$. These produce a chain homotopy equivalence 
$$\Phi(\bH_{i-1}, \H_{i-1})\co \CFo(\bH_{i-1}, \varpi|_{Y_{i-1}}) \to \CFo(\H_{i-1}, \varpi|_{Y_{i-1}})$$
which can be used to construct the involutive complex $\CFIo(\H_{i-1}, \varpi|_{Y_{i-1}}).$ We also choose (independently) a sequence of moves from $\bH'_{i}$ to $\H'_{i}$, which give a homotopy equivalence $\Phi(\bH'_i, \H'_i)$ and a complex $\CFIo(\H'_i, \varpi|_{Y_i}).$ We now consider the diagram 
\begin{equation}
\label{eq:handles}
\xymatrixcolsep{4pc}\xymatrix{
 \CFo(\bH_{i-1}, \varpi|_{Y_{i-1}}) \ar[d]_{\Phi(\bH_{i-1}, \H_{i-1})} \ar[r]^{ \bar f^{\circ}_{W_i, \varpi|_{W_i}}} & \CFo(\bH'_i, \varpi|_{Y_i})\ar[d]^{\Phi(\bH'_i, \H'_i)}\\
 \CFo(\H_{i-1}, \varpi|_{Y_{i-1}}) \ar[r]^{f^{\circ}_{W_i, \varpi|_{W_i}} } & \CFo(\H'_i, \varpi|_{Y_i})
}
\end{equation}
where the horizontal maps are chain maps induced by handle additions in Heegaard Floer theory, as in \cite{HolDiskFour}. Note that these horizontal maps also depend on the data $\aa$. Observe that the compositions $\Phi(\bH'_i, \H'_i) \circ \bar f^{\circ}_{W_i, \varpi|_{W_i}}$ and $f^{\circ}_{W_i, \varpi|_{W_i}}  \circ \Phi(\bH_{i-1}, \H_{i-1})$ can also be viewed as cobordism maps associated to $W_i$. By the well-definedness results for cobordism maps proved in \cite{HolDiskFour}, any two such maps are related by chain homotopies. It follows that the diagram \eqref{eq:handles} commutes up to chain homotopy. If we let $\Upsilon_i$ be a chain homotopy of this type, we can combine it with $\bar f^{\circ}_{W_i, \varpi|_{W_i}}$ and $ f^{\circ}_{W_i, \varpi|_{W_i}}$ to construct the desired map $f_i$ from $\CFIo(\H_{i-1}, \varpi|_{Y_i})$ to  $\CFIo( \H'_{i}, \varpi|_{Y_{i}})$; we are using here conjugation invariance of the cobordism maps (Theorem 3.6 in \cite{HolDiskFour}) at the chain level. Note that the data $\aa_i$ needed for defining $f_i$ consists of the two sequences of Heegaard moves, together with the chain homotopy $\Upsilon_i$.

The maps $g_i$ are constructed just as in the proof of Proposition~\ref{prop:CFI}. The corresponding data $\bb_i$ consists of chain homotopies $\Upsilon(\H'_i, \H_i)$ from $\CFI(\H'_i, \varpi|_{Y_i})$ to $Q \ccdot \CFo(\H_i, \varpi|_{Y_i}) [-1].$

This concludes the definition of the map $F^{I, \circ}_{W, \varpi, \aa}$. The commutativity of \eqref{eq:commd} follows from the construction.
\end{proof}

\begin{remark}
We conjecture that the map $F^{I, \circ}_{W, \varpi, \aa}$ depends on $\aa$ only through the choice of the path $\gamma$. However, since we did not prove naturality for $\HFIo$ (cf. the remarks at the end of Section~\ref{sec:remarks}), we cannot prove the conjecture with the available technology. In fact, since $\aa$ includes a choice of Heegaard diagram for $Y$ and $Y'$, even the target and domain of  $F^{I, \circ}_{W, \varpi, \aa}$ are not yet well-defined as three-manifold invariants; only their isomorphism classes are.
\end{remark}

We have the following analogue of the composition law in Heegaard Floer theory, \cite[Theorem 3.4]{HolDiskFour}.

\begin{proposition}
\label{prop:compose}
Suppose we have cobordisms $W$ from $Y$ to $Y'$, and another cobordism $W'$ from $Y'$ to $Y''$. 
Let $W$ be equipped with an equivalence class of $\spinc$ structures $\varpi$, and with some additional data $\aa$ as in Proposition~\ref{prop:mapsHFI}. Similarly, we let $W'$ be equipped with a class $\varpi'$ and data $\aa'$. Then, we can find data $\aa_{\operatorname{tot}}$ for the cobordism $W \cup W'$, such that the following gluing result holds:
$$ F^{I,\circ}_{W', \varpi', \aa'} \circ F^{I,\circ}_{W, \varpi, \aa} = \sum_{\{\zeta \in [\Spinc(W \cup W')] \mid  \zeta|_{W}=\varpi, \zeta|_{W'}=\varpi'\}} F^{I,\circ}_{W \cup W', \zeta, \aa_{\operatorname{tot}}}.$$
\end{proposition}

The proof of Proposition~\ref{prop:compose} is similar to that of  \cite[Theorem 3.4]{HolDiskFour}.

Next, let us recall, from \cite[Theorem 7.1]{HolDiskFour}, that if $W$ is a cobordism from $Y$ to $Y'$ equipped with a $\spinc$ structure $\s$ whose restrictions $\s|_{Y}$ and $\s|_{Y'}$ are both torsion, then for $x \in \HFo(Y,\s|_{Y})$,
\begin{equation}
\label{eq:GradingShift}
{\gr}(F^\circ_{W,\s}(x)) - {\gr}(x) = \frac{c_1(\s)^2 - 2\chi(W) - 3\sigma(W)}{4}.
\end{equation}
Because our cobordism maps are induced by the Ozsv{\'a}th-Szab{\'o} maps, we have the same result for involutive Heegaard Floer homology:

\begin{lemma}Let $W$ be a cobordism from $Y$ to $Y'$ equipped with an equivalence class of $\spinc$ structures, $\varpi$, whose restrictions $\varpi|_{Y}$ and $\varpi|_{Y'}$ are both torsion. Let $c_1(\varpi)$ be $c_1(\s)$ for any $\s \in \varpi$. Let also $\aa$ be some additional data for $W$, as in Proposition~\ref{prop:mapsHFI}. Then for $x \in \HFIo(Y,\varpi|_{Y})$,
$$
{\gr}(F^{I,\circ}_{W, \varpi, \aa}(x)) - {\gr}(x) = \frac{c_1(\varpi)^2 - 2\chi(W) - 3\sigma(W)}{4}.
$$
\end{lemma}
\medskip

\section{New correction terms}
\label{sec:ds}

\subsection{Definitions}
In this section, we introduce new correction terms arising from $\HFIp(Y)$. In order to motivate the definition, recall from \cite{AbsGraded} that the ordinary correction term $d(Y, \mathfrak s)$ associated to a rational homology three-sphere $Y$ and a $\spinc$ structure $\s$ is the lowest homological degree of any element in $U^n\HFp(Y, \mathfrak s)$, where $n\gg 0$. Equivalently, $d(Y,\s)$ is the minimal grading $r$ such that the map $\pi \co \HFinf_r(Y,\s) \rightarrow \HFp_r(Y,\s)$ is nontrivial. Or, more concretely, one can show that $\HFp(Y, \s)$ can be decomposed (non-canonically) as $\T^{+}\oplus \HFred^+$, where $\T^+$ is an infinite $U$-tower as in Section~\ref{sec:Lspaces}, and $\HFred^+ := \coker(\pi)$ is a finite dimensional $\mathbb Z_2$-vector space (with some $U$-action). Then, $d(Y, \s)$ is simply the lowest degree of an element in $\mathcal T^{+}$.

We mimic this construction to produce two new correction terms. For $\mathfrak s$ a spin structure on $Y$, we consider the exact triangle \eqref{pic:exact},
\begin{equation}
\begin{tikzpicture}[baseline=(current  bounding  box.center)]
\label{pic:exact2}
\node(1)at(0,0){$\HFIp(Y,\s)$};
\node(2)at (-2,1){$\HFp(Y, \s)$};
\node(3)at (2,1){$Q \ccdot \HFp(Y,\s)[-1]$};
\path[->](3)edge node[below right]{$g$}(1);
\draw[->](1)edge node[below left]{$h$}(2);
\path[->](2)edge node[above]{$Q (1+\inv_*)$}(3);
\end{tikzpicture}
\end{equation}
consisting of $U$-equivariant maps, here given names for convenience. For $r \gg 0$, we have that $\HFp_r(Y, \s)$ is either trivial or a one-dimensional $\Z_2$-vector space. Since $\inv_*$ is an isomorphism, this implies that $\inv_*$ is the identity, so $Q (1+\inv_*)$ is trivial. Hence, the elements of $\HFp_r(Y, \s)$ for $r$ large are of the form $h(x)$ for $x \in \HFIp(Y,\s)$, necessarily such that $x \in \operatorname{Im}(U^n), x \not \in \operatorname{Im}(U^nQ)  \text{ for } n \gg 0$. This allows us to define the {\em lower involutive correction term} as follows.
$$\dl(Y,\s) = \min \{r \mid \exists \ x \in \HFIp_r(Y, s), x \in \operatorname{Im}(U^n), x \not \in \operatorname{Im}(U^nQ)  \text{ for } n \gg 0\} - 1.$$
Meanwhile, if $y \in Q\HFp_r(Y, \s)[-1]$ for $r$ large, then $y \in \operatorname{Im}(U^n)$ and $y$ is not in the image of $Q (1+\inv_*)$, so it must map to a non-zero element $x = g(y)  \in \operatorname{Im}(U^nQ).$ We define the {\em upper involutive correction term} as 
$$\du(Y,\s) = \min \{r \mid \exists \ x \in \HFIp_r(Y, s), x \neq 0, \ x  \in \operatorname{Im}(U^nQ)  \text{ for } n \gg 0\}.$$

More concretely, the exact triangle \eqref{pic:exact2} implies that, as a $\Z_2[U]$-module, $\HFIp(Y,\s)$ decomposes (non-canonically) as 
$$\T^+ \oplus \T^+ \oplus \HFIp_{\operatorname{red}}(Y, \s),$$
where $\HFIp_{\operatorname{red}}(Y, \s)$ is a finite dimensional $\Z_2$-vector space. One infinite tower $\T^+$ contains elements in gradings congruent to $d(Y, \s)+1$ modulo $2\Z$; we call it the {\em first tower}.  The other tower lies in the image of multiplication by $Q$ and has elements in gradings congruent to $d(Y, \s)$ modulo $2\Z$. We call it the {\em second tower}. Thus, $\dl(Y, \s)$ is one less than the grading of the lowest element in the first tower, and $\du(Y, \s)$ is the grading of the lowest element in the second tower. It is worth stressing that this decomposition of $\HFIp(Y,\s)$ is only as a direct sum of $\mathbb Z_2[U]$ modules; the $\mathcal R$-module structure need not respect this decomposition.

Alternatively, we can think of $\du(Y,\s)=t$ as the minimal degree such that $t \equiv d(Y,\s)$ modulo $2\Z$  and $\pi \co \HFIinf_t(Y,\s) \rightarrow \HFIp_t(Y,\s)$ is nontrivial, and of $\dl(Y,\s)+1=r$ as the minimal degree such that $r \equiv d(Y,\s)+1$ modulo $2\Z$  and $\pi \co \HFIinf_r(Y,\s) \rightarrow \HFIp_r(Y,\s)$ is nontrivial.

Note that, by construction, we have
\begin{equation}
\label{eq:mod2}
 \du(Y, \s) \equiv \dl(Y, \s) \equiv d(Y, \s) \pmod {2\Z}.
 \end{equation}

\subsection{Properties} Let us prove some basic properties of the invariants we have defined.

\begin{proposition} \label{propn:basicproperties} The involutive correction terms satisfy the inequalities
$$ \dl(Y, \s)\leq d(Y, \s) \leq \du(Y, \s).$$
\end{proposition}

\begin{proof} Looking at the exact triangle \eqref{pic:exact2}, let $x=g(y)$ be the element of lowest degree in the image of $U^ng$ for $n \gg 0$, so that the homological degree of $x$ is $\du(Y,s)$. We have $\operatorname{deg}(x)=\operatorname{deg}(y)$. Furthermore, by assumption $x \in U^n \HFIp(Y, \s)$ for $n \gg 0$, and $g$ is $U$-equivariant, leading us to conclude that $y \in U^nQ\HFp(Y, \s)[-1]$ for $n \gg 0$. Therefore $d(Y, \s) \leq \deg(y)=\du(Y, \s)$. 

By analyzing the map $h$ in \eqref{pic:exact2}, we obtain the other inequality.\end{proof}

\begin{proposition}\label{prop:reversal} The involutive correction terms are related under orientation reversal by $$\dl(Y, \s) = -\du(-Y, \s).$$
\end{proposition}

\begin{proof} Let $\dl(Y,\s)+1=r$ be the minimal degree such that $r \equiv d(Y,\s)+1$ modulo $2\Z$  and $\pi \co \HFIinf_r(Y,\s) \rightarrow \HFIp_r(Y,\s)$ is nontrivial. In view of the long exact sequence
$$\cdots \rightarrow \HFIm(Y,\s) \xrightarrow{i} \HFIinf(Y,\s) \xrightarrow{\pi} \HFIp(Y,\s) \rightarrow \cdots$$
we see that if $k$ is the maximal grading such that $k \equiv d(Y,\s)+1$ modulo $2\Z$  and the map $i_{k}$ is nontrivial, then $k=r-2=\dl(Y,\s)-1$. Similarly, since $\du(Y,\s)=t$ is the minimal grading such that $t \equiv d(Y,\s)$ modulo $2\Z$  and the map $\HFIinf_t(Y,\s) \xrightarrow{\pi} \HFIp_t(Y,\s)$ is nontrivial, we see that if $\ell$ is the maximal grading such that $\ell \equiv d(Y,\s)$ modulo $2\Z$  and $i_{\ell}$ is nontrivial, then $\ell = t-2 = \du(Y,\s)-2$.

Now, from the discussion in Section \ref{sec:prop} leading up to Proposition \ref{propn:reversalisomorphisms} , we have a commutative diagram
\begin{center}
\begin{tikzpicture}
\node(1){$\HFIinf_r(Y,\s)$};
\node(2)[right = 2 cm of 1]{$\HFIp_r(Y,\s)$};
\node(3)[below = 1 cm of 1]{$\HFI_{\infty}^{-r-1}(-Y, \s)$};
\node(4)[below = 1 cm of 2]{$\HFI_{-}^{-r-1}(-Y, \s)$};
\path[->](1)edge node[above] {$\pi_r$}(2);
\path[->](3)edge node[above]{$i^{-r-1}$}(4);
\path[->](1)edge node[left]{$\mathcal D^{I,\infty}$}(3);
\path[->](2)edge node [right] {$\mathcal D^{I,+}$}(4);
\end{tikzpicture}
\end{center}
The vertical maps are the duality isomorphisms between involutive Heegaard Floer homology and cohomology introduced in Proposition~\ref{propn:reversalisomorphisms}. Using the universal coefficients theorem and the fact that $\HFIinf(Y,\s)$ is a free module in every dimension, we see that the image of $i^{-r-1}$ is nontrivial if and only if the image of $i_{-r-1}$ is nontrivial. This implies that if $r=\dl(Y,\s)+1$ is the minimal grading such that $r \equiv d(Y,\s)+1$ modulo $2\Z$  and $\pi_r$ is nontrivial, then $-r-1 = -\dl(Y,\s)-2$ is the maximal grading such that $i_{-r-1}\co \HFIm_{-r-1}(-Y,\s)\rightarrow \HFIinf_{-r-1}(-Y,\s)$ is nontrivial. Therefore $\du(-Y,\s) = (-\dl(Y,\s)-2)+2=-\dl(Y,\s)$.\end{proof}

Next, recall from \cite[Theorem 1.3]{AbsGraded} that, when $Y$ is an integer homology sphere, the  correction term $d(Y) \in 2\Z$ is related to the Casson invariant $\lambda(Y) \in \Z$ and to the Euler characteristic $\chi(\HFred^+(Y))$ by the formula
$$ \lambda(Y)=\chi(\HFred^+(Y)) - \frac{1}{2}d(Y).$$

In the involutive setting, here is the formula for the Euler characteristic of the group $\HFIred^+(Y)$ defined in \eqref{eq:HFIred}.

\begin{proposition}
For any integer homology sphere $Y$, we have
$$\chi(\HFIred^+(Y))= \frac{1}{2}\bigl( \du(Y) - \dl(Y) \bigr).$$ 
\end{proposition}

\begin{proof}
Consider the exact triangle~\eqref{pic:exact} relating $\HFIp(Y)$ to $\HFp(Y)$. Let us truncate all the groups and thus focus on degrees $\leq 2n-1$, for $n \gg 0$. Since $d(Y)$ is even, we have $\HFp_{2n-1}(Y)=0$ for $n \gg 0$, so there is an exact triangle relating the truncated groups. By taking Euler characteristics we get
$$\chi (\HFI^+_{\leq 2n-1}(Y))= \chi (\HF^+_{\leq 2n-1}(Y))-\chi (\HF^+_{\leq 2n-1}(Y))=0.$$

On the other hand, $\HFI^+_{\leq 2n-1}(Y)$ decomposes (non-canonically) as a direct sum of $\HFIred^+(Y)$ and two truncated $U$-towers. The first truncated tower starts in degree $\dl(Y)+1$ and ends in degree $2n-1$, and thus has $\tfrac{1}{2}\dl(Y) + n$ generators. The second truncated tower starts in degree $\du(Y)$ and ends in degree $2n-2$, so it has $\tfrac{1}{2}\du(Y) + n$ generators. The conclusion follows readily from this.
\end{proof}

We now prove the version of Fr{\o}yshov's inequality for spin cobordisms that was announced in the 
Introduction.

\begin{proof}[Proof of Theorem~\ref{thm:froyshov}] Recall from \cite[Section 9]{AbsGraded} that if $W$ is a negative definite cobordism from $Y_1$ to $Y_2$, then the map $F_{W,\t}\co \HFinf(Y_1,\s_1) \rightarrow \HFinf(Y_2,\s_2)$ is an isomorphism. (Here $\s_1 = \t|_{Y_1}$ and $\s_2 = \t|_{Y_2}$.) Therefore, from the diagram of long exact sequences \eqref{eq:commd},
\begin{center}
\begin{tikzpicture}[node distance = 2 cm, auto]
\node(1){$\HFIinf(Y_1,\s_1)$};
\node(2)[left = 1 cm of 1]{$Q\ccdot \HFinf(Y_1,\s_1)[-1]$};
\node(3)[right = 1 cm of 1]{$\HFinf(Y_1,\s_1)$};
\node(4)[left =1 cm of 2]{$\cdots$};
\node(5)[right = 1 cm of 3]{$\cdots$};
\node(6)[below = 1 cm of 1]{$\HFIinf(Y_2,\s_2)$};
\node(7)[below = 1 cm of 2]{$Q\ccdot \HFinf(Y_2,\s_2)[-1]$};
\node(8)[below = 1 cm of 3]{$\HFinf(Y_2,\s_2)$};
\node(9)[below = 1.2 cm of 4]{$\cdots$};
\node(10)[below = 1.2 cm of 5]{$\cdots$};
\path[->](4)edge(2);
\path[->](2)edge(1);
\path[->](1)edge(3);
\path[->](3)edge(5);
\path[->](9)edge(7);
\path[->](7)edge(6);
\path[->](6)edge(8);
\path[->](8)edge(10);
\path[->](1)edge node[left]{$F_{W, \t, \aa}^{I, \infty}$}(6);
\path[->](2)edge node[left]{$F_{W,\t}^{\infty}$}(7);
\path[->](3)edge node[left]{$F_{W, \t}^{\infty}$}(8);
\end{tikzpicture}
\end{center}
and the five lemma we see that $F^{I, \infty}_{W, \t, \aa}$ is also an isomorphism. Notice that this isomorphism is also $\mathcal R$-equivariant. 

Now, following the strategy of \cite[Proposition 9.6]{AbsGraded}, we delete a ball from $X$ to obtain  a spin, negative definite cobordism $W$ from $S^3$ to $Y$. The isomorphism $\HFinf(S^3) \rightarrow \HFinf(Y,\s)$ shifts gradings upward by $-\frac{2\chi(W) + 3\sigma(W)}{4} = \frac{b_2(X)}{4}$, so $d(Y,\s) = \frac{b_2(X)}{4} \mod 2\Z$. Consider the following commutative square.
\begin{center}
\begin{tikzpicture}
\node(1){$\HFIinf_r(S^3,\s_0)$};
\node(2)[right = 1 cm of 1]{$\HFIinf_{\dl(Y)+1}(Y,\s)$};
\node(3)[below = 1 cm of 1]{$\HFIp_r(S^3,\s_0)$};
\node(4)[below = .9 cm of 2]{$\HFIp_{\dl(Y)+1}(Y,\s)$};
\path[->](1)edge node[above]{$F_{W,\t, \aa}^{I,\infty}$}(2);
\path[->](3)edge node[above]{$F_{W,\t, \aa}^{I,+}$}(4);
\path[->](1)edge node[left]{$\pi$}(3);
\path[->](2)edge node[right]{$\pi$}(4);
\end{tikzpicture}
\end{center}
We see there must be some element $y \in \HFIp(S^3)$ with the property that $\operatorname{gr}(F^{I,+}_{W,\t, \aa}(y)) = \dl(Y,\s) + 1$. Because the grading shift in the cobordism map is still by $\frac{b_2(X)}{4}$ and $\dl(Y,\s) \equiv d(Y,\s)$ modulo $2\Z$ , the homological grading of $y$ must be odd. (That is, $y$ lies in the first tower in $\HFIp(S^3)$.) But the lowest odd homological grading in $\HFIinf(S^3)$ is 1, and therefore we have
$$\frac{b_2(X)}{4} = \operatorname{gr}(F^{I,+}_{W,\t, \aa}(y)) - \operatorname{gr}(y) \leq (\dl(Y,\s)+1)-1.$$
The conclusion follows.\end{proof}

\begin{proposition} \label{prop:ddTheta} Let $(Y_1, \s_1)$ and $(Y_2,\s_2)$ be rational homology spheres equipped with spin structures, and let $(W,\t)$ be a spin rational homology cobordism between them. Then the involutive correction terms of $(Y_1,\s_1)$ and $(Y_2,\s_2)$ are equal.\end{proposition}

\begin{proof} As in the proof of Theorem~\ref{thm:froyshov}, the map $F_{W,\t, \aa}^{I, \infty} \co \HFIinf(Y_1,\s_1) \rightarrow \HFIinf(Y_2,\s_2)$ is an isomorphism, and consequently 
\begin{align*}
\dl(Y_2,\s_2)-\dl(Y_1, \s_1) \geq \frac{-2\chi(W)-3\sigma(W)}{4}, \\
\du(Y_2,\s_2)-\du(Y_1, \s_1) \geq \frac{-2\chi(W)-3\sigma(W)}{4}.
\end{align*}
For $W$ a rational homology cobordism, this implies that $\dl(Y_2,\s_2) \geq \dl(Y_1, \s_1)$ and $\du(Y_2,\s_2) \geq \du(Y_1, \s_1)$. But reversing the orientation of $W$ gives the opposite inequalities as well.\end{proof}

\begin{proof}[Proof of Theorem~\ref{thm:ds}]
Every $\Z_2$-homology sphere has a unique spin structure, and a $\Z_2$-homology cobordism between $\Z_2$-homology spheres is a spin rational homology cobordism. Thus, the claim about $\dl$ and $\du$ descending to $\Theta^3_{\Z_2}$ follows from Proposition~\ref{prop:ddTheta}. 

When $Y$ is an integer homology sphere, $d(Y)$ is an even integer (cf. \cite{AbsGraded}). In view of \eqref{eq:mod2}, so are $\dl(Y)$ and $\du(Y)$. 
\end{proof}
 
The maps $\dl$ and $\du$ are not group homomorphisms; see Section~\ref{sec:nonadditive} below for an explanation.

\section{The large surgery formula}
\label{sec:surgery}

Ozsv\'ath and Szab\'o \cite{Knots} and, independently, Rasmussen \cite{RasmussenThesis} proved a formula expressing the Heegaard Floer homology of a large surgery on a knot in terms of the knot Floer complex. Our goal here is to prove an involutive analogue of their formula.

Throughout the section, $Y$ will be an oriented integer homology three-sphere, and $K \subset Y$ will be an oriented knot. (However, with minor modifications, all our results can be extended to null-homologous knots in any three-manifold $Y$.) In practice, we will mostly be interested in the case $Y=S^3$.

\subsection{The conjugation symmetry on the knot Floer complex}
\label{sec:iotaHFK}
Let $ H = (\Sigma, \alphas, \betas, w, z)$
be a doubly-pointed Heegaard diagram that represents $K \subset Y$, and let $J$ be a suitable family of almost complex structures on the symmetric product. We then say that $\H=(H, J)$ is a choice of {\em Heegaard data} for $K$.

Ozsv\'ath-Szab\'o and Rasmussen \cite{Knots, RasmussenThesis} defined a version of Heegaard Floer homology for knots, called knot Floer homology; see \cite{HFKsurvey} for a survey. In their theory, to Heegaard data $\H$ one associates a $\Z$-graded, doubly filtered complex $\CFKi(\H)$, as follows. First, one defines maps $A, M \co \Ta \cap \Tb \to \Z$, called the Alexander and Maslov gradings. Then, as a $\Z_2$-vector space, $\CFKi(\H)$ is set to have generators
$$ U^{-i}\x= [\x, i, j], \ \ \x \in \Ta \cap \Tb, \ i, j \in \Z \text{ such that } A(\x) = j-i.$$
We can also view $\CFKi(\H)$ as a free $\Z_2[U, U^{-1}]$-module with generators $\x \in \Ta \cap \Tb$. The differential is given by counting pseudo-holomorphic disks, with the values of $i$ and $j$ keeping track of going over the $w$ and $z$ basepoints, respectively. The $\Z \oplus \Z$ filtration is given by $\Filt([\x, i, j])=(i, j)$. The functions $A$ and $M$ can be extended to the generators of $\CFKi(\H)$, by setting
$$ A([\x, i, j])=j, \ \ M([\x, i, j])=M(\x) + 2i.$$
The knot Floer complex $\CFKi(\H)$ is usually drawn in the $(i,j)$ plane, with the generators represented by dots and the differential by arrows; see Figure~\ref{fig:lht} below for an example.

The knot Floer complex is natural:
\begin{proposition}
\label{prop:JT2}
If we fix $(Y, K, w, z)$, then the complexes $\CFKi(\H)$ form a transitive system in the homotopy category of $(\Z \oplus \Z)$-filtered chain complexes. In other words, for every two Heegaard pairs $\H$ and $\H'$ representing $(Y, K, w, z)$, we have a $(\Z \oplus \Z)$-filtered chain homotopy equivalence 
$$ \Phi(\H, \H') \co \CFo(\H, \s) \to \CFo(\H', \s),$$
well-defined up $(\Z \oplus \Z)$-filtered chain homotopy; moreover, these equivalences satisfy the analogs of conditions (i) and (ii) from the statement of Theorem~\ref{thm:JT}, with equality replaced by $(\Z \oplus \Z)$-filtered chain homotopy. 
\end{proposition}

Proposition~\ref{prop:JT2} is the analogue of Proposition~\ref{prop:JT}, and has a similar proof. Note that a version of naturality (at the level of homology) for link Floer homology was proved in \cite[Theorem 1.8]{Naturality}.

Let us now imitate the constructions from Section~\ref{sec:invo}. Let $\bh=(-\Sigma, \betas, \alphas, z, w)$ and $\bH=(\bh, \bJ)$. By \cite[Section 3.5]{Knots}, there is a canonical isomorphism
\begin{align*}
\eta_K \co \CFKi(\H) \xrightarrow{\phantom{u} \cong \phantom{u}} \CFKi(\bH),\ \ \eta([\x, i, j])=[\x, j, i].
\end{align*}

Moreover, $\H$ and $\bH$ represent the same knot $K \subset Y$, albeit with the roles of the $w$ and $z$ basepoints switched. Let $\psi: Y \to Y$ be a self-diffeomorphism of $Y$ given by a half Dehn twist along the oriented knot $K$, so that $\psi(K) = K$, $\psi(w)=z$, $\psi(z)=w$, and $\psi$ is the identity outside a small neighborhood of $K$. Note that $\psi$ is isotopic to the identity. Further, the push-forward $\psi_*(\bH)$ represents the same doubly-pointed knot as $\H$. Proposition~\ref{prop:JT2} gives a $(\Z \oplus \Z)$-filtered chain homotopy equivalence
\begin{align*}
\Phi(\bH, \H)  
\co \CFKi(\psi_*(\bH)) \xrightarrow{\phantom{u} \sim \phantom{u}} \CFKi(\H).
\end{align*}

By pre-composing $\Phi(\bH, \H)$  with the isomorphism induced by $\psi$, we obtain a $(\Z \oplus \Z)$-filtered chain homotopy equivalence
\begin{align*}
\Phi(\bH, \H) \co \CFKi(\bH) \xrightarrow{\phantom{u} \sim \phantom{u}} \CFKi(\H),
\end{align*}
well-defined up $(\Z \oplus \Z)$-filtered chain homotopy.

Then, we set
$$ \inv_K = \Phi(\bH, \H) \circ \eta_K \co \CFKi(\H) \to \CFKi(\H).$$
This is the analogue of the map $\inv$ from Section~\ref{sec:invo}. By construction, it has the property that if $[\y, i', j']$ is a term in $\inv_K([\x, i, j])$, then
\begin{equation}
\label{eq:filterediota}
M([\y, i', j'])=M([\x, i, j]), \ \ i' \leq j, \ \ j' \leq i. 
\end{equation}
The first equality says that $\inv_K$ is grading-preserving, and the two inequalities say that $\inv_K$ is {\em skew-filtered}---in the sense that it is filtered as a map from $(\CFKi(\H), \Filt)$ to $(\CFKi(\H), \bFilt)$, where $\bFilt([\x, i, j])=(j, i)$. This follows from the fact that $\Phi(\bH,H)$ is a filtered map. Furthermore, by construction, $\inv_K$ is a skew-filtered quasi-isomorphism, i.e., it induces an isomorphism between the homology of the associated graded complexes.

This motivates the following.
\begin{definition}
\label{def:CFKI}
A {\em set of $\CFKI$-data} $\Ch=(C, \Filt, M, \del, \inv)$ consists of a free, finitely generated $\Z_2[U, U^{-1}]$-module $C$, equipped with a $(\Z \oplus \Z)$-filtration $\Filt=(i,j)$ and a $\Z$-grading $M$ such that the action of $U$ decreases $\Filt$ by $(1,1)$ and $M$ by $2$. Further, there is a differential $\del \co C \to C$ that preserves $\Filt$ and decreases $M$ by $1$, and a grading-preserving, skew-filtered quasi-isomorphism $\inv \co C \to C$.
\end{definition}

Let $\Ch=(C, \Filt, M, \del, \inv)$ and $\Ch'=(C', \Filt', M', \del', \inv')$ be two sets of $\CFKI$-data. A {\em morphism} $(f, S) \co \Ch \to \Ch'$ consists of  a filtered (and grading-preserving) chain map $f \co C \to C'$ such that $f \circ \inv$ and $\inv' \circ f$ are chain homotopic via a skew-filtered chain homotopy $S \co C \to C'[-1]$. We say that $(f, S)$ is a {\em quasi-isomorphism} if $f$ induces an isomorphism on the homology of the associated graded complexes. We say that $\Ch$ and $\Ch'$ are quasi-isomorphic if they are related by a chain of (back and forth) quasi-isomorphisms. The composition of two morphisms $(f,S)\co \Ch \to \Ch'$ and $(g,T) \co \Ch \to \Ch'$ is given by
$$(g,T)\circ (f,S)=(g\circ f, g\circ S+ T\circ f).$$

\begin{proposition}
The quasi-isomorphism type of the set of $\CFKI$-data
$$(\CFKi(\H), \Filt, M, \del, \inv_K)$$ is an invariant of the oriented knot $K \subset Y$. 
\end{proposition}

\begin{proof}
This is similar to the proof of Proposition~\ref{prop:CFI}.
\end{proof}

\begin{remark}
Given a set of $\CFKI$-data $\Ch=(C, \Filt, M, \del, \inv)$, by analogy with the definition of $\CFI$ in Section~\ref{sec:CFI}, we can construct an associated complex
$$\CI := (C[-1] \oplus Q \ccdot C[-1], \del + Q \ccdot (1+\inv)).$$ 
A morphism, respectively quasi-isomorphism, of $\CFKI$-data induces a morphism, respectively quasi-isomorphism, between the associated $\CI$, as chain complexes over the ring $\Z_2[Q, U, U^{-1}]/(Q^2)$. However, we are losing some information in this process, because (since the identity is filtered and $\inv$ is skew-filtered) there is no natural $(\Z \oplus \Z)$-filtration on $\CI$. 
\end{remark}
 
When trying to compute the map $\inv_K$ for specific knots, we often know a simplified complex that is filtered chain homotopic to $\CFKi(\H)$, and it is helpful to transfer $\inv_K$ to that complex, to make calculations easier. Thus, the following lemma will be useful.

\begin{lemma}
\label{lem:transfer}
If $\Ch=(C, \Filt, M, \del, \inv)$ is a set of $\CFKI$-data, and let $(C', \Filt', M', \del')$ be a filtered complex as in Definition~\ref{def:CFKI}, but without the map $\inv$. Suppose that $(C, \Filt, M, \del)$ and $(C', \Filt', M', \del')$ are filtered chain homotopy equivalent. Then, there exists a grading-preserving, skew-filtered quasi-isomorphism $\inv'$ of $(C', \del')$, such that $\Ch' := (C', \Filt', M', \del', \inv')$ is a set of $\CFKI$-data quasi-isomorphic to $\Ch$. 
\end{lemma}

\begin{proof}
Let $f \co C \to C'$ and $g \co C' \to C$ be the two chain homotopy equivalences, so that $fg - 1 = \del v+v\del$ and $gf-1 = \del w + w \del$. We set $\iota' = f \iota g$. Then, we construct morphisms
$$ (f, f\iota w) \co \Ch \to \Ch', \ \ \ (g, w\iota g) \co \Ch' \to \Ch.$$
(Here $f\iota w$ plays the role of the skew-filtered chain homotopy $S \co C \rightarrow C'[-1]$ in the definition of a morphism between sets of $\CFKI$ data, and similarly for $w \iota g$.) Either of these is a quasi-isomorphism, because $f$ and $g$ are chain homotopy equivalences and hence quasi-isomorphisms.
\end{proof}

\subsection{The Sarkar map}
\label{sec:Sarkar}
Lemma~\ref{lem:inv} implies that the map $\iota$ induces an involution on Heegaard Floer homology. By contrast, for knots, recall that in constructing $\iota_K$ we used Heegaard moves that   take the basepoint $w$ to $z$ and vice versa, following arcs on the knot $K$. Therefore, a similar argument as the one in Lemma~\ref{lem:inv} shows that
 $$\iota_K^2 \sim \sarkar,$$
where $\sarkar$ is the map on $\CFKi(\H)$ induced by moving the basepoints once around $K$, i.e., associated to the positive Dehn twist around $K$ in the sense of \cite[Section 3]{SarkarMoving}. 
The map $\sarkar$ is a filtered, grading-preserving chain map, well-defined up to filtered chain homotopy equivalence. 

As we shall see in Section~\ref{sec:thin}, in some cases $\iota_K$ is determined by its behavior with respect to the grading and filtration, together with (partial) knowledge of its square $\sarkar$. 
To find $\sarkar$, we can use the results from \cite{SarkarMoving}. There, Sarkar considered the complex $g\CFKm(\H)$, freely generated over $\Z_2[U]$ by intersection points $\x \in \Ta \cap \Tb$, and with differential
$$ \del \x = \sum_{\y \in \Ta \cap \Tb} \sum_{\{ \phi \in \pi_2(\x, \y) \mid \mu(\phi)=1, n_z(\phi)=0  \}} \# \Mh(\phi)\cdot U^{n_w(\phi)}\y.$$

We can view $g\CFKm(\H)$ as the associated graded of the subcomplex $C(i \leq 0) \subset$ $\CFKi(\H)$ generated by $[\x, i, j]$ with $i \leq 0$, where the associated graded is taken with respect to the vertical filtration by $j$. 

The following is Theorem 1.1 in \cite{SarkarMoving}.
\begin{proposition}[Sarkar \cite{SarkarMoving}]
\label{prop:sarkar}
Let $K \subset S^3$ be an oriented knot, and let $\H$ be a choice of Heegaard data for $K$. Then, the map $\sarkar_g$ induced by $\sarkar$ on the complex $g\CFKm(\H)$ is chain homotopic to $1 + \Psi \Phi$, where 
$$ \Psi (\x) = \sum_{\y \in \Ta \cap \Tb} \sum_{\{ \phi \in \pi_2(\x, \y) \mid \mu(\phi)=1, n_z(\phi)=1  \}} \# \Mh(\phi)\cdot U^{n_w(\phi)}\y$$
and
$$\Phi(\x)=\sum_{\y \in \Ta \cap \Tb} \sum_{\{ \phi \in \pi_2(\x, \y) \mid \mu(\phi)=1, n_z(\phi)=0  \}} \# \Mh(\phi)\cdot n_w(\phi) U^{n_w(\phi)-1}\y.$$
Further, the square $(\sarkar_g)^2$ is chain homotopic to the identity.
\end{proposition}

Sucharit Sarkar also informed us of the following conjecture.\footnote{Since the first draft of this paper appeared, this conjecture has been proved by Ian Zemke \cite[Theorem B]{Zemke3}.}
\begin{conjecture}[Sarkar]
\label{conj:sarkar}
Let $K$ be an oriented, null-homologous knot in a three-manifold $Y$, and let $\H$ be a choice of Heegaard data for $K$. Let $\del =\sum_{i, j \geq 0} \del_{i j}$ be the differential on the complex $\CFKi(\H)$, where the term $\del_{ij}$ decreases the two filtration levels by $i$ and $j$, respectively. Then, up to $(\Z \oplus \Z)$-filtered chain homotopy, the map $\sarkar$ on the complex $\CFKi(\H)$ is given by the formula
$$ \sarkar \sim 1 + U^{-1} \bigl( \sum_{\substack {i, j \geq 0 \\ i \text{ odd} } } \del_{ij} \bigr ) \circ  \bigl ( \sum_{\substack {i, j \geq 0 \\ j \text{ odd} } } \del_{ij} \bigr ).$$
\end{conjecture}

\subsection{Preliminaries on the large surgery formula}
\label{sec:large}
Let us recall the large surgery formula for Heegaard Floer homology, and its proof. The original references are \cite{Knots} and \cite{RasmussenThesis}; here we use the notation from \cite{IntSurg}.

Fix an integer $p > 0$. Let $Y_p(K)$ be the manifold obtained by surgery along $K$ with coefficient $p$, and $W_p(K)$ be the two-handle surgery cobordism from $Y$ to $Y_p(K)$. We denote by $W'_p(K)$ the cobordism from $Y_p(K)$ to $Y$ obtained by turning around $-W_p(K)$. Let $F$ be a Seifert surface for $K$, and $\hF \subset W'_p(K)$ the surface obtained from $F$ by capping it off with the core of the two-handle. 

The $\spinc$ structures over $Y_p(K)$ are identified with the elements of $\Z/p\Z$ as follows. The  structure $\s \in \Spinc(Y_p(K))$ corresponds to $[s] \in \Z/p\Z$ if there is an extension of $\s$ to $W'_p(K)$ such that
$$ \langle c_1(\s), [\hF]\rangle - p \equiv 2s \pmod{2p}.$$

Pick Heegaard data $\H = (H, J)$ for $K \subset Y$, with 
$$H =(\Sigma, \alphas, \betas, w, z)$$ being such that $w$ and $z$ are on each side of the beta curve $\beta_g = \mu$ (the meridian of the knot). Here, $g=g(\Sigma)$ is the genus of the Heegaard surface $\Sigma$, which should not be confused with the Seifert genus of the knot, $g(K)$. 

From this we get Heegaard data $\H_p=(H_p, J_p)$ for $Y_p(K)$, where
$$ H_p=(\Sigma, \alphas, \gammas, z)$$
is such that the first $g-1$ curves in the set $\gammas$ differ from the corresponding ones in $\betas$ by small Hamiltonian isotopies, whereas the last curve $\gamma_g$ is the result of winding a knot longitude $\lambda$ a number of times around $\beta_g$. For simplicity, we can assume that $\lambda$ is the Seifert longitude, so that the winding is done $p$ times. Observe that  
$(\Sigma, \alphas, \gammas, \betas, z)$ is a triple Heegaard diagram that represents the cobordism $W'_p(K)$. See Figure~\ref{fig:triple}.

\begin {figure}
\begin {center}
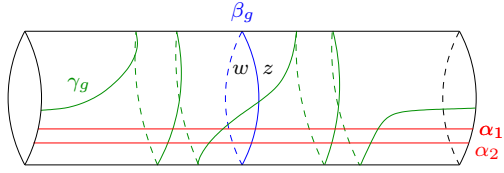
\caption {A triple Heegaard diagram that represents the cobordism $W'_p(K)$. We show here the neighborhood of the meridional curve $\beta_g$. This curve can intersect any number of alpha curves; two are shown in the picture.}
\label{fig:triple}
\end {center}
\end {figure}

For each $s \in \Z$, we consider the subcomplex of $\CFKi(\H)$ generated by $[\x, i, j]$ with $i < 0 \text{ and } j < s$:
$$A_s^- = C\{i < 0 \text{ and } j < s\}$$
Let
$$ A_s^+ = C\{i \geq 0 \text{ or } j \geq s\}.$$
be the corresponding quotient complex.

\begin{theorem}[Ozsv\'ath-Szab\'o \cite{Knots}, Rasmussen \cite{RasmussenThesis}]
\label{thm:LargeSurg}
There exists $N \geq 0$ such that for all $p \geq N$ and for each $s \in \Z$ with $|s|\leq p/2$, we have an isomorphism of relatively graded $\Z_2[U]$-modules
\begin{equation}
\label{eq:LargeSurg}
 \HFp(Y_p(K), [s]) \cong H_*(A_s^+).
 \end{equation}
\end{theorem}

The proof of Theorem~\ref{thm:LargeSurg} involves defining a chain map 
\begin{equation}
\label{eq:XiInf}
\Gamma^\infty_{p,s} \co \CFi(\H_p, s) \to \CFKi(\H)
\end{equation}
 by
\begin{equation}
\label{eq:XiFormula}
 \Gamma^\infty_{p,s} ([\x,i]) = \sum_{\y \in \Ta \cap \Tb} \sum_{\substack{\psi \in \pi_2(\x, \Theta, \y) \\ \mu(\psi)=0, \ n_w(\psi)-n_z(\psi)=s} } \# \M(\psi) \cdot [\y, i-n_w(\psi), i-n_z(\psi)].
 \end{equation}
We use here the standard notational conventions in Heegaard Floer theory. Thus, $\psi$ is a homotopy class of triangles with boundaries on the three tori $\Ta, \Tb, \Tg$, and with one vertex at the intersection point $\Theta \in \Tb \cap \Tg$ that represents the maximal degree generator in homology. Note that the homotopy class $\psi$ produces a $\spinc$ structure $\s_{\psi}$ on $W'_p(K)$, and the condition $n_w(\psi)-n_z(\psi)=s$ is equivalent to $\langle c_1(\s_\psi), [\hF] \rangle = 2s-p$.

One can easily check that $\Gamma^\infty_{p,s}$ takes the subcomplex $\CFm(\H_p, s) \subset \CFi(\H_p, s)$ to the subcomplex $A_s^- \subset \CFKi(\H)$. Therefore, we get an induced map on quotient complexes,
\begin{equation}
\label{eq:Xi}
\Gamma^+_{p,s} \co \CFp(\H_p, s) \to A_s^+.
\end{equation}

One shows that for $p \gg 0$, the map $\Gamma^+_{p,s}$ is an isomorphism of chain complexes. How large $p$ has to be depends on $s$ and on the Heegaard diagram under consideration. However, a posteriori, by using the surgery exact triangle and the adjunction inequality, we see that if the isomorphism \eqref{eq:LargeSurg} holds for $p \gg 0$, then it must hold for all $p \geq g(K)+|s|$. This implies that we can take $N=2g(K)-1$ in Theorem~\ref{thm:LargeSurg}; see \cite[Remark 4.3]{Knots}.

In fact, we have the following.

\begin{proposition}
\label{prop:LargeSurg}
For every $p, s \in \Z$ with $p \geq g(K) + |s|$, the map $\Gamma^+_{p,s}$ induces an isomorphism on homology.
\end{proposition}

\begin{proof}
By the general integer surgery formula \cite{IntSurg}, the homology $\HFp(Y_p(K), [s])$ can be computed as the homology of a mapping cone complex $\X_{s}^+(p)$ of the form
$$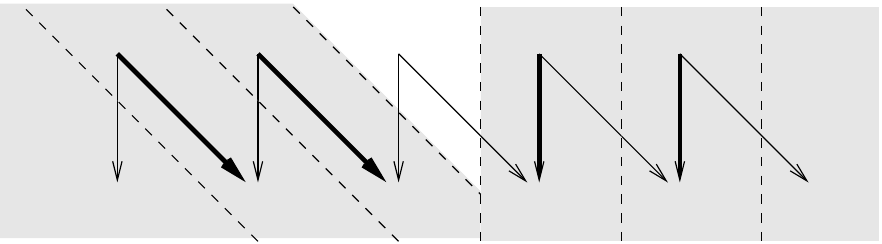$$

Here, each $B^+$ is a copy of $\CFp(Y)$, the vertical arrows represent maps $v^+_{s+pi} \co A^+_{s+pi} \to B^+$, and the diagonal arrows represent maps $h^+_{s+pi}\co A^+_{s+pi} \to B^+$. 

Observe that $A^+_s$ is a quotient complex of $\X_{s}^+(p)$. The natural projection induces a map
$$\HFp(Y_p(K), [s])  \to H_*(A_s^+).$$
Furthermore, it follows from the proof of the surgery formula in \cite{IntSurg} that this map is the same as the one induced by $\Gamma^+_{p,s}$ on homology.

Next, from the adjunction inequality for knot Floer homology (Theorem 5.1 in \cite{Knots}) we know that $ v^+_{s+pi}$ are quasi-isomorphisms for $s+pi \geq g(K)$, and $h^+_{s+pi}$ are quasi-isomorphisms for $s+pi \leq -g(K)$. Hence, if $p \geq g(K)+|s|$, all the thick arrows pictured in the mapping cone complex above are quasi-isomorphisms. By standard filtration arguments, the subcomplex of $\X_{s}^+(p)$ shaded in the picture (whose quotient complex is $A_s^+$) is acyclic. The conclusion follows.
\end{proof}

Our goal is to prove an involutive analogue of Theorem~\ref{thm:LargeSurg}. In view of Proposition~\ref{propn:nonspin}, it suffices to focus on spin structures. There are two spin structures on $Y_p(K)$ when $p$ is even, and one when $p$ is odd. In this paper we will only consider the spin structure corresponding to $s=0$.

The map $\inv_K$ on $\CFKi(\H)$ induces a map 
 $$\inv_0 \co A^+_0 \to A^+_{0}.$$ 
 
 We set
$$\AI_0^+ = (A^+_0[-1]\otimes \Z_2[Q]/(Q^2), \del + Q(1+\inv_0)).$$

The involutive analogue of Theorem~\ref{thm:LargeSurg} is then Theorem~\ref{thm:Large} from the Introduction, which states that, for $p \geq g(K)$, there is an isomorphism
\begin{equation}
\label{eq:LargeHFI}
 \HFIp(S^3_p(K), [0]) \cong H_*(\AI_0^+).
 \end{equation}

To prove \eqref{eq:LargeHFI}, it suffices to construct a chain homotopy
\begin{equation}
\label{eq:Rps}
 R^+_{p} \co \CFp(\H_p, [0]) \to A_0^+
 \end{equation}
between the compositions $\iota_0 \circ \Gamma^+_{p,0}$ and $\Gamma^+_{p,0} \circ \iota$. Indeed, if we had  
 $R^+_{p}$, we could combine it with $\Gamma^+_{p,0}$ to produce a chain map between the mapping cone complexes $\CFIp(\H_p, [0])$ and $\AI_0^+$, as in the diagram
\[
\xymatrix{
\CFp(\H_p, [0]) \ar[r]^{\phantom{uuu}\Gamma^+_{p,0}} \ar[d]_{Q(1 + \inv)} \ar[dr]^{R^+_{p}} & A_0^+  \ar[d]^{Q(1 + \inv_0)}\\
Q \ccdot \CFp(\H_p, [0])[-1] \ar[r]_{\phantom{uuuuuu} \Gamma^+_{p,0}} & Q\ccdot A_0^+[-1]
}
\]

By Proposition~\ref{prop:LargeSurg}, the horizontal maps are quasi-isomorphisms for $p \geq g(K)$. If we consider the natural two-step filtrations on the mapping cones $\CFIp(\H_p, [0])$ and $\AI_0^+$, since the map on the associated graded is a quasi-isomorphism, so is the combined map from $\CFIp(\H_p, [0])$ to $\AI_0^+$. (Compare the argument at the end of the proof of Proposition~\ref{prop:CFI}.) This would imply Theorem~\ref{thm:Large}.

It remains to construct the chain homotopy $R^+_{p}$. Note that $\Gamma^+_{p,0}$ is a filtered version of the cobordism map associated to the cobordism $W'_p(K)$. Thus, the construction of $R^+_{p,0}$ will be inspired by the proof of conjugation invariance for cobordism maps, Theorem 3.6 in \cite{HolDiskFour}; compare Section~\ref{sec:cobordisms}. Roughly, since the map $\iota$ on $ \CFp(\H_p, [0])$ is determined by a sequence of Heegaard moves from $\bH_p$ to $\H_p$, and the map $\iota_0$ on $A_0^+$ is determined by a sequence of moves from $\bH$ to $\H$, what we need to do is to choose these sequences in a compatible way, so that $R^+_{p}$ will be given by a suitable count of pseudo-holomorphic quadrilaterals. The way to choose the two sequences of Heegaard moves will be explained in Section~\ref{sec:HMsurgery}.

\subsection{Compound stabilizations}

Before moving forward, it is helpful to introduce a new kind of composite Heegaard move. Suppose we are given a Heegaard diagram $H=(\Sigma, \alphas, \betas, z)$ and a path $\zeta$ on $\Sigma$ that is disjoint from the $\beta$ curves and the $z$ basepoint. The path $\zeta$ may intersect some $\alpha$ curves in its interior (such as $\alpha_1, \alpha_2, \alpha_3$ shown in the top left corner of Figure~\ref{fig:compound}). An {\em compound $\alpha$-stabilization} along $\zeta$ consists of introducing a one-handle on $\Sigma$ with its feet at the ends of $\zeta$, turning $\zeta$ into a new $\beta$-circle, and making the co-core of the handle into a new $\alpha$-circle.\footnote{Compound $\alpha$-stabilizations are the same as $(0, l)$-stabilizations, in the terminology of \cite[Definition 6.26]{Naturality}.} This construction is shown in the top row of Figure~\ref{fig:compound}. It can be viewed as the composition of ordinary Heegaard moves (an isotopy, a stabilization, and some $\alpha$-handleslides). 

\begin {figure}
\begin {center}
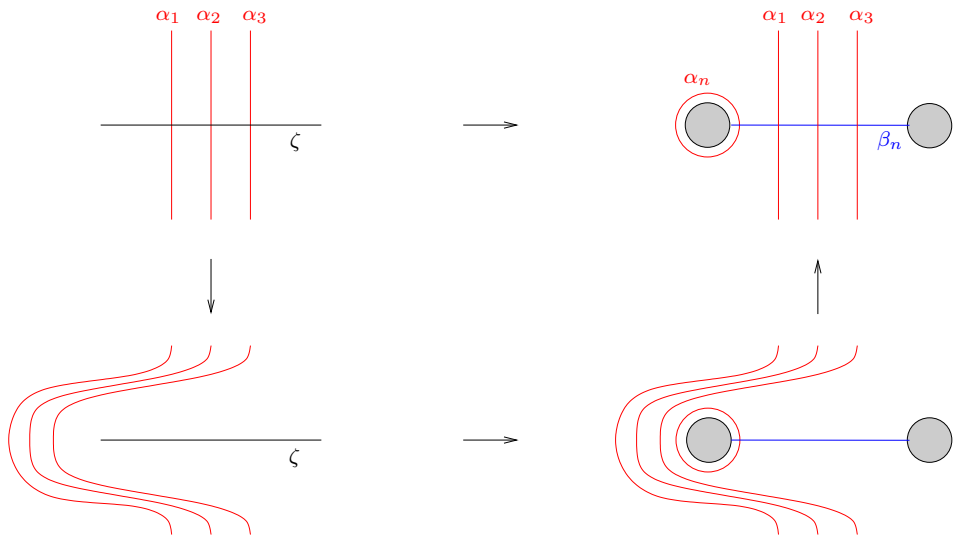
\caption {A compound $\alpha$-stabilization (top row), viewed as an the combination of an isotopy, an ordinary stabilization, and several $\alpha$-handleslides.}
\label{fig:compound}
\end {center}
\end {figure}

The reverse process to a compound $\alpha$-stabilization is called a compound $\alpha$-destabilization. We define compound $\beta$-stabilizations and destabilizations similarly, by switching the roles of the $\alpha$ and $\beta$ curves.

\subsection{Heegaard moves from surgery}
\label{sec:HMsurgery}
Recall that we have a doubly-pointed, triple Heegaard diagram of the form
$$ (\Sigma, \alphas, \gammas, \betas, w, z),$$ 
as shown in Figure~\ref{fig:triple}. Here, $(\Sigma, \alphas, \betas, w, z)$ represents the knot $K \subset Y$, and $(\Sigma, \alphas, \gammas, \betas, z)$ represents the two-handle cobordism $W'_p(K)$ from $Y_p(K)$ to $Y$.

We seek to find a sequence of Heegaard moves relating the knot diagrams $(-\Sigma, \betas, \alphas, z, w)$ to $(\Sigma, \alphas, \betas, w, z)$, such that they induce (in some way) a sequence of moves from $(-\Sigma, \gammas, \alphas, z)$ to $(\Sigma, \alphas, \gammas, z)$. The former  sequence will be used to define the map $\inv_0$ on $A_0^+$, and the latter to define the map $\inv$ on the Heegaard Floer complex of the surgery $Y_p(K)$. This will enable us to construct the chain homotopy \eqref{eq:Rps}.

First, notice that, in order for the roles of the $\alpha$ and $\beta$ curves to be more symmetric, it is helpful to consider a Heegaard diagram for $K \subset Y$ such that $w$ and $z$ are not only on each side of a $\beta$ curve, but also on each side of an $\alpha$ curve. This can be arranged by doing a compound $\alpha$-stabilization, as in Figure~\ref{fig:KnotStabilized}. (The same picture appeared in the proof of conjugation symmetry for knot Floer homology; see \cite[Figure 4]{Knots}.) Now $w$ and $z$ are on each side of the newly introduced curve $\alpha_{g+1}$. 

\begin {figure}
\begin {center}
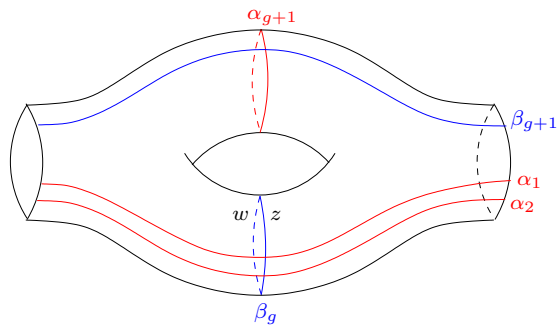
\caption {The diagram $(\Sigma, \alphas, \betas, w, z)$ from Figure~\ref{fig:triple}, after a compound $\alpha$-stabilization.}
\label{fig:KnotStabilized}
\end {center}
\end {figure}

From now on let us write
$$ H = (\Sigma, \alphas, \betas, w, z)$$
for the stabilized diagram in Figure~\ref{fig:KnotStabilized}. 

Let us also consider the diagram 
$$H'=(\Sigma', \alphas',  \betas', w', z)$$ shown on the right of Figure~\ref{fig:Base}. This is obtained from the (unstabilized) $\alpha$-$\beta$ diagram from Figure~\ref{fig:triple} by replacing $\beta_g$ with a knot longitude $\beta_g'$, as well as replacing $w$ with a new basepoint $w'$ on the back side of the cylinder. Thus, we have $$\alphas = \alphas' \cup \{\alpha_{g+1}\}, \ \ \ \betas= (\betas' \setminus \{\beta_g'\}) \cup \{\beta_g, \beta_{g+1}\}.$$

\begin {figure}
\begin {center}
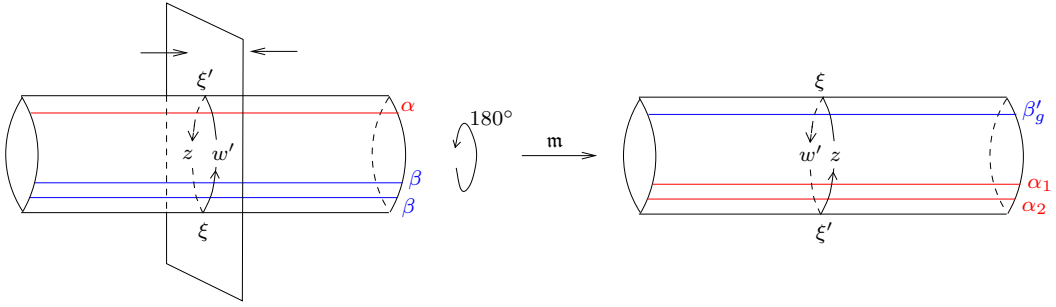
\caption {The arrow labeled $\mm$ indicates a base sequence of Heegaard moves going from the diagram $\bh'$ to $H'$. The alpha and beta curves are switched in the two diagrams, and so is the orientation of the surface. The moves include a diffeomorphism that reflects the shown cylinder into the center plane, and also rotates it by $180^\circ$ about the core axis.}
\label{fig:Base}
\end {center}
\end {figure}

In the diagram $H'$, we can connect $z$ to $w'$ by an arc $\xi$ that intersects $\beta'_g$, and we can connect $w'$ to $z$ by another arc $\xi'$ that intersects some alpha curves ($\alpha_1$ and $\alpha_2$ in our picture). The diagram $H'$ represents the meridian $\mu$ for $K$, viewed as a knot inside the longitudinal surgery $Y_{0}(K)$. We choose a suitable sequence of Heegaard moves $\mm$ that go from the diagram $\bh' = (-\Sigma', \betas', \alphas', z, w')$ to $H'$.  We call this a {\em base sequence of moves}. The base sequence is supposed to satisfy certain assumptions, which will be discussed shortly.

In general, if we have a Heegaard diagram representing a knot inside a three-manifold, we can connect the two basepoints by an arc in the complement of the alpha curves, and by another arc in the complement of the beta curves; we call their union a {\em trace of the knot} on the Heegaard diagram. In the case at hand, an example of a trace is the union $c=\xi \cup \xi'$. Note that the orientation of the knot induces an orientation on the trace.

The base sequence of moves $\mm$ consists of some stabilizations, destabilizations, curve isotopies, handleslides, and diffeomorphisms that come from an ambient isotopy of the Heegaard surface inside the three-manifold.

Without loss of generality, we can assume that the diagram $H'$ has already been stabilized as many times as needed, so that $\mm$ only consists of moves that do not change the surface (curve isotopies and handleslides), together with some diffeomorphism coming from an ambient isotopy; see for example \cite[Theorem 1.1, part 2]{LaudenbachCerf}. Note that the diffeomorphism must be orientation-reversing, because it takes $-\Sigma'$ to $\Sigma'$. Furthermore, it must map a trace of the knot on $-\Sigma'$ to a trace on $\Sigma'$, as oriented curves.

In fact, we can assume that the diffeomorphism takes the trace $c=\xi \cup \xi'$ to itself, preserving an annular neighborhood $A$ of that trace on the Heegaard surface. To arrange this, we choose a regular neighborhood $U$ of the knot $K$ (a solid torus) that also contains the curve $c$, as in Figure~\ref{fig:nbhd}. We then fix a self-indexing Morse function $f$ on $U$ with only two critical points, one of index $3$ and one of index $0$, both on the knot $K$, such that $K$ consists of two Morse trajectories between these critical points, and the level set $f^{-1}(3/2)$ intersects $U$ in an annular neighborhood $A$ of $c$. Let $h$ denote the self-diffeomorphism of $U$ given by a $180^\circ$ rotation along the longitude, followed by a $180^\circ$ rotation along the meridian. We can assume that $f=(3-f)\circ h$, and that $h$ preserves both the knot and the trace (and rotates them by $180^\circ$). Note that $h|_{\del U}$ is isotopic to the identity (being a composition of rotations), so we can extend $h$ to a diffeomorphism of all of $Y$ such that $h$ is the identity outside a slightly larger neighborhood $U' \supset U$. We now extend $f$ to a self-indexing Morse function on all of $Y$, with only index $1$ and $2$ critical points outside $U$. This produces the Heegaard diagram $H'$. Moreover, we interpolate between $f$ and $(3-f)\circ h$ on the complement of $U$, without changing the values of the function on $\del U$. We let $\mm$ be the sequence of Heegaard moves induced by this interpolation, via Cerf theory, combined with the diffeomorphism $h$ inside $U$.

\begin {figure}
\begin {center}
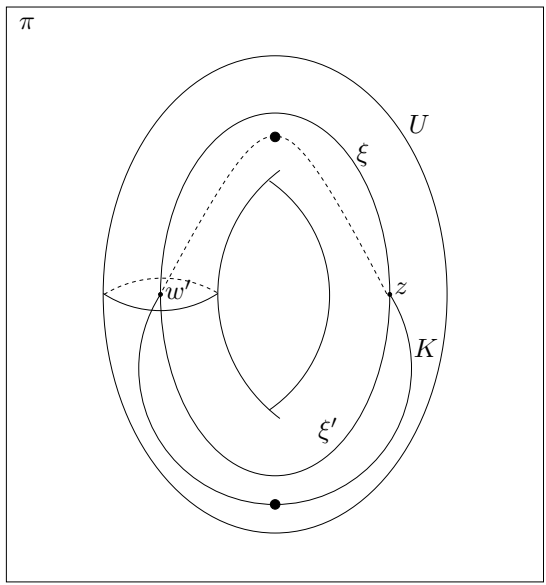
\caption {A neighborhood $U$ of the knot and a fixed trace. The two bigger black dots indicate the index $0$ and index $3$ critical points of the Morse function $f$. The plane $\pi$ cuts $U$ into two halves, along the level set $A=f^{-1}(3/2)$, which contains $c = \xi \cup \xi'$. Half of the knot is in front of $\pi$ and half is behind. 
}
\label{fig:nbhd}
\end {center}
\end {figure}

Note that the diffeomorphism $h$  takes a neighborhood of the trace (the cylinder $A$ shown on the left of Figure~\ref{fig:Base}) to the cylinder on the right.  Since the orientation of the trace is preserved but that of the surface is reversed, the two boundaries of the cylinder must be swapped. Moreover, the diffeomorphism swaps the basepoints $w'$ and $z$. In fact, from our construction we see that in the part of the diagram shown in Figure~\ref{fig:Base}, the diffeomorphism consists of a reflection into the plane of the trace, followed by a $180^\circ$ rotation about the core of the cylinder. 

The reason we started with a base sequence of moves $\mm$ as above is because it induces a well-behaved set of moves from $\bh$ to $H$. Indeed, we can view $H$ (the diagram in Figure~\ref{fig:KnotStabilized}) as obtained from $H'$ by adding a one-handle with feet near $z$ and $w'$, re-labeling $w'$ as $w$, re-labeling $\beta'_g$ as $\beta_{g+1}$, and introducing two new curves $\alpha_{g+1}$ and $\beta_{g}$. Let us say that $H$ is obtained from $H'$ by {\em drilling}, and so is $\bh$ from $\bh'$. With this in mind, observe that any Heegaard move (isotopy or handlelside) of the $\alpha$ curves that is part of the sequence $\mm$ induces a move of the corresponding $\alpha$ curves on the diagram obtained by drilling, without involving the new curve $\alpha_{g+1}$. Indeed, the original $\alpha$-moves are supposed to not cross the basepoints $w'$ and $z$, so we may as well assume that they do not cross the arc $\xi$. It is then clear that there are similar moves in Figure~\ref{fig:MovesKnot}. By the same token, we can assume that the original $\beta$-moves do not cross the arc $\xi'$, and therefore induce $\beta$-moves on the diagrams obtained by drilling, without involving $\beta_g$. Finally, the diffeomorphism $h$ that is part of the moves $\mm$ induces one on $\bh$ that is still given by reflection in the vertical plane, composed with a $180^\circ$ rotation along the horizontal axis in Figure~\ref{fig:MovesKnot}.

\begin {figure}
\begin {center}
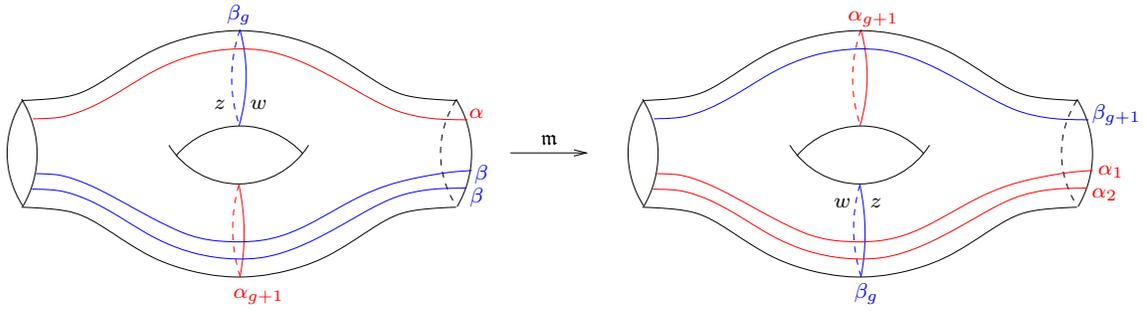
\caption {The base sequence of moves $\mm$ from Figure~\ref{fig:Base}, after drilling, induces a set of moves ending at the diagram $H$ from Figure~\ref{fig:KnotStabilized}. Here, $H$ is shown on the right hand side of the picture. The left hand side represents the conjugate diagram $\bh$ (modulo an isotopy that moves $z$ and $w$ lower in the diagram, on each side of the curve $\alpha_{g+1}$, and keeps all the curves fixed).  Note that, because of the $180^\circ$ rotation involved in $\mm$, the curves $\alpha_{g+1}$ and $\beta_g$ are interchanged in the two pictures. Also, the left vs. right position of $w$ and $z$ is switched because of the reflection in the center plane.}
\label{fig:MovesKnot}
\end {center}
\end {figure}

We have now identified a sequence of Heegaard moves that relate the knot diagrams $\bh$ and $H$.  We would like to have a related sequence of moves for the Heegaard Floer complexes of the surgery $Y_p(K)$. For this, we extend $H$ to a triple Heegaard diagram
$$ (\Sigma, \alphas, \gammas,  \betas, w, z),$$
where in the new set of curves $\gammas$ we exchanged $\beta_g$ for a longitude $\gamma_g$, twisted $p$ times around $\beta_g$. (Compare Figure~\ref{fig:triple}.) For our stabilized diagram $H$, the result of adding the curves $\gamma$ is shown on the right hand side of Figure~\ref{fig:BaseTriple}. Furthermore, we can follow the same sequence of moves as in Figure~\ref{fig:MovesKnot} (induced by $\mm$), with $\gamma$ curves instead of $\beta$ curves. This is possible because the moves do not involve $\beta_g$ (except for the diffeomorphism), and hence we may take them to keep $\gamma_g$ fixed as well. The result of the moves induced by $\mm$ on the triple Heegaard diagrams is shown in Figure~\ref{fig:BaseTriple}.

\begin {figure}
\begin {center}
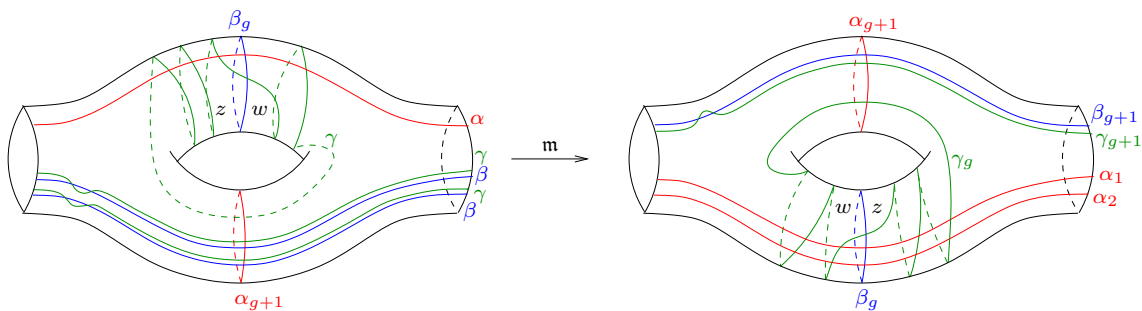
\caption {These are the moves from Figure~\ref{fig:MovesKnot}, with a third set of curves added. The triple diagrams represent the cobordism $W'_p(K)$.}
\label{fig:BaseTriple}
\end {center}
\end {figure}

The $\alpha$ and $\gamma$ curves on the right hand side of Figure~\ref{fig:BaseTriple} form a diagram representing the surgery $Y_p(K)$. However, the corresponding $\alpha$ and $\gamma$ curves on the left hand side do not give the conjugate diagram. In order to obtain a sequence of moves that can be used to define the involution on $\CFp(Y_p(K))$, we add another set of moves, which we call the {\em standard sequence}, shown in Figure~\ref{fig:Standard}. Let us ignore the $\delta$ curves from that figure for a moment. Then, the standard sequence starts at the diagram conjugate to $(\Sigma, \alphas, \gammas, w, z)$, and ends at the $\alpha$-$\gamma$ diagram on the right hand side of Figure~\ref{fig:BaseTriple}. The standard sequence consists of a handleslide of the $\alpha$ longitude over the $\alpha$ curve, followed by a compound $\alpha$-destabilization (removing the upper handle), followed by a compound $\alpha$-stabilization (introducing a new lower handle), followed by several $\gamma$-handleslides over the $\gamma$ curve in the middle, 
and finally a diffeomorphism (several Dehn twists that unfurl the $\alpha$ curve at the expense of furling the $\gamma$ curve).

\begin {figure}
\begin {center}
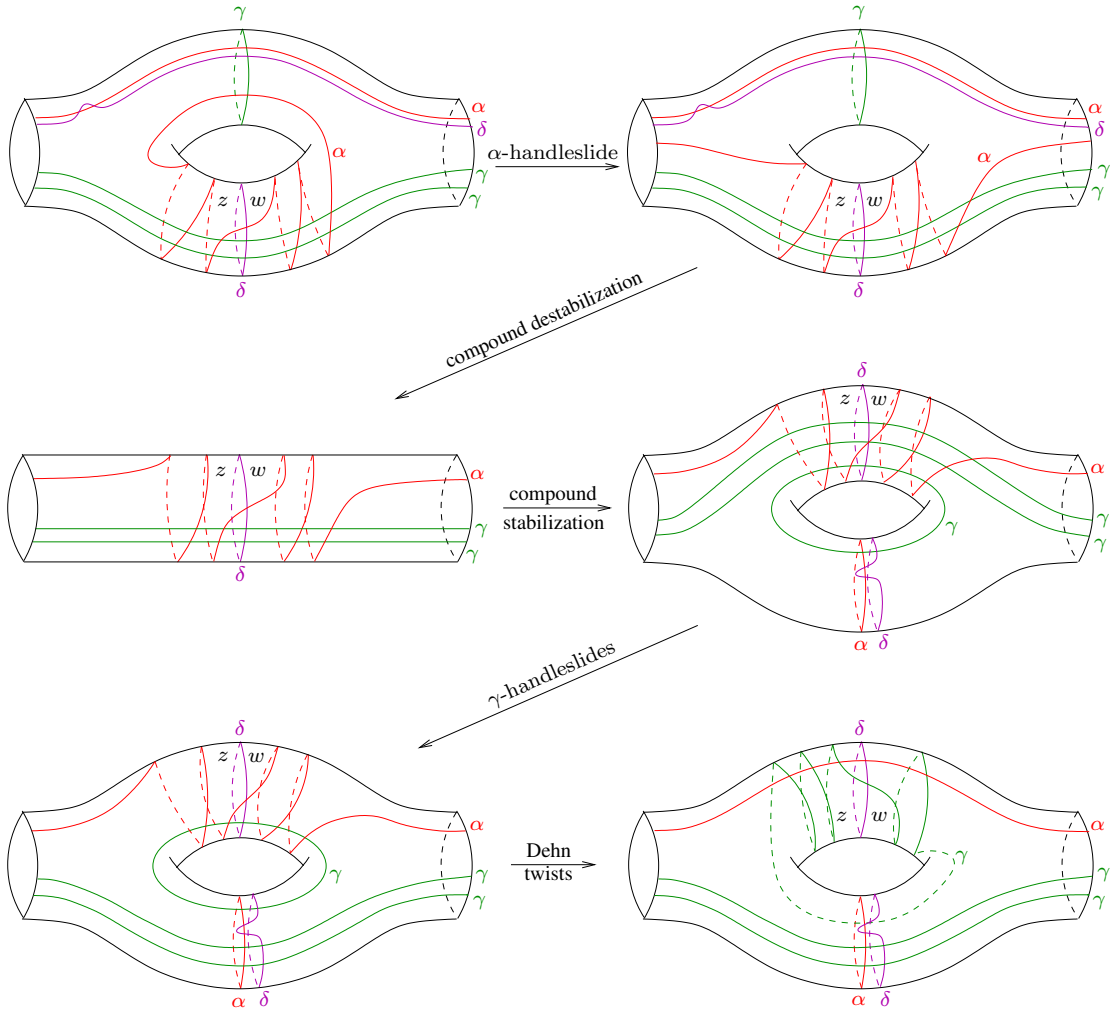
\caption {The standard sequence of moves on triple diagrams.}
\label{fig:Standard}
\end {center}
\end {figure}

If we do the standard sequence of moves, and follow it by the sequence induced by $\mm$ on the $\alpha$-$\gamma$ curves, we obtain a sequence of moves from $(-\Sigma, \gammas, \alphas, z)$ to $(\Sigma, \alphas, \gammas, z)$. This is the kind of sequence needed to define the involution on $\CFp(Y_p(K))$.

Thus, we have constructed sequences of moves for both the knot Floer complex (with the $\alpha$-$\beta$ curves) and for the Floer complex of the surgery (with the $\alpha$-$\gamma$ curves). Unfortunately, the two sequences are not of the same length: the former consists of the moves induced by $\mm$, whereas the latter involves both the standard sequence and the moves induced by $\mm$. In particular, we do not yet have a suitable set of moves for triple diagrams (with the $\alpha$, $\beta$, and $\gamma$ curves). 

To understand the moves on triple diagrams, let us first clarify what we mean by the conjugate of such a diagram. Normally, in Heegaard Floer theory, when we have a cobordism given by surgery on a link, we consider triple diagrams that are {\em right-subordinate} to that cobordism as in \cite[Definition 4.2]{HolDiskFour} (where the terminology is ``diagrams subordinate to a bouquet for the link''). For example, $(\Sigma, \alphas,\gammas,  \betas, z)$ is a triple diagram subordinate to the  cobordism $W'_p(K)$ from $Y_p(K)$ to $Y$. This induces a map from $\CFp(\Sigma, \alphas, \gammas, z)$ to $\CFp(\Sigma, \alphas, \betas, z)$. The name right-subordinate refers to the fact that, in these  two Floer complexes, the boundary conditions on the right of the pseudo-holomorphic disks are the same; in our case, they are given by the $\alpha$ curves.

In their proof of conjugation invariance for cobordism maps \cite[Section 5.2]{HolDiskFour}, Ozsv\'ath and Szab\'o introduced the notion of a triple Heegaard diagram that is {\em left-subordinate} to a surgery cobordism. This means that the boundary conditions on the left are being fixed. For example, if we have a triple diagram $(\Sigma, \alphas, \gammas, \betas, z)$ that is right-subordinate to a cobordism, its conjugate is the diagram $(-\Sigma, \alphas, \betas, \gammas, z)$, left-subordinate to the same cobordism. The conjugate induces a map from $\CFp(-\Sigma, \gammas, \alphas,z)$ to $\CFp(-\Sigma, \betas, \alphas, z)$.

In our situation, with $(\Sigma, \alphas, \gammas, \betas, z)$ being the diagram on the right hand side of Figure~\ref{fig:BaseTriple}, its conjugate is the diagram shown at the very beginning of the standard sequence from Figure~\ref{fig:Standard}. Indeed, notice that the $\alpha$ and $\gamma$ curves are interchanged, and the $\beta$ curves are re-labelled as $\delta$. Let us add the new set of curves $\deltas$ to all the diagrams involved in the standard sequence, following the respective moves. Throughout Figure~\ref{fig:Standard}, the $\delta$-$\gamma$ diagrams (with the basepoints $z$ and $w$) represent the knot $K \subset Y$. The triple diagrams $(-\Sigma, \gammas, \deltas, \alphas, z)$ are left-subordinate to the cobordism $W'_p(K)$; such left-subordinate diagrams induce cobordism maps from $\CFp(-\Sigma, \alphas, \gammas, z)=\CFp(Y_p(K))$ to $\CFp(-\Sigma, \deltas, \gammas, z)=\CFp(Y)$, by counting $\gamma$-$\delta$-$\alpha$ triangles. 

Notice that we cannot go from the conjugate diagram $(-\Sigma, \alphas, \betas, \gammas, z)$ to $(\Sigma, \alphas, \gammas, \betas, z)$ by a sequence of usual Heegaard moves, because these cannot transform a left-subordinate diagram into a right-subordinate one. Instead, the best we can do is the following:
\begin{enumerate}[(i)]
\item Go through the standard sequence of moves in Figure~\ref{fig:Standard};
\item In the last diagram from Figure~\ref{fig:Standard}, replace the delta curves with beta curves as on the right hand side of Figure~\ref{fig:BaseTriple}. This changes a left-subordinate diagram into a right-subordinate one;
\item Finally, go through the moves induced by $\mm$, as in Figure~\ref{fig:BaseTriple}. 
\end{enumerate}

This is the process we will use to define the chain homotopy \eqref{eq:Rps}. Step (ii) in the process may seem abrupt, because, for example, it involves passing directly between two (seemingly unrelated) diagrams for $K \subset Y$: the $\delta$-$\gamma$ and the $\alpha$-$\beta$ diagrams. Nevertheless, a step of this kind seems unavoidable, and was also considered in the proof of conjugation invariance for cobordism maps \cite[Section 5.2]{HolDiskFour}. The quadruple Heegaard diagram that appears in Step (ii) is shown in Figure~\ref{fig:Quadruple}.

\begin {figure}
\begin {center}
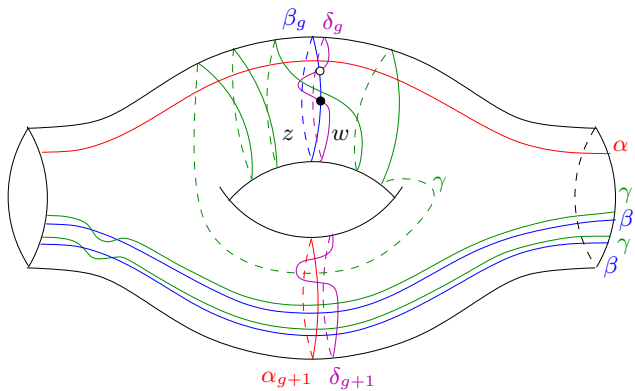
\caption {A quadruple Heegaard diagram.}
\label{fig:Quadruple}
\end {center}
\end {figure}

\subsection{Proof of the involutive large surgery formula}
\label{sec:Proof}

We are now ready to complete the proof of Theorem~\ref{thm:Large} from the Introduction. By the discussion at the end of Section~\ref{sec:large}, our task is to construct the homotopy \eqref{eq:Rps} between $\iota_0 \circ \Gamma^+_{p,0}$ and $\Gamma^+_{p,0} \circ \iota$.

We use the moves on triple Heegaard diagrams described near the end of Section~\ref{sec:HMsurgery}, sequences (i)-(iii). At each step in these sequences, we have a triple Heegaard diagram that is either left- or right-subordinate to the cobordism $W'_p(K)$. Further, we have not one but two basepoints, so by dropping one set of curves we obtain a Heegaard diagram $H^{(k)}$ for the knot $K \subset Y$, for $k=0,1, \dots, m$. Let $H_p^{(k)}$ be the corresponding Heegaard diagrams for the surgery $Y_p(K)$. Let us add suitable almost complex structures, to obtain Heegaard data $\H^{(k)}$ for $K \subset Y$ and $\H_p^{(k)}$ for $Y_p(K)$. For each $i$, we have a chain map as in \eqref{eq:Xi}, with $s=0$:
$$\Gamma^{+, (i)}_{p,0} \co \CFp(\H_p^{(k)}, [0]) \to A_0^+(\H^{(k)}).$$
Recall from \eqref{eq:XiFormula} that these maps are obtained by counting rigid pseudo-holomorphic triangles in homotopy classes $\psi$ with $n_w(\psi) = n_z(\psi)$.

The moves on Heegaard diagrams produce chain homotopy equivalences
$$ \Phi(\H_p^{(k-1)}, \H_p^{(k)}) \co \CFp(\H_p^{(k-1)}, [0]) \to  \CFp(\H_p^{(k)}, [0])$$
and 
$$\Phi(\H^{(k-1)}, \H^{(k)}) \co A_0^+(\H^{(k-1)}) \to  A_0^+(\H^{(k)})$$
for $k=1, \dots, m.$ In the case of $\Phi(\H^{(k-1)}, \H^{(k)})$, there is a special situation for a certain value of $k$, corresponding to the abrupt Step (ii) mentioned at the end of Section~\ref{sec:HMsurgery}. Nevertheless, in this case we can still construct an equivalence $\Phi(\H^{(k-1)}, \H^{(k)})$; see Figure~\ref{fig:ChainHomotopies} and the discussion of Step (ii) at the end of this section.

Observe that $\H_p^{(0)}$ (the $\alpha$-$\gamma$ diagram at the beginning of Figure~\ref{fig:Standard}) is conjugate to $\H_p^{(m)}$ (the $\alpha$-$\gamma$ diagram at the end of Figure~\ref{fig:BaseTriple}). Moreover, the composition
$$ \Phi(\H_p^{(0)}, \H_p^{(m)}) := \Phi(\H_p^{(m-1)}, \H_p^{(m)}) \circ \dots \circ \Phi(\H_p^{(0)}, \H_p^{(1)})\co \CFp(\H_p^{(0)}, [0])\to \CFp(\H_p^{(m)}, [0]),$$
together with the canonical identification $\eta$ of the complexes $\CFp(\H_p^{(0)}, [0])$ and $\CFp(\H_p^{(m)}, [0])$, produces the map $\iota \co \CFp(\H_p^{(m)}, [0]) \to  \CFp(\H_p^{(m)}, [0]).$

Similarly, the knot diagram $\H^{(0)}$ (the $\delta$-$\gamma$ diagram at the beginning of Figure~\ref{fig:Standard}) is conjugate to $\H^{(m)}$ (the $\alpha$-$\beta$ diagram at the end of Figure~\ref{fig:BaseTriple}). The composition
$$  \Phi(\H^{(0)}, \H^{(m)}) := \Phi(\H^{(m-1)}, \H^{(m)}) \circ \dots \circ \Phi(\H^{(0)}, \H^{(1)})\co A_0^+(\H^{(0)})\to A_0^+(\H^{(m)}),$$
together with the identification $ A_0^+(\H^{(0)})\cong A_0^+(\H^{(m)})$, gives the map $\iota_0 \co A_0^+(\H^{(m)}) \to A_0^+(\H^{(m)})$.

Consider the sequence of diagrams
\[
\xymatrixcolsep{5pc}
\xymatrix{
\CFp(\H_p^{(0)}, [0]) \ar[r]^{\phantom{uuu}\Gamma^{+, (0)}_{p,0}} \ar[d]_{\Phi(\H_p^{(0)}, \H_p^{(1)})} \ar@{-->}[dr]^{R^{+, (1)}_{p}} & A_0^+(\H^{(0)})  \ar[d]^{\Phi(\H^{(0)}, \H^{(1)})}\\
\dots  \ar[d] & \dots  \ar[d] \\
\CFp(\H_p^{(k-1)}, [0]) \ar[r]^{\phantom{uuu}\Gamma^{+, (k-1)}_{p,0}} \ar[d]_-{\Phi(\H_p^{(k-1)}, \H_p^{(k)})} \ar@{-->}[dr]^{R^{+, (k)}_{p}} & A_0^+(\H^{(k-1)})  \ar[d]^-{\Phi(\H^{(k-1)}, \H^{(k)})}\\
\CFp(\H_p^{(k)}, [0]) \ar[r]^{\phantom{uuu}\Gamma^{+, (k)}_{p,0}} \ar[d] & A_0^+(\H^{(k)})  \ar[d] \\
\dots  \ar[d]_-{\Phi(\H_p^{(m-1)}, \H_p^{(m)})} \ar@{-->}[dr]^{R^{+, (m)}_{p}} & \dots  \ar[d]^-{\Phi(\H^{(m-1)}, \H^{(m)})} \\
\CFp(\H_p^{(m)}, [0]) \ar[r]_{\phantom{uuu}\Gamma^{+, (m)}_{p,0}} & A_0^+(\H^{(m)}).
}
\]

We claim that, at each step, we can construct chain homotopies $R^{+, (k)}_{p}$ between $\Phi(\H^{(k-1)}, \H^{(k)}) \circ \Gamma^{+, (k-1)}_{p,0}$ and $\Gamma^{+, (k)}_{p,0} \circ \Phi(\H_p^{(k-1)}, \H_p^{(k)})$. If so, this would imply that we have chain homotopies 
$$\Phi(\H^{(0)}, \H^{(m)}) \circ \Gamma^{+, (0)}_{p,0} \ \sim \ \Gamma^{+, (m)}_{p,0} \circ \Phi(\H_p^{(0)}, \H_p^{(m)})$$ and hence $$\iota_0 \circ \Gamma^{+, (m)}_{p,0} \ \sim \ \Gamma^{+, (m)}_{p,0} \circ \iota,$$ as desired.

The construction of the homotopies $R^{+, (k)}_{p}$ proceeds by standard arguments in Heegaard Floer theory. Typically, the maps of the form $\Phi(\H^{(k-1)}, \H^{(k)})$, $\Phi(\H_p^{(k-1)}, \H_p^{(k)})$ and $\Gamma^{+, (k)}_{p,0}$ are given by counts of pseudo-holomorphic triangles with $n_w(\psi) = n_z(\psi)$. We can then define $R^{+, (k)}_{p}$ by counting suitable pseudo-holomorphic quadrilaterals, again in homotopy classes $\psi$ with $n_w(\psi) = n_z(\psi)$.

A special situation appears at Step (ii), illustrated in Figure~\ref{fig:Quadruple}. In this case, the diagram that represents $Y_p(K)$ is unchanged (it is the $\alpha$-$\gamma$ diagram), but the diagram that represents $K \subset Y$ changes directly from the $\delta$-$\gamma$ diagram to the $\alpha$-$\beta$ diagram. We then construct the corresponding chain homotopy $R^{+, (k)}_{p}$ following Figure~\ref{fig:ChainHomotopies}. (This is inspired by \cite[Section 5.2]{HolDiskFour}.) Observe that the $\beta$-$\delta$ diagram represents the knot $K$ in $Y \# (S^1 \times S^2)$, and the associated $A_0^+$ complex consists of two copies of the $A_0^+$ complex for $K \subset Y$. These two copies, $A_0^+(\bullet)$ and $A_0^+(\circ)$ can be distinguished by whether their generators contain the black dot or the white dot from the intersection $\beta_g \cap \delta_g$ in Figure~\ref{fig:Quadruple}. Each copy is isomorphic to the $A_0^+$ complex formed using the $\beta$ and $\delta$ curves in the destabilized diagram obtained by deleting $\beta_g$, $\delta_g$, and the handle. Furthermore, as a corollary to Proposition~\ref{prop:JT2}, we have that all $A_0^+$ complexes coming from different diagrams for $K \subset Y$ are chain homotopy equivalent, by canonical equivalences (up to chain homotopy). Thus, in particular, we have canonical equivalences
\begin{equation}
\label{eq:equiv}
 A_0^+(\Sigma, \alphas, \betas, w, z) \simeq A_0^+(\bullet) \simeq A_0^+(\Sigma, \deltas, \gammas, w, z),
 \end{equation}
which are indicated by the long diagonal arrows on the right of Figure~\ref{fig:ChainHomotopies}.

The $\delta$-$\gamma$-$\beta$ and $\beta$-$\delta$-$\alpha$ triangle maps in Figure~\ref{fig:ChainHomotopies} are right-, resp. left-subordinate to the two-handle cobordism from $Y$ to $Y\# (S^1 \times S^2)$, and therefore they are chain homotopic to the inclusion 
$$A_0^+ \simeq A_0^+(\bullet) \hookrightarrow A_0^+(\bullet) \oplus A_0^+(\circ).$$

\begin {figure}
\begin {center}
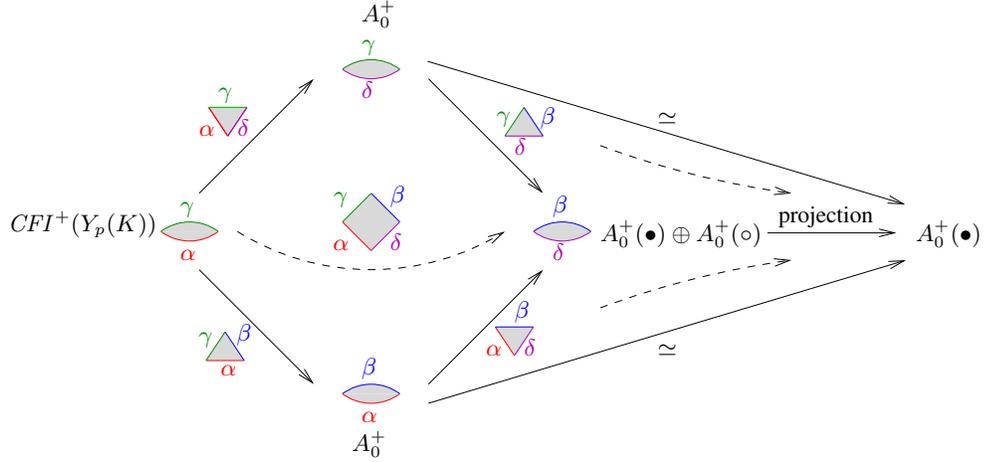
\caption {Chain homotopies in Step (ii), for the maps coming from Figure~\ref{fig:Quadruple}.}
\label{fig:ChainHomotopies}
\end {center}
\end {figure}

This shows that the post-composition of either of these triangle maps with the projection to $A_0^+(\bullet)$ is chain homotopic to the respective canonical equivalence from \eqref{eq:equiv}. (We can think of the composition as adding a two-handle, followed by a cancelling three-handle.) The corresponding chain homotopies are shown as dashed arcs towards  the right of Figure~\ref{fig:ChainHomotopies}; let us denote them by $R_{\delta\gamma\beta}$ and $R_{\beta\delta\alpha}$, respectively. There is yet another chain homotopy in Figure~\ref{fig:ChainHomotopies}, represented by the dashed arc on the left, and given by counting  holomorphic $\alpha$-$\gamma$-$\beta$-$\delta$ quadrilaterals, with $n_z(\psi) = n_w(\psi)$. We denote this chain homotopy by $R_{\alpha\gamma\beta\delta}.$

We now define the chain homotopy $R^{+, (k)}_{p}$ between the $\gamma$-$\delta$-$\alpha$ and $\alpha$-$\gamma$-$\beta$ triangle maps as the sum of three maps, following Figure~\ref{fig:ChainHomotopies}. Specifically, the three maps are the composition of the $\gamma$-$\delta$-$\alpha$ triangle map with $R_{\delta\gamma\beta}$, the composition of $R_{\alpha\gamma\beta\delta}$ with the projection to $A_0^+(\bullet)$, and the composition on the $\alpha$-$\gamma$-$\beta$ triangle map with $R_{\beta\delta\alpha}$.

This finishes the proof of Theorem~\ref{thm:Large}.

\subsection{Involutive correction terms for large surgeries} 
\label{sec:dSurgery}
Let us now specialize to knots in $S^3$. Given Heegaard data $\H$ for $K \subset S^3$, let 
$$B^+ = C\{i \geq 0\} = \CFp(\Sigma, \alphas, \betas, w) \subset \CFKi(\H)$$ and consider the natural projection
$$ v_0^+ \co A_0^+ \to B^+.$$

Since the kernel of $v_0^+$ is finite dimensional and $H_*(B^+) \cong \HFp(S^3) = \T^+$, we see that the map induced on homology, $(v_0^+)_*$, must be surjective. Following the notation from \cite{RatSurg}, we let $V_0=V_0(K)$ be the nonnegative integer such that, for $n \gg 0$,  the map  on $U$-towers
$$(v_0^{+})_*|_{U^nH_*(A_0)} \co U^nH_*(A_0^+) \rightarrow H_*(B^+)$$
is given by multiplication by $U^{V_0}$.

Let $p \geq g(K)$. Under the identifications $H_*(A_0^+) \cong \HFp(Y_p(K), [0])$ from Theorem~\ref{thm:LargeSurg} and $H_*(B^+) \cong \HFp(S^3),$ the map $(v_0^+)_*$ is the one associated to the cobordism $W'_p(K)$ (and the zero spin structure). From the formula~\eqref{eq:GradingShift} for the grading shift, it follows that
\begin{equation}
\label{eq:dsurg}
 d(S^3_p(K), [0]) = \frac{p-1}{4} - 2V_0(K).
 \end{equation}

The quantity $V_0(K)$ can also be identified with $-d(S^3_1(K))/2$, and is an invariant under smooth knot concordance. It was studied in \cite{RasmussenThesis, RasmussenGT, RatSurg, Peters}. In \cite[Definition 7.1]{RasmussenThesis}, it appears in the form of Rasmussen's local $h$-invariant $h_0(K)=V_0(K)$. Different conventions are used in \cite{RasmussenGT}; with $h_0$ as defined there, we have $V_0(K)=h_0({K}^*)$, where ${K}^*$ is the mirror of $K$.

Now we define analogous invariants using involutive Floer complexes.  In view of Equation~\eqref{eq:dsurg}, it is natural to set
\begin{equation}
\label{eq:dlsurg}
\Vl_0(K)= \frac{1}{2}\Bigl(\frac{p-1}{4} - \dl(S^3_p(K), [0]) \Bigr)
\end{equation}
and
\begin{equation}
\label{eq:dusurg}
\Vu_0(K)= \frac{1}{2}\Bigl(\frac{p-1}{4} - \du(S^3_p(K), [0]) \Bigr). 
\end{equation}

Then, Theorem~\ref{thm:dLarge} is a tautology. However, there is something to prove, namely that $\Vl_0=\Vl_0(K)$ and $\Vu_0=\Vu_0(K)$ are independent of $p \geq g(K)$. To see this, observe that the isomorphism from Theorem~\ref{thm:Large} fits into a commutative diagram of long exact sequences
\[
\xymatrix{
\dots \ar[r] & Q \ccdot \HFp(Y_p, [0]) [-1] \ar[d]_{\cong} \ar[r] & \HFIp(Y_p, [0]) \ar[d]_{\cong} \ar[r] & \HFp(Y_p, [0]) [-1] \ar[d]_{\cong} \ar[r] & \dots \\
\dots \ar[r] & Q \ccdot H_*(A_0^+) [-1] \ar[r] & H_*(\AI_0^+) \ar[r] & H_*(A_0^+) [-1] \ar[r] & \dots 
}
\]
Thus, we can compute $\Vl_0$ and $\Vu_0$ by looking at the two infinite $U$-towers in $H_*(\AI_0^+)$, and comparing the minimal gradings in each tower to the minimal grading of the $U$-tower in $H_*(A_0^+)$. Precisely, let $x$ be a generator in the minimal grading of the tower in $H_*(A_0^+)$, and let $\xl$ and $Q \cdot \xu$ be generators in the minimal gradings of the first and second towers in $H_*(\AI_0^+)$, respectively. Then
$$ \Vl_0 = V_0 + \frac{1}{2} (\gr(x) - \gr(\xl)), \ \ \ \Vu_0 = V_0 + \frac{1}{2} (\gr(x) - \gr(\xu)).$$
These formulae do not involve $p$. Furthermore, this is how we will compute $\Vl_0$ and $\Vu_0$ in practice.

It follows from \eqref{eq:dlsurg}, \eqref{eq:dusurg} and Proposition \ref{propn:basicproperties} that
$$\Vu_0 \leq V_0 \leq \Vl_0.$$
Since $V_0$ is nonnegative, so is $\Vl_0$. However, we shall see in examples that $\Vu_0$ can be negative.

\begin{proposition}
$\Vl_0$ and $\Vu_0$ are invariants of smooth knot concordance.
\end{proposition}

\begin{proof}
By Proposition~\ref{prop:ddTheta}, $\dl$ and $\du$ are invariant under (smooth) homology cobordisms. Since a concordance of knots induces a homology cobordism between the $p$-surgeries, we obtain the result. Compare \cite[Proposition 2.1]{Peters}. 
\end{proof}

\subsection{An example: $+1$ surgery on the trefoil}
\label{subsec:lht} As an example, we compute the involutive Heegaard Floer homology of $S^3_{1}(T_{\ell})=-\Sigma(2,3,7)$, where $T_{\ell}$ is the left-handed trefoil. Note that $T_{\ell}$ has genus $1$, so this is a large surgery. 

A picture of $\CFKinfty(T_{\ell})$ appears in Figure~\ref{fig:lht}. Notice that it splits into the direct sum of complexes $U^k C, k \in \Z$, where the complex $C$ consists of three elements $x_0, x_1^1, x_1^2$ with $\del x_1^1=\del x_1^2 = x_0$. (This notation, while slightly cumbersome in such a small complex, has been chosen to match our notation for a more general case, which appears in Section~\ref{sec:Lspace} and \ref{sec:thin}.) Here, $x_0$ is in homological degree $1$, and $x_1^1$ and $x_1^2$ are in degree $2$. Therefore $$\HFp(-\Sigma(2,3,7)) \cong H_*(A_0^+) \cong \Tower_{0} \oplus (\Z_2)_{(0)},$$ where we use the standard notation in Heegaard Floer homology: $\Tower_j$ denotes an infinite tower $\Z_2[U, U^{-1}]/\Z_2[U]$ with the lowest element in degree $j$, and $(\Z_2)_{(j)}$ denotes a copy of $\Z_2$ in degree $j$ with trivial $U$ action. In our case, the lowest-degree element in the tower is represented by $[U(x^1_1+x_1^2)]$, and the additional element is represented by $[Ux^1_1]$. 

\begin {figure}
\begin {center}
\includegraphics{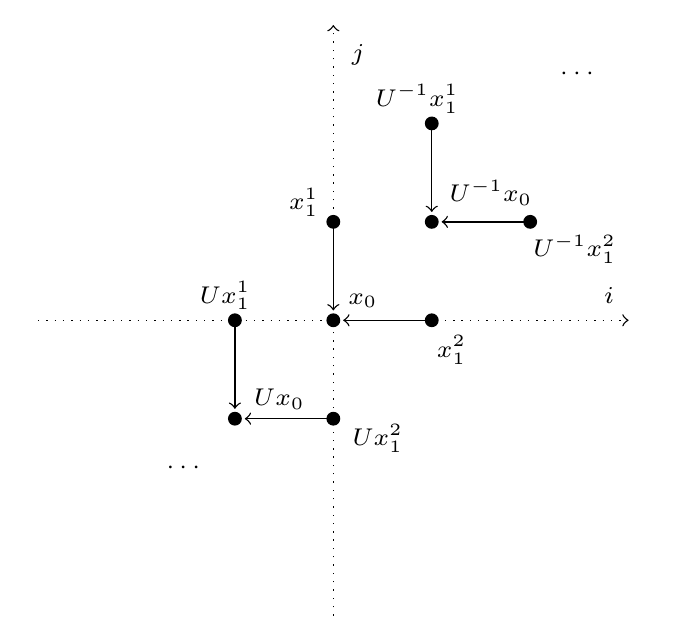}
\caption {The knot Floer complex of the left-handed trefoil.}
\label{fig:lht}
\end {center}
\end {figure}

\begin {figure}
\begin {center}
\includegraphics[scale=.8]{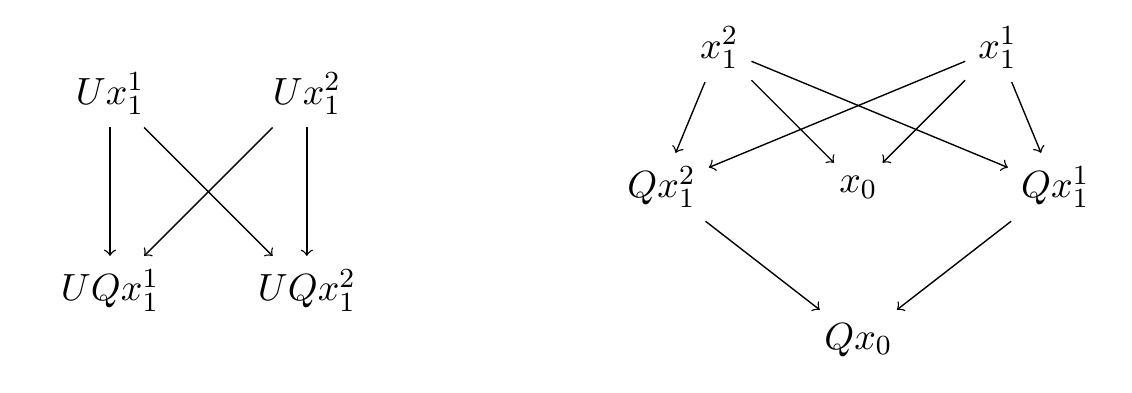}
\caption {Calculation of $+1$ surgery on the left-handed trefoil.}
\label {fig:lhtcomplex}
\end {center}
\end {figure}

Note that the map $\inv_K$ on $\CFKinfty(T_{\ell})$ is completely determined by its behavior with respect to the Maslov grading and the $\Z \oplus \Z$ filtration. Specifically, $\inv_K$ preserves $x_0$ and swaps $x^1_1$ with $x^2_1$.

We now compute the involutive homology $H_*(\AI_0^+)$. The chain complex 
$$\AI_0^+ = (A_0^+[-1] \oplus Q \ccdot A_0^+[-1], \del + Q \ccdot (1+\inv_0))$$ 
breaks into a direct sum of subcomplexes, two of which are pictured in Figure~\ref{fig:lhtcomplex}. The first has four elements; the second has six elements, and all subsequent complexes (not pictured) are copies of the second multiplied by $U^{-k}$ for $k>0$. We see that the homology of the four element complex is $$\Z_2\langle [U(x_1^1+x_1^2)], [UQx_1^1]=[UQx_1^2] \rangle$$ and the homology of the six element complex is $$\Z_2 \langle [x_1^1+x_1^2], [Q(x_1^1+x_1^2)]=[x_0] \rangle.$$ Note that $U[x_1^1+x_1^2]=[U(x_1^1+x_1^2)]$, whereas $U[Q(x_1^1+x_1^2)] = [UQx_1^1]+[UQx_1^2]=0$. Therefore we see that 
$$\HFIp(-\Sigma(2,3,7)) \cong \Tower_{1} \oplus \Tower_{2} \oplus (\Z_2)_{(0)}$$ where the lowest-degree elements of the towers can be taken to be $[U(x_1^1+x_1^2)]$ and $[Q(x_1^1+x_1^2)]$. Here the degrees are fixed by comparison with $\HFp(-\Sigma(2,3,7))$. We conclude that 
$$\dl(-\Sigma(2,3,7)) = 1-1=0, \ \ \ \du(-\Sigma(2,3,7)) = 2.$$ 
Using the properties of $\dl$ and $\du$ under orientation reversal (Proposition~\ref{prop:reversal}), we obtain the calculations for $\Sigma(2,3,7)$ stated in the Introduction. Combining these calculations with Corollary~\ref{corollary:lspace}, we get Corollary~\ref{cor:Lcor}.

With regard to the smooth concordance invariants defined in Section~\ref{sec:dSurgery}, we have $V_0(T_{\ell})= \Vl_0(T_{\ell})=0$ and $\Vu(T_{\ell})=-1.$

\subsection{Failure of additivity}
\label{sec:nonadditive}
We are now ready to give an example showing that the maps 
$$\dl, \du \co \Theta^3_{\Z_2} \to \Q$$
are not homomorphisms. Take $Y=\Sigma(2,3,7)$ and consider the connected sum $Y \# Y$. We have computed that $\dl(Y)=-2$. If $\dl$ were additive, then we would get $\dl(Y \# Y)=-4.$ On the other hand, from the calculation of $\HFp(Y)$ in \cite[p.47]{AbsGraded} and the connected sum formula \cite[Theorem 1.5]{HolDiskTwo} we get
$$\HFp(Y \# Y)\cong \T^+_0 \oplus (\Z_2)^3_{(-1)} \oplus (\Z_2)_{-2}.$$

From the exact triangle \eqref{pic:exact} relating $\HFIp(Y \# Y)$ to $\HFp(Y \# Y)$ we see that $\HFIp(Y \# Y)$ cannot have any elements in degrees less than $-2$. Hence $\dl(Y \# Y)$ cannot be $-4$, and $\dl$ is not additive. 

By changing orientation, in view of Proposition~\ref{prop:reversal}, we see that $\du$ is not additive either.

\section{L-space knots and their mirrors}
\label{sec:Lspace}

A knot $K \subset S^3$ is called an {\em L-space knot} if the surgery $S^3_p(K)$ is an L-space for some integer $p > 0$ (and hence for all $p \gg 0$). In \cite{OSLens}, Ozsv\'ath and Szab\'o  proved that if $K$ is an $L$-space knot, then the Alexander polynomial of $K$ is of the form
$$ \Delta_K(t) = (-1)^m + \sum_{i=1}^m (-1)^{m-i}(t^{n_i} + t^{-n_i})$$
for a sequence of positive integers $0 < n_1 < n_2 < \dots < n_m.$ Here, $n_m = g(K)$ is the genus of $K$. Let $n(K) \geq 0$ be the quantity
$$ n(K):= n_m - n_{m-1} + \dots + (-1)^{m-2} n_2 + (-1)^{m-1} n_1.$$ Furthermore, let $\ell_s = n_s - n_{s-1}$.

Ozsv\'ath and Szab\'o also proved that the knot Floer complex $\CFKinfty(K)$ is completely determined by this information. Before giving their description, it is helpful to introduce the notion of a model complex for $\CFKinfty(K)$.

\begin{definition} Given a knot $K$, we say that $C$ is a \emph{model complex} for $\CFKinfty (K)$ if $\CFKinfty(K) \simeq C \otimes \Z_2[U^{-1}, U]$, where $\Z_2[U^{-1}, U]$ is regarded as a chain complex with trivial differentials. \end{definition}

For example, the complex $C$ from Section~\ref{subsec:lht} is a model complex for $\CFKinfty(T_{\ell})$.

\begin{remark} There is an important subtlety to note at this juncture. To say that $C$ is a model complex for $\CFKinfty$ implies that it is preserved by $\del$, but not necessarily by $\inv_K$. This will not be relevant to the proofs in this section, but in the future computations in Section~\ref{sec:thin} it will frequently be the case that only $C \otimes \Z_2[U^{-1},U]$ is preserved by $\inv_K$. \end{remark}

If $K$ is an $L$-space knot, then $\CFKinfty(K)$ has a model $\Z_2$-complex $C$ with generators $x_0, x_1^1,x_1^2, \cdots, x_m^1,x_m^2$, as follows. The $(i, j)$-grading of $x_0$ is $(0,0)$, the grading of $x_m^1$ is $(-n(K), g(K)-n(K))$, and in general the gradings of $x^t_{m-2s}$ and $x^t_{m-(2s+1)}$ differ only in the $i$-grading by $\ell_{m-2s}$ and the gradings of $x^t_{m-(2s+1)}$ and $x^t_{m-(2s+2)}$ differ only in the $j$-grading by $\ell_{m-(2s+1)}$. Moreover the complex is symmetric: if $\gr(x_s^1)=(i,j)$, then $\gr(x_s^2)=(j,i)$. There is no need to keep track of the homological grading.

The differentials in the complex $C$ form a ``staircase'' as shown in Figure~\ref{fig:LComplex}. If $m$ is odd, then the nonzero differentials are
\begin{align*}
\del(x_0) &= x_1^1+x_1^2 & \del(x_s^t) &= x_{s-1}^t+x_{s+1}^t \text{ for } s>0 \text{ even},\ t \in \{1,2\}
\end{align*}
\noindent whereas if $m$ is even, the nonzero differentials are
\begin{align*}
\del(x_1^t) &= x_0 + x_2^t & \del(x_s^t) &= x_{s-1}^t+x_{s+1}^t \text{ for } s>1 \text{ odd},\ t \in \{1,2\}.
\end{align*}

\begin{figure}
\includegraphics[scale=.9]{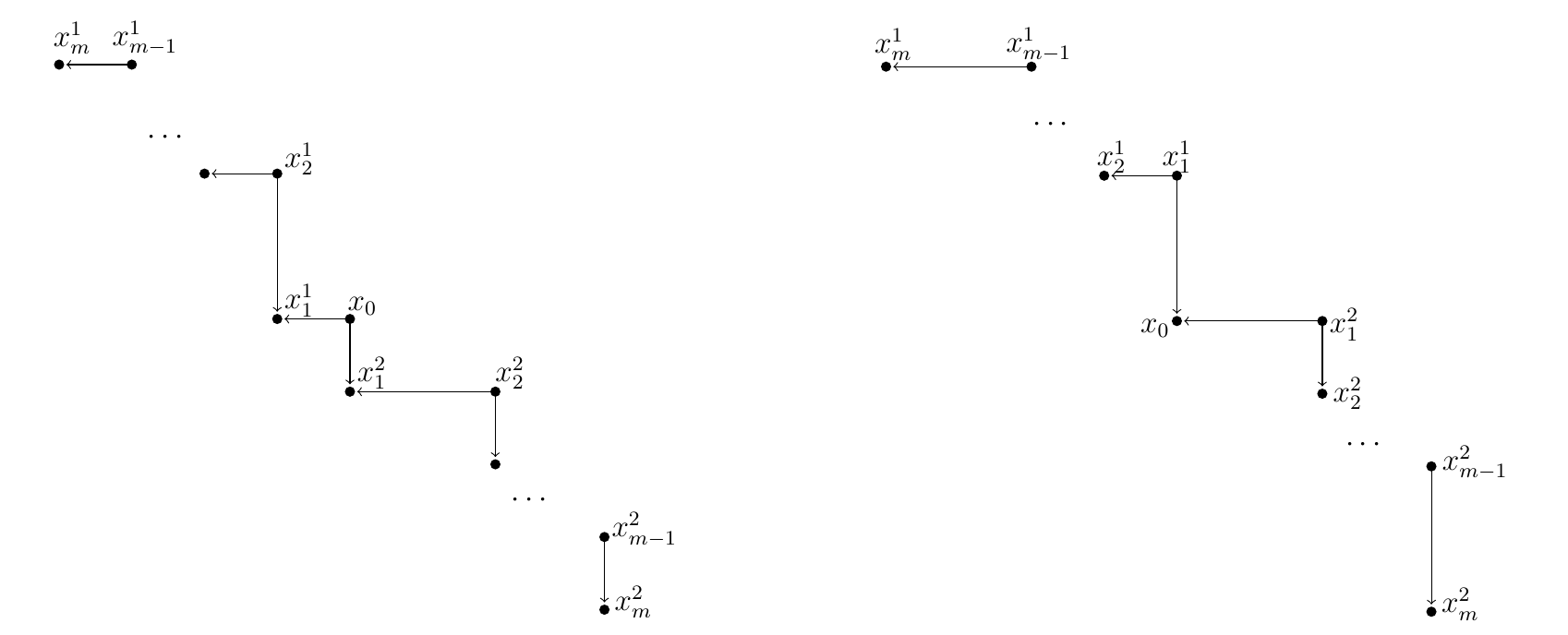}
\caption{Model complexes for the knot Floer complex of an $L$-space knot. The case $m$ odd is on the left and the case $m$ even is on the right.}
\label{fig:LComplex}
\end{figure}

The quantity $n(K)$ is the total length of the horizontal arrows in the top half of the complex $C$.

From the above description of $\CFKinfty(K)$ we see that $H_*(A_0^+)\simeq\Tower$ with lowest-degree element represented by $[x_m^1]=[x_m^2]$, and $H_*(B_0^+)\simeq \Tower$ with lowest-degree element represented by $[U^{-n(K)}x_m^1]=[U^{-n(K)}x_m^2]$. We deduce that 
$$V_0=n(K).$$

By definition, large surgeries on L-space knots are L-spaces. By  Corollary~\ref{corollary:lspace}, their involutive Heegaard Floer homology is determined by their usual Heegaard Floer homology, and in particular we have
$$ \Vl_0(K)=\Vu_0(K)= V_0(K)=n(K).$$

It is more interesting to consider mirrors of L-space knots. These have the property that sufficiently negative surgeries along them yield L-spaces. An example of such a knot is the left-handed trefoil $T_{\ell}$ from Section~\ref{subsec:lht}.

If $K$ is the mirror of an L-space knot, let $n_i, \ell_i$ and $n(K)$ be as before. (The Alexander polynomial, and hence these quantities, are unchanged under taking mirrors.) The knot Floer complex does change: There is now a model $\mathbb Z_2$-complex $C$ for $\CFKinfty(K)$ with $2m+1$ generators $x_0, x_1^1,x_1^2,\cdots,x_m^1,x_m^2$, as follows. The $(i,j)$-gradings of $x_0$ and $x_m^1$ are  $\operatorname{gr}(x_0)=(0,0)$ and $\operatorname{gr}(x_m^1) =  (n(K),g(K)-n(K))$. The gradings of $x_{m-2s}$ and $x_{m-(2s+1)}$ differ only in the $j$-grading by $\ell_{m-2s}$ and the gradings of $x_{m-(2s+1)}$ and $x_{m-(2s+2)}$ differ only in the $i$ grading by $\ell_{m-(2s+1)}$. Moreover, the complex is symmetric: if $\operatorname{gr}(x_s^1)=(i,j)$, then $\operatorname{gr}(x_s^2)=(j,i)$. The differentials in this complex form a ``staircase'' as in Figure~\ref{fig:mLComplex}. Precisely, if $m$ is odd, the nonzero differentials are
\begin{align*}
\del(x_1^t) &= x_0 + x_2^t \\
\del(x_s^t) &= x_{s+1}^t+x_{s-1}^t \text{ for } s \text{ odd, }  1 < s < m\\
\del(x_m^t) &= x_{m-1}^t
\end{align*}
\noindent whereas if $m$ is even, the differentials are
\begin{align*}
\del(x_0) &= x_1^1+x_1^2 \\
\del(x_s^t) &= x_{s+1}^t+x_{s-1}^t \text{ for } s \text{ even, } 0 < s < m \\
\del(x_m^t) &= x_{m-1}^t.
\end{align*}

\begin{figure}
\includegraphics[scale=.9]{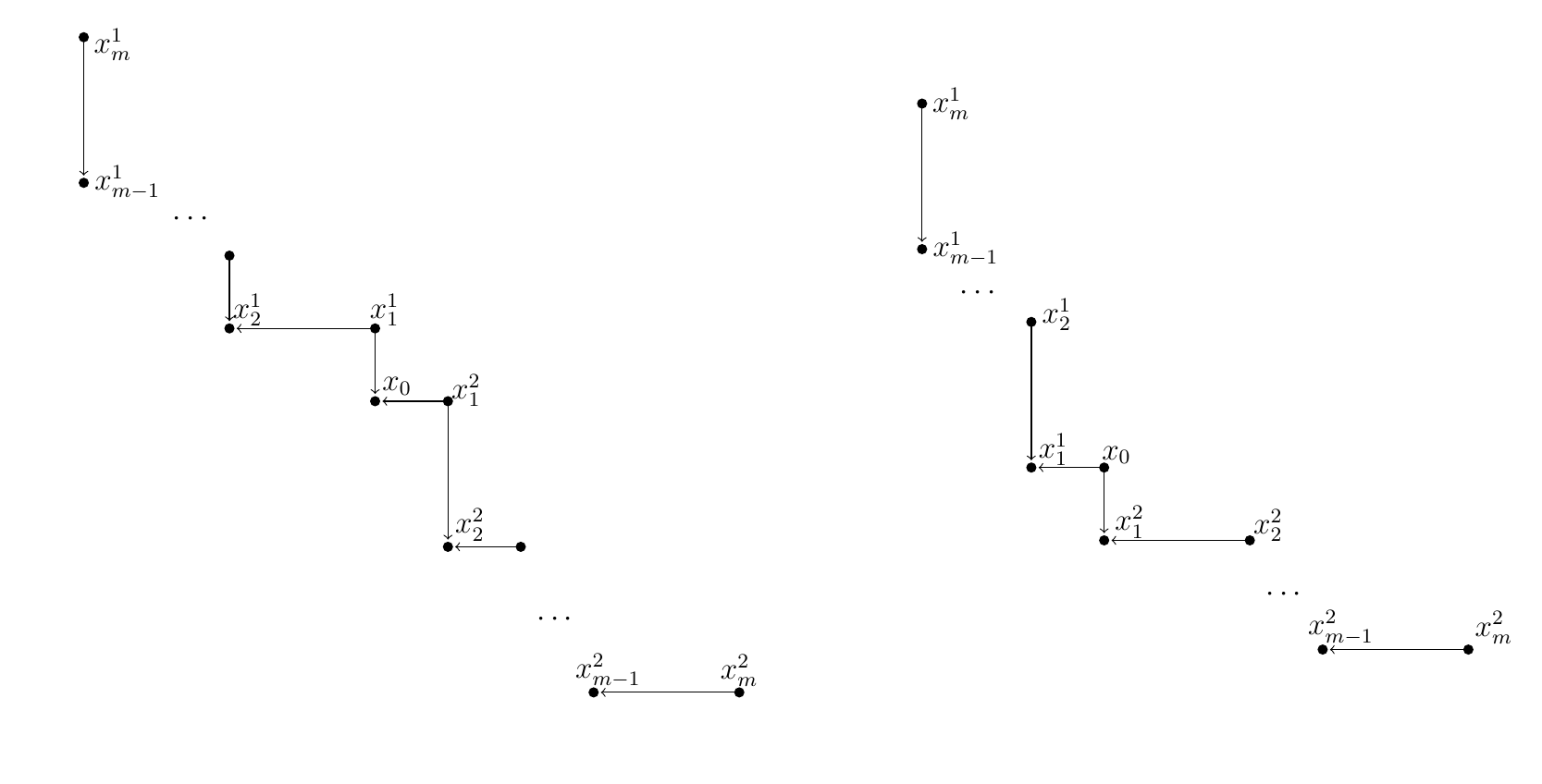}
\caption{Model complexes for the knot Floer complex of the mirror of an $L$-space knot. The case $m$ odd is on the left and the case $m$ even is on the right.}
\label{fig:mLComplex}
\end{figure}

Observe that now $n(K)$ represents the total length of the vertical arrows in the top half of the complex $C$. 

By Lemma~\ref{lem:transfer}, we can transfer the conjugation map $\inv_K$ to any complex filtered chain homotopy equivalent to $\CFKi(K)$; in particular, to $C \otimes \Z_2[U^{-1},U]$. Thus, we can assume that  $\CFKi(K)=C \otimes \Z_2[U^{-1},U]$, and look for a map $\inv_K$ on this complex that respects the Maslov grading, and is a skew-filtered quasi-isomorphism. 

In our case, since $U$ decreases Maslov grading by two, the map $\inv_K$ must preserve each subcomplex $U^n C$, for $n \in \Z$. In fact, given the requirements for $\inv_K$, it is easy to see that there is a unique possibility:
\begin{equation}
\label{eq:invL}
 \inv_K(x_0) = x_0,\ \inv_K(x^1_s)=x^2_s, \ \inv_K(x^2_s)=x^1_s.
 \end{equation}

Thus, we have all the information needed to compute the involutive Heegaard Floer homology of large surgeries on $K$. In particular, let us spell out the values of $\Vl_0(K)$ and $\Vu_0(K)$, which by Theorem~\ref{thm:dLarge} determine the involutive correction terms of the large surgeries.

\begin{proposition} \label{prop:mlspaceknots} Let $K$ be the mirror of an $L$-space knot. Then $\Vl_0(K) = V_0(K) =0$ and $\Vu_0(K) = -n(K)$.
\end{proposition}

\begin{proof} 
From the above description of $\CFKinfty(K)$, we see that $U^{n(K)}C$ is the largest $U$-power of $C$ with nontrivial intersection with $A_0^+$. For context, let us begin by computing $H_*(A_0^+)$ and $H_*(B_0^+)$, and thus $V_0$. (The reader who is familiar with this computation may skip the next two paragraphs.)

Let $C^{(k)} = (U^k  C) \cap A_0^+$, so that $H_*(A_0^+) = \bigoplus H_*(C^{(k)})$. For $k>n(K)$, we have $C^{(k)}=0$. For $n(K)\geq k > n(K)-\ell_m$, the complex $C^{(k)}$ has generators $U^kx_m^1$ and $U^kx_m^2$ and no nonzero differentials, hence has homology 
$$H_*(C^{(k)}) = \Z_2\langle [U^kx_m^1 + U^kx_m^2], [U^kx_m^2]\rangle.$$ 
Similarly, if $n(K)-\ell_m-\cdots -\ell_{m-2s}\geq k > n(K)-\ell_m-\cdots -\ell_{m-(2s+2)}$, then $C^{(k)}$ is generated by the $U^k$ powers of $x_m^1,x_m^2,\cdots,x_{m-2s}^1,x_{m-2s}^2$ with the inherited differentials, and has homology
$$\Z_2\langle [U^k\sum_{t=0}^{s}(x_{m-2t}^1+x_{m-2t}^2)], [U^k\sum_{t=0}^{s}x^2_{m-2t}] \rangle.$$ 

Finally, for $k\leq 0$, $C^{(k)}=U^kC \subset A_0^+$, and the homology $H_*(C^{(k)})$ is one of the following, according to the parity of $m$:
\begin{align*}
H_*(C^{(k)}) = \begin{cases} \Z_2 \langle [U^k(\sum_{t=0}^{\frac{m}{2}} (x^1_{m-2t}+x^2_{m-2t}) +x_0)]\rangle \text{ for } m \text{ even},\\
\Z_2 \langle [U^k(\sum_{t=0}^{\frac{m-1}{2}}(x^1_{m-2t}+x^2_{m-2t})]\rangle \text{ for } m \text{ odd.}\end{cases}
\end{align*}

We see that $H_*(A_0^+) \simeq \Tower \oplus \mathbb \Z_2^{\oplus n(K)}$, where the lowest-degree element in $\Tower$ can be taken to be $[U^{n(K)}(x_m^1+x_m^2)]$. An extremely similar computation shows that $H_*(B_0^+) \simeq \Tower$, where the lowest-degree element can be taken to be $[U^{n(K)}x_m^2]$, which is the image of $[U^{n(K)}(x_m^1+x_m^2)]$ under the map $v_0 \co H_*(A^+_0)\rightarrow H_*(B^+_0)$. Therefore, $v_0$ is modelled on the identity in sufficiently large degrees, and we have $V_0=0$.

We now turn our attention to the computation of $\Vl_0$ and $\Vu_0$. Recall that we must have $(1+\inv_K)x_s^t = x_s^1+x_s^2$ and $(1+\inv_K)x_0 = 0$.

Let us compute the homology of the complex $\AI^+_0$. Since $1+\inv_K$ preserves $C$, it suffices to compute the homology of
$$\CI^{(k)} := (C^{(k)}[-1] \oplus Q \ccdot C^{(k)}[-1], \del + Q \ccdot (1+\inv_K))$$
for each integer $k$. 

For $k>n(K)$, the complex $\CI^{(k)}$ is trivial. For $n(K)\geq k > n(K)-\ell_m$, the complex $\CI^{(k)}$ has generators $U^kx_m^t$ and $QU^kx_m^t$ ($t \in \{1,2\}$) and the only contribution to the differential comes from $Q\ccdot(1 + \inv_K)$; therefore, the homology $H_*(\CI^{(k)})$ is 
$$\Z_2\langle [U^k(x_m^1 +x_m^2)], [U^kQx_m^1]=[U^kQx_m^2]\rangle.$$ 
Similarly, if $n(K)-\ell_m-\cdots -\ell_{m-2s}\geq k > n(K)-\ell_m-\cdots -\ell_{m-(2s+2)}$, then $\CI^{(k)}$ is generated by $x_m^t,\cdots,x_{m-2s}^t,Qx_m^t,\cdots,Qx_{m-2s}^t,$ with the inherited differentials. The homology of this complex is 
$$\Z_2\langle [U^k\sum_{t=0}^{s}(x_{m-2t}^1+x_{m-2t}^2)], [U^kQ\sum_{t=0}^{s}x^1_{m-2t}]= [U^kQ\sum_{t=0}^{s}x^2_{m-2t}] \rangle.$$ 

Finally, for $k\leq 0$, the homology $H_*(\CI^{(k)})$ is one of the following, according to the parity of $m$:
\begin{align*}
H_*(\CI^{(k)}) = \begin{cases} \Z_2 \langle [U^k(\sum_{t=0}^{\frac{m}{2}} (x^1_{m-2t}+x^2_{m-2t}) +x_0)], \\ \ \ \ \ [U^kQ(\sum_{t=0}^{\frac{m}{2}} (x^1_{m-2t}+x^2_{m-2t}) +x_0)] \rangle\text{ for } m \text{ even,}\\
\Z_2 \langle [U^k(\sum_{t=0}^{\frac{m-1}{2}}(x^1_{m-2t}+x^2_{m-2t})],\\ \ \ \ \ [U^kQ(\sum_{t=0}^{\frac{m-1}{2}}(x^1_{m-2t}+x^2_{m-2t}))]\rangle \text{ for } m \text{ odd.}\end{cases}
\end{align*}

Notice that regardless of the parity of $m$, the generator of $H_*(\CI^{(0)})$ in the image of $Q$ is annihilated by multiplication by $U$. Therefore, combining all three computations above, we see that, as a $\Z_2[U]$-module, the involutive homology $H_*(\AI_0^{+})$ is $\Tower \oplus \Tower \oplus \mathbb Z_2^{\oplus n(K)}$, where the first tower has lowest-degree element $[U^{n(K)}(x_m^1+x_m^2)]$ and the second has lowest-degree element either $[Q(\sum_{t=0}^{\frac{m}{2}} (x^1_{m-2t}+x^2_{m-2t}) +x_0)]$ or $[Q(\sum_{t=0}^{\frac{m-1}{2}}(x^1_{m-2t}+x^2_{m-2t}))]$, according to the parity of $m$. Comparing with the homology of $A_0^+$, we conclude that $\Vl_0 = 0$ and $\Vu_0 = -n(K)$. 
\end{proof}

\section{Large surgeries on thin knots}
\label{sec:thin}

\subsection{Thin knots}
\label{sec:introthin}
A knot $K \subset S^3$ is called {\em Floer homologically thin} (or {\em thin}) if its knot Floer homology 
$$\HFKhat(K)=\bigoplus_{j,k\in \Z} \HFKhat_k(K, j)$$ is supported in a single diagonal line $j-k=\tau$, for some $\tau \in \Z$. Thin knots were introduced in \cite{RasmussenThesis, RasSurvey}. The value $\tau=\tau(K)$ is the knot concordance invariant defined in \cite{RasmussenThesis, 4BallGenus}.

Alternating knots were shown to be thin in \cite{AltKnots}. A more general class of knots, called quasi-alternating knots, were also found to be thin \cite{MOQuasi}. Furthermore, in these cases we have $\tau=-\sigma/2$, where $\sigma$ is the signature of the knot. It is an open question whether $\tau=-\sigma/2$ for all thin knots.

For thin knots, the complex $\CFKinfty(K)$ has been described by Petkova \cite[Lemma 7]{Petkova}. By her work, we know that there is a model complex $C$ for $\CFKinfty(K)$ consisting of a direct sum of a single staircase and some number of square complexes. 

Let us introduce some notation. For a thin knot $K$, all differentials in $\CFKinfty(K)$ are of length one, going in either the horizontal or the vertical direction. Given a basis element $x \in \CFKinfty(K)$ lying in grading $(i,j)$, we let $\dhorz(x)$ denote the horizontal part of the differential, so that $\dhorz(x)$ lies in grading $(i-1,j)$. Similarly, we let $\dvert(x)$ denote the vertical part of the differential, so that $\dvert(x)$ lies in grading $(i,j-1)$. If $y$ is a sum of basis elements in multiple gradings, then $\dhorz(y)$ and $\dvert(y)$ are computed using linearity. 

We now introduce the $\Z_2$ complexes that are used to build the model complex for $\CFKinfty$. A \textit{staircase complex of step length one} consists of $2m+1$ basis elements $x_0, x_1^1, x_2^1,\cdots, x_m^1, x_m^2$, such that $x_0$ is in grading $(0,0)$, and the complex has one of the forms shown in Figure \ref{fig:staircases}. 

\begin{figure}
\includegraphics[scale=.5]{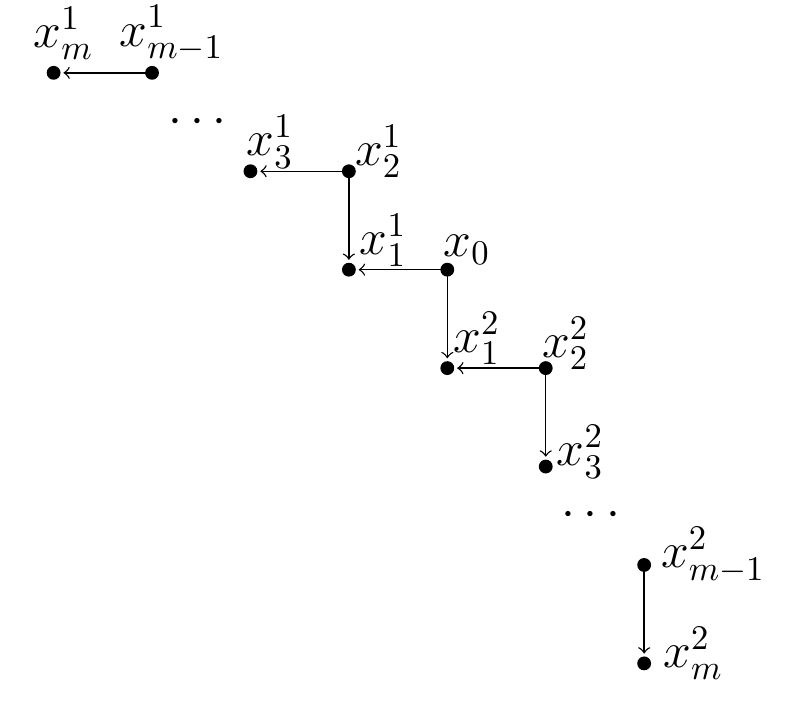}
\includegraphics[scale=.5]{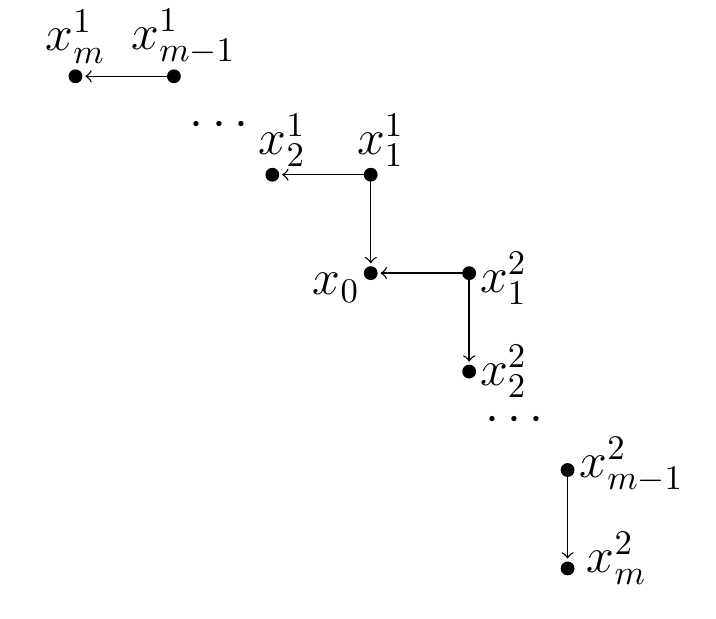}
\includegraphics[scale=.5]{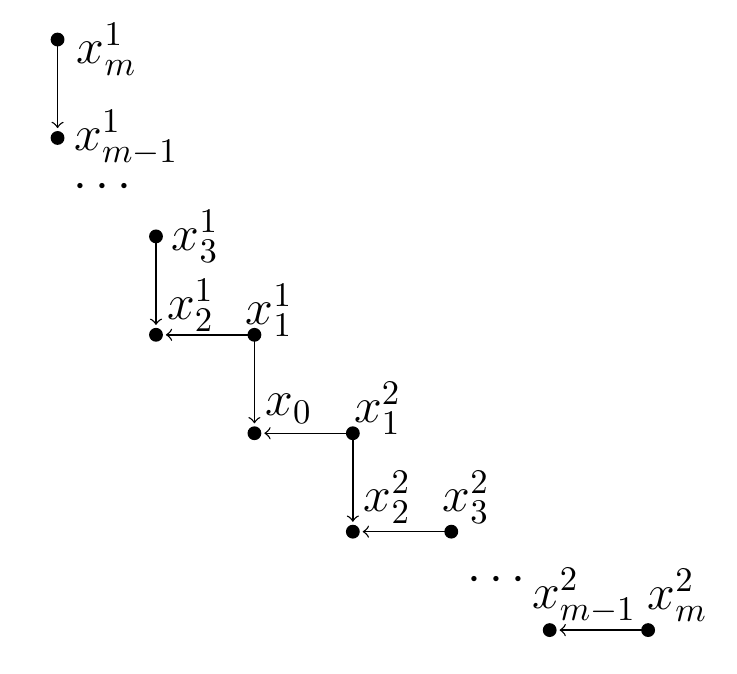}
\includegraphics[scale=.5]{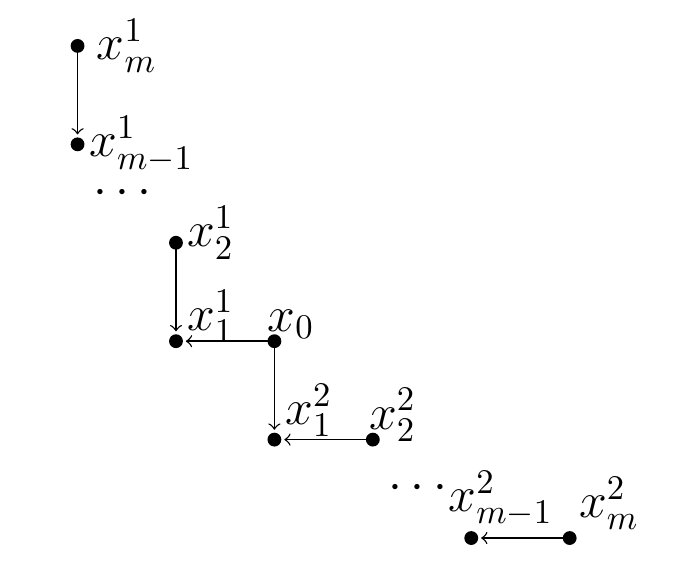}
\caption{The four staircase complexes with step length one. For larger versions, see Figures \ref{fig:thincomplex1}, \ref{fig:thincomplex2}, \ref{fig:thincomplex4}, and \ref{fig:thincomplex3}.}
\label{fig:staircases}
\end{figure}

A \textit{square complex} consists of four elements $a$ in grading $(i,j)$, $b$ in grading $(i-1,j)$, $c$ in grading $(i-1,j)$, and $Ue$ in grading $(i-1,j-1)$, with differentials
\begin{align*}
\dhorz(a) = b, \ \ 
\dvert(a) = c, \ \
\dvert(b) = \dhorz(c) = Ue, \ \
\del(Ue)=0. 
\end{align*}
A remark on notation: we use $Ue$ instead of $e$ so that $a$ and $e$ lie in the same grading, which will be convenient in future. Furthermore, because of the potential of confusion with a differential, we do not use ``$d$'' to label any element in the chain complex. 

The element $a$ is called the \textit{initial corner} of the square. An important feature of a initial corner $a$ is that $\dhorz\dvert(a) = \dvert\dhorz(a) \neq 0$. Notice that no other basis element of a square complex or a staircase complex has this property.

\begin{figure}
\includegraphics{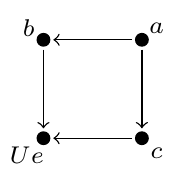}
\caption{A square complex}
\end{figure}

As mentioned above, the model complex $C$ consists of a single staircase and some number of square complexes. The value of the knot invariant $\tau$ can be read from the staircase. Indeed, for a staircase with $2m+1$ generators, we have 
\begin{equation}
\label{eq:mtau}
m = |\tau|.
\end{equation}
 Further, the sign of $\tau$ is positive if we are in one of the first two pictures in Figure~\ref{fig:staircases} (that is, the differential connecting $x^1_m$ and $x^1_{m-1}$ is horizontal), and is negative otherwise.

The Euler characteristic of $\HFKhat(K)$, which is the Alexander-Conway polynomial of the knot, takes the form
$$ \Delta_K(t) = (t^{-m} - t^{-m+1} + \dots - t^{m-1} + t^m) - (t^{-1}-2+t)r(t),$$
for some symmetric Laurent polynomial $r(t)$ of the form
$$ r(t)=r_0 - r_1(t^{-1}+t) + \dots + (-1)^k r_k(t^{-k} +t^k),$$
with $r_i$ being nonnegative integers. The polynomial $r(t)$ encodes the positions of the square complexes. In particular, $r(-1)$ represents the total number of squares.

The determinant of the knot is given by 
\begin{equation}
\label{eq:det}
D = \Delta_K(-1)=2m+1+4r(-1).
\end{equation}
Thus, the number of square complexes is
\begin{equation}
\label{eq:rminus}
r(-1)=(D-2|\tau|-1)/4.
\end{equation}
The parity of $r(-1)$ will play an important role in our considerations. Note that this is simply the parity of $r_0$, so we will refer to it as such.

Let us also define
\begin{equation}
\label{eq:nk}
 n(K)= \lceil m/2 \rceil.
\end{equation}
When there are no squares (so that $K$ is an L-space knot), this coincides with the definition of $n(K)$ in Section~\ref{sec:Lspace}.

Our goal will be to analyze the conjugation map $\inv_K$ on $\CFKinfty(K)$. By Lemma~\ref{lem:transfer}, we can assume that $\CFKinfty(K)$ is given by $C \otimes \Z_2[U, U^{-1}]$, where $C$ is our model complex. Suppose we have a generator $x=[\x, i, j]$ of $C \otimes \Z_2[U, U^{-1}]$.  Recall that $\inv_K$ is skew-filtered, that is, it takes $x$ to a sum of elements supported in the quadrant with upper right corner at $(i, j)$. Furthermore, since $K$ is a thin knot, the Maslov grading of $x$ is given by $i+j-\tau$. It follows that, in the $(i, j)$-plane, $\inv_K(x)$ must be on the diagonal with the same value of $i+j$ as $x$. Combining this with the skew-filtered property, we actually get that $\inv_K(x)$ must be supported in the point with coordinates $(j,i)$. 

To get more information about $\inv_K$, it is helpful to know its square $\inv_K^2$, which equals the Sarkar involution $\sarkar$; cf. Section~\ref{sec:Sarkar}. Since $\inv_K^2=\sarkar$, we must have that $\sarkar$ preserves the support $(i, j)$. One consequence of this is that $\sarkar$ equals the map $\sarkar_g$ induced by $\sarkar$ on the associated graded (with respect to the vertical filtration $j$).

Now, recall from Proposition~\ref{prop:sarkar} that the map $\sarkar_g$ can be computed explicitly. In fact, for thin knots, the result of this computation appears in \cite[Section 6]{SarkarMoving}.  We can describe it as follows. On the staircase complex, $\sarkar=\sarkar_g$ is the identity. On each square complex, we have
\begin{equation}
\label{eq:SarkarSquare}
\sarkar(a)= a+e, \ \ \sarkar(b)=b,\ \ \sarkar(c)=c,\  \ \sarkar(e)=e.
\end{equation}

We are left to find $\inv_K$ that squares to this map, and takes elements supported at $(i,j)$ to elements supported at $(j,i)$. 

\subsection{The figure-eight knot} 
\label{sec:figure-eight}
As a first and motivating example, we compute $\inv_K$ in the case of the figure-eight knot $4_1$, for which $m=0$ and $r(t)=1$. The complex $\CFKinfty(4_1)$ appears in Figure~\ref{fig:full}. We write $x=x_0$ for simplicity.

\begin {figure}
\begin {center}
\includegraphics{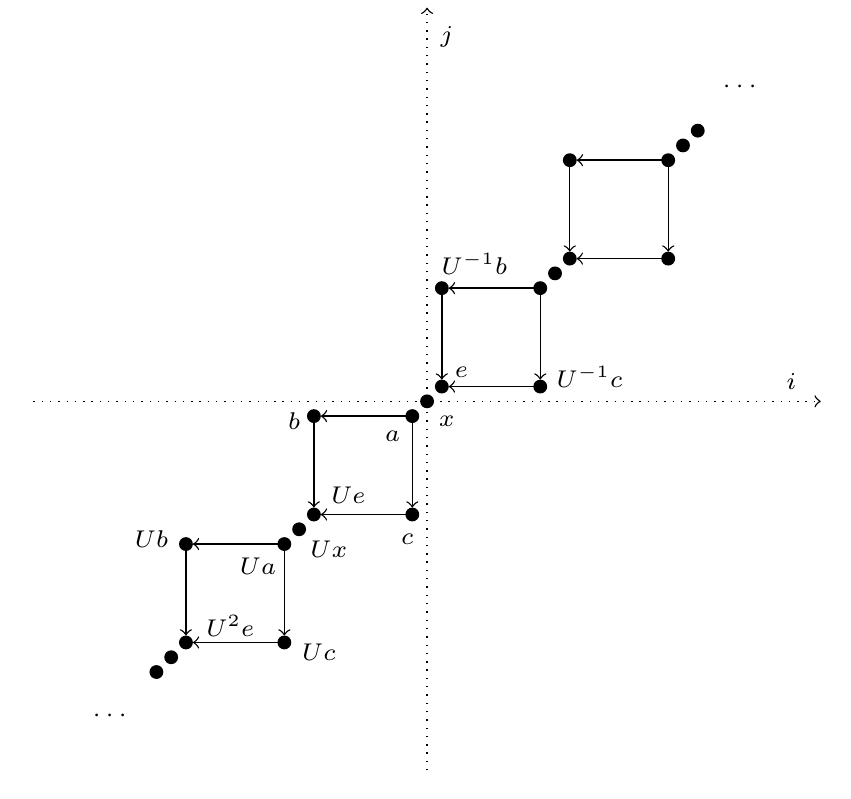}
\caption {The knot Floer complex for the figure-eight.}
\label {fig:full}
\end {center}
\end {figure}

This chain complex has eight grading-reversing automorphisms. Indeed, in all of them we must have $\inv_K(b)=c, \inv_K(c)=b$ and (because $\inv_K$ is a chain map) $\inv_K(e)=e$. Further, $\inv_K(x)$ is either $x$ or $x+e$, and $\inv_K(a)$ is $a, a+e, a+x$ or $a+x+e$. This gives eight possibilities. However, six are eliminated because they do not square to the Sarkar involution $ x \to x, \ a \to a+e.$

The remaining two maps
$$ x \mapsto x+ e, \ a \mapsto a+x$$
$$ x \mapsto x+e, \ a \mapsto a+x+e$$
are conjugate to each other (via turning $x$ into $x+e$), so up to change of basis there is only one map on $\CFKinfty(4_1)$ that could be $\inv_K$. Without loss of generality, let $\inv_K$ be the first map above, so that $1+\inv_K$ is
$$ e \mapsto 0, \ \ x \mapsto e,\ \ a \mapsto x, \ \ b,c \mapsto b+c.$$

We can use this to compute the knot invariants $\Vl_0$ and $\Vu_0$. We see that 
$$H_*(A_0^+) = \Z_2\langle [U^nx] : n\leq 0\rangle \oplus \Z_2 \langle [b] \rangle$$
and $H_*(B_0^+) = \Z_2\langle [U^nx] : n\leq 0\rangle,$ so $V_0=0$. Now consider the complex $\AI_0^+$. This chain complex breaks up into a direct sum of subcomplexes, the first two of which are illustrated in Figure~\ref{fig:eight}. The subcomplex containing the lowest graded elements, depicted on the left, is generated by the seven elements $a, b, c, Qx, Qa, Qb$, and $Qc$. The next subcomplex, depicted on the right, is generated by ten elements,  with differentials as shown in the figure; all further subcomplexes are identical to this subcomplex multiplied by $U^{-n}$. The homology of the left subcomplex is  $\Z_2\langle[Qx], [Qa+c], [Qc]\rangle$ and the homology of the right subcomplex is $\Z_2\langle  [U^{-1}Qx], [x+U^{-1}Qc] \rangle$. Observe that $U[x+ U^{-1}Qc] = [Qc]$. Therefore, 
$$H_*(\AI_0^+) \simeq \Tower_{-1} \oplus \Tower_{0} \oplus (\Z_2)_{(-1)},$$ 
where the towers can be taken to have lowest-degree elements $[Qc]$ and $[Qx]$. We conclude that $\Vl_0 = 1$ and $\Vu_0 = 0$.

\begin {figure}
\begin {center}
\includegraphics[scale=.8]{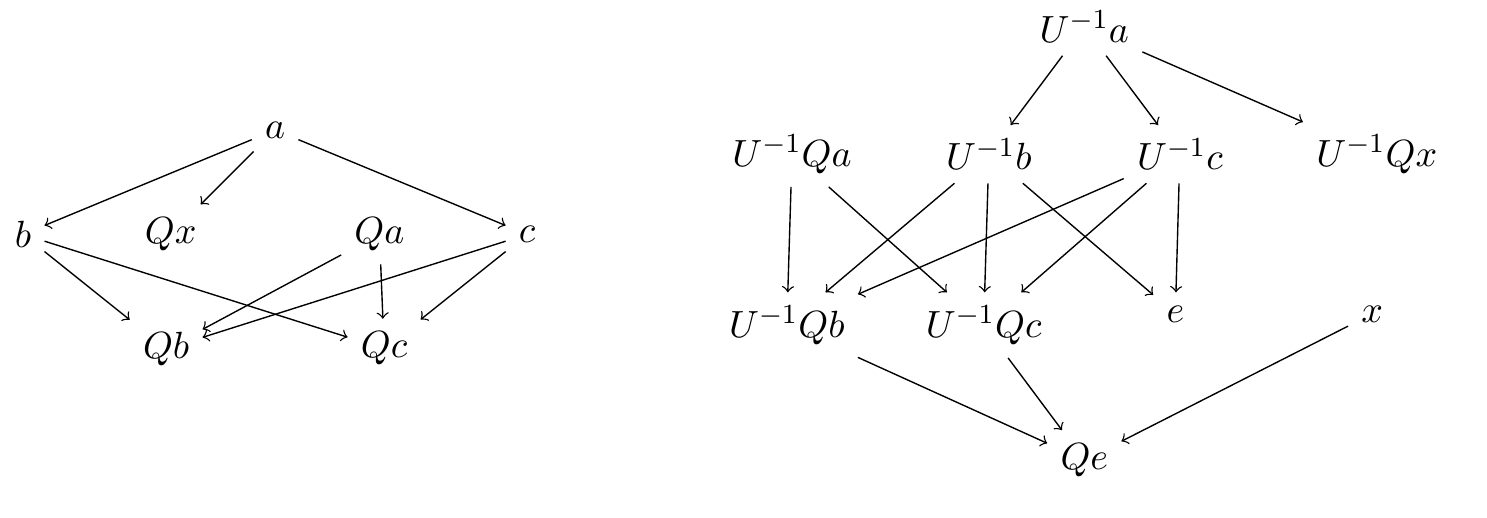}
\caption {Parts of the $\AI_0^+$ complex for the figure-eight.}
\label {fig:eight}
\end {center}
\end {figure}

\subsection{The general case}
Now let $K$ be an arbitrary thin knot. We seek to determine the conjugation map $\inv_K$ on $\CFKinfty(K)$. Let $C$ be the model complex for $\CFKinfty(K)$, described in Section~\ref{sec:introthin}.

For a staircase complex, we define the \textit{standard staircase map} to be the involution that exchanges $x^1_s$ and $x^2_{s}$ and fixes $x_0$. Note that this is a genuine involution. 

Given two squares $\{a,b,c,Ue\}$ and $\{a',b',c',Ue'\}$ such that the grading of $a$ is $(i,j)$ and the grading of $a'$ is $(j,i)$, we define \textit{standard square map on the pair} to be the map
\begin{align*}
a \mapsto a', \ \  b \mapsto c', \ \ c \mapsto b', \ \ e \mapsto e', \\
a' \mapsto a+ e, \ \ c' \mapsto b, \ \ b' \mapsto c, \ \ e' \mapsto e.
\end{align*}

Note that this squares to the Sarkar involution \eqref{eq:SarkarSquare} on the squares, which takes $a\mapsto a+e$, $a' \mapsto a' +e'$, and is the identity on $b,c,e,b',c',e'$.

\begin{proposition} \label{propn:thin} Let $K$ be a thin knot. Up to a (grading-preserving) change in basis, there is a unique automorphism $\inv_K$ on $\CFKinfty(K)$ that exchanges the $i$ and $j$ gradings and squares to the the Sarkar involution.
\end{proposition}

\begin{proof}  
Notice that because $\inv_K$ is a chain map that interchanges the gradings on $\CFKinfty(K)$, if $x$ is a basis element in a single grading, we must have $\dhorz  (\inv_K (x)) = \inv_K (\dvert(x))$. This implies that $\dhorz\inv_K = \inv_K  \dvert$ on $\CFKinfty$, and similarly $\dvert  \inv_K = \inv_K  \dhorz$. In particular, $\dvert  \dhorz  \inv_K = \dvert  \inv_K  \dvert = \inv_K  \dhorz  \dvert = \inv_K  \dvert  \dhorz$, so $\inv_K$ and $\dvert  \dhorz=\dhorz \dvert$ commute. 

The idea of this proof is to split off pairs of square complexes related by the standard square map from the model complex $C$, until what remains is a staircase or (more interestingly) a staircase and a single square.

Consider the set of elements $\{a_s\}$ which lie at the initial corners of squares in the model complex $C$, so that $\dhorz\dvert(a_s) = \dvert\dhorz(a_s) \neq 0$. Recall that $a_s$ are the only basis elements that are not mapped to zero by $\dhorz \dvert$. 

First, assume there is some $a_s$ in planar grading $(i,j)$ such that $i\neq j$. Then $\inv_K(a)$ lies in planar grading $(j,i)$ and has $\dhorz \dvert(\inv_K(a)) \neq 0$. Therefore $\inv_K(a)$, written as a sum of basis elements each appearing once, contains an initial corner $a' \neq a$. We will use this fact to do a change of basis so that the square $\{a,b,c,Ue\}$ is exchanged with another square complex by the standard square map on pairs.

Let $\dhorz(a)=b$, $\dvert(a)=c$, and $\dhorz\dvert(a)=Ue$. Then consider $\inv_K(a)$. Since $\inv_K$ is a chain map that interchanges the vertical and horizontal gradings, we observe that $\dhorz(\inv_K(a))=\inv_K(c)$, $\dvert(\inv_K(a))=\inv_K(b)$, and $\dhorz\dvert(\inv_K(a))=\inv_K(Ue) \neq 0$. Change basis by adding $\inv_K(a)+a'$ to $a'$, $\inv_K(b)+ \dvert(a')$ to $\dvert(a')$, $\inv_K(c)+ \dhorz(a')$ to $\dhorz(a')$, and $\inv_K(Ue)+\dvert\dhorz(a')$ to $\dvert\dhorz(a')$. This change of basis has the effect of changing the model complex $C$ to $C'$ in which the squares $\{a,b,c,Ue\}$ and $\{\inv_K(a), \inv_K(c), \inv_K(b), \inv_K(Ue)\}$ are interchanged by the standard square map.

Our goal now is to split $C\otimes \mathbb \Z_2[U^{-1}, U]$ as $(C^1 \otimes \Z_2[U^{-1},U]) \oplus (C^2 \otimes \Z_2[U^{-1},U]))$ such that each $C^i \otimes \Z_2[U^{-1},U]$ is preserved by $\inv_K$, and $C^2$ is still a direct sum of square complexes and a single staircase. To do this, we must change basis again. Let $C^1$ be the direct sum of the squares $\{a,b,c,Ue\}$ and $\{\inv_K(a), \inv_K(c), \inv_K(b), \inv_K(Ue)\}$. Suppose there is some $x \in C'$ a basis element not in $C^1$ such that $\inv_K(x)$, as a sum of elements in $C'$ in which each element appears at most once, contains $\inv_K(a)$. Then we change basis by adding $a$ to $x$, so that $\inv_K(x+a)$ no longer contains $\inv_K(a)$. Of course, we should now add $b$ to $\dhorz(x)$, $c$ to $\dvert(x)$, and $e$ to $\dhorz\dvert(x)$. Repeating this process for other elements as necessary, we can arrange to have a complex $C''$ in which there is no $x \neq a$ such that $\inv_K(a)$ appears nontrivially in $\inv_K(x)$. Consequently, there is also no $y \neq \inv_K(a)$ such that $a$ appears nontrivially in $\inv_K(y)$, because then $\inv_K(a)$ would appear in $\inv_K^2(y)$; since $\inv_K(a)$ is now the only basis element with the property that $\inv_K$ applied to it contains $a$, nothing can cancel with this term. Next, we do similar changes of bases to produce a complex $C'''$ for which there is no basis element $x \notin C^1$ such that $\inv_K(x)$ contains $b,c,d, \inv_K(b), \inv_K(c)$, or $\inv_K(d)$ nontrivially, achieving our goal.

We repeat the argument above, splitting off subcomplexes generated by pairs of squares with the standard square involution, until all square complexes remaining have initial corners on the main diagonal. We may repeat the same procedure with any square whose initial corner $a$ lies on the main diagonal which has the property that $\inv_K(a)$, expressed as a sum of basis elements each appearing once, includes an initial corner $a'$ in $C$ other than $a$. 

This leaves us with a set of initial corners $\{a_1,\dots ,a_t\}$ such that $\inv_K(a_r)$ contains only $a_r$ and no other initial corner, for $r=1, \dots, t$. (Of course, since $\dhorz\dvert(\inv_K(a_r))=\inv_K(e_r) \neq 0$, $\inv_K(a_r)$ must contain at least one initial corner in $C$, since these are the only elements with $\dhorz\dvert\neq 0$.) Suppose that the number of remaining squares, $t$, is at least $2$. Without loss of generality (because of the tensor product with $\Z_2[U^{-1},U]$), each $a_r$ lies in planar grading $(0,0)$. Then $e_1, \dots, e_t$ must also lie in grading $(0,0)$. Since we have already split off any other square complexes in pairs, we see that $\{a_1,\dots, a_t, e_1,\dots e_t, x_0\}$ are all of the basis elements in the grading $(0,0)$. Furthermore, $\dhorz\dvert(\inv_K(a_r))=\dhorz\dvert(a_r)=Ue_r$, so each $e_r$ must be fixed by $\inv_K$.  Finally, because $\dhorz\dvert(x_0) = 0$, we see that $\inv_K(x_0)$, as a sum of basis elements, cannot contain any $a_r$. This implies that we know the following about the behavior of the involution in grading $(0,0)$.
$$\inv_K(e_r)=e_r, \ \ \ \inv_K(a_r) = a_r + \beta_r x_0 + \sum_{s=1}^t \gamma_r^s e_s, \ \ \  \inv_K(x_0) = x_0 + \sum_{s=1}^t \delta_s e_s.$$
Here each $\beta_r, \gamma_r^s$ and $\delta_s$ is either $0$ or $1$ as appropriate. This implies that 
\begin{align*}
\inv_K^2(a_r) &= \inv_K(a_r+\beta_r x_0 + \sum_{s=1}^t \gamma_r^s e_s)\\
&=a_r+\beta_r x_0+ \sum_{s=1}^t \gamma_r^s e_s + \beta_r x_0 + \beta_r \sum_{s=1}^t \delta_s e_s + \sum_{s=1}^t \gamma_r^s e_s \\
&= a_r + \beta_r \sum_{s=1}^t \delta_s e_s.
\end{align*}
However, by assumption $\inv_K^2$ is the Sarkar involution $\sarkar$, so $\inv_K^2(a_r)=a_r+e_r$. This implies that $\beta_r=1, \ \delta_r = 1$ and all other $\delta_s$ are 0. Since $r$ was arbitrary, this is a contradiction when $t\geq 2$. We conclude that $t=0$ or $1$.

Therefore, we can always split off pairs of squares related by the standard square map until we have either a staircase complex $C_m$ of step length one or a staircase complex $C_m$ of step length one and a single square complex whose initial corner $a$ lies on the main diagonal. (The first case occurs when the original complex had an even number of square summands, i.e. $r_0 \equiv 0 $ mod 2, and the second case occurs when the original complex had an odd number of squares, i.e. $r_0 \equiv 1$ mod 2.) In the first case, $\inv_K$ must be the standard staircase map. For the second case, without loss of generality $x_0$ and $a$ lie in the grading $(0,0)$. An argument similar to the one given for the figure-eight knot in Section \ref{sec:figure-eight} shows that up to change of basis, the map $\inv_K$ in the grading $(0,0)$ must be
\begin{align*}
a\mapsto a+x_0, \ \ x_0 \mapsto x_0+e, \ \ e \mapsto e.
\end{align*}
Once this is established, the fact that $\inv_K$ is a chain map and interchanges the gradings completely determines the automorphism. There are four different cases, illustrated in Figures \ref{fig:thincomplex1}, \ref{fig:thincomplex2}, \ref{fig:thincomplex4}, and \ref{fig:thincomplex3}. 

In the first and fourth cases (shown in Figures~\ref{fig:thincomplex1} and \ref{fig:thincomplex3}, respectively) we must have
\begin{align*}
x_s^1&\mapsto x_s^2, &
b&\mapsto c+x_1^2,\\ 
x_s^2&\mapsto x_s^1,& 
c&\mapsto b+x_1^1. 
\end{align*}

In the second and third cases (shown in Figures~\ref{fig:thincomplex2} and \ref{fig:thincomplex4}) we get
\begin{align*}
x_1^1&\mapsto x_1^2+U^{-1}c,  &
x_s^1&\mapsto x_s^2  \ \text{  for } s>1, &
b&\mapsto c, \\
x_1^2&\mapsto x_1^1+U^{-1}b, &
x_s^2&\mapsto x_s^1\ \text{ for } s>1, &
c&\mapsto b.
\end{align*}
This completes the description of the automorphism $\inv_K$. \end{proof}

In view of Theorem~\ref{thm:LargeSurg} and Proposition~\ref{propn:thin}, we have all the information needed to calculate $\HFIp$ of large surgeries on thin knots. In particular, by Theorem~\ref{thm:dLarge}, the involutive correction terms are determined by the values of $\Vl_0$ and $\Vu_0$ for these knots. These values can be computed explicitly: 

\begin{proposition} \label{propn:lowervthin} Let $K$ be a thin knot. The invariants $\Vl_0$ and $\Vu_0$ associated to $K$ depend on the sign and parity of $\tau(K)$ and on the parity of $r_0$ as follows.
\begin{enumerate}
\item If $r_0\equiv 0$ mod 2, we have the following cases:
\begin{enumerate}
\item If $\tau(K)\geq 0$, then $\Vl_0=\Vu_0=V_0 = n(K)$.
\item If $\tau(K)<0$, then $\Vl_0=V_0 = 0$ and $\Vu_0=-n(K)$.
\end{enumerate}

\item If $r_0 \equiv 1$ mod 2, then we have the following cases:

\begin{enumerate}
\item If $\tau(K)=0$, then $\Vl_0 = 1$ and $V_0=\Vu_0 = 0$.
\item If $\tau(K)>0$ and $\tau(K) \equiv 1$ mod 2, then $\Vl_0=V_0=n(K)$ and $\Vu_0 = n(K)-1$.
\item If $\tau(K)>0$ and $\tau(K) \equiv 0$ mod 2, then $\Vl_0 = n(K)+1$ and $V_0=\Vu_0 = n(K)$.
\item If $\tau(K)<0$ and $\tau(K) \equiv 1$ mod 2, then $\Vl_0=V_0=0$ and $\Vu_0=-n(K)+1$.
\item If $\tau(K)<0$ and $\tau(K) \equiv 0$ mod 2, then $\Vl_0 =V_0= 0$ and $\Vu_0=-n(K)$.
\end{enumerate}
\end{enumerate}
\end{proposition}

\begin{remark}
Of course, the values of $V_0$ for thin knots were already known. See \cite[Corollary 1.5]{AltKnots} for the case of alternating knots; the arguments there readily extend to all thin knots. We included $V_0$ in the statement of Proposition~\ref{propn:lowervthin} for easy comparison with $\Vl_0$ and $\Vu_0$.
\end{remark}

\begin{proof}[Proof of Proposition~\ref{propn:lowervthin}] By Proposition \ref{propn:thin}, if $r_0\equiv 0$ mod 2 then it suffices to consider the case of a staircase complex. We obtain the same results as in case of $L$-space knots and mirrors of $L$-space knots. This proves Cases (1a) and (1b).

The more interesting calculation is when $r_0\equiv 1$ mod 2. Then, again by Proposition \ref{propn:thin}, it suffices to consider the case of a staircase complex together with a single square complex on the main diagonal. If $\tau(K)=0$ the staircase in question is a single point, and this is the complex for the figure-eight knot. By our computation in Section \ref{sec:figure-eight}, we obtain Case (2a).

\begin {figure}
\begin {center}
\includegraphics{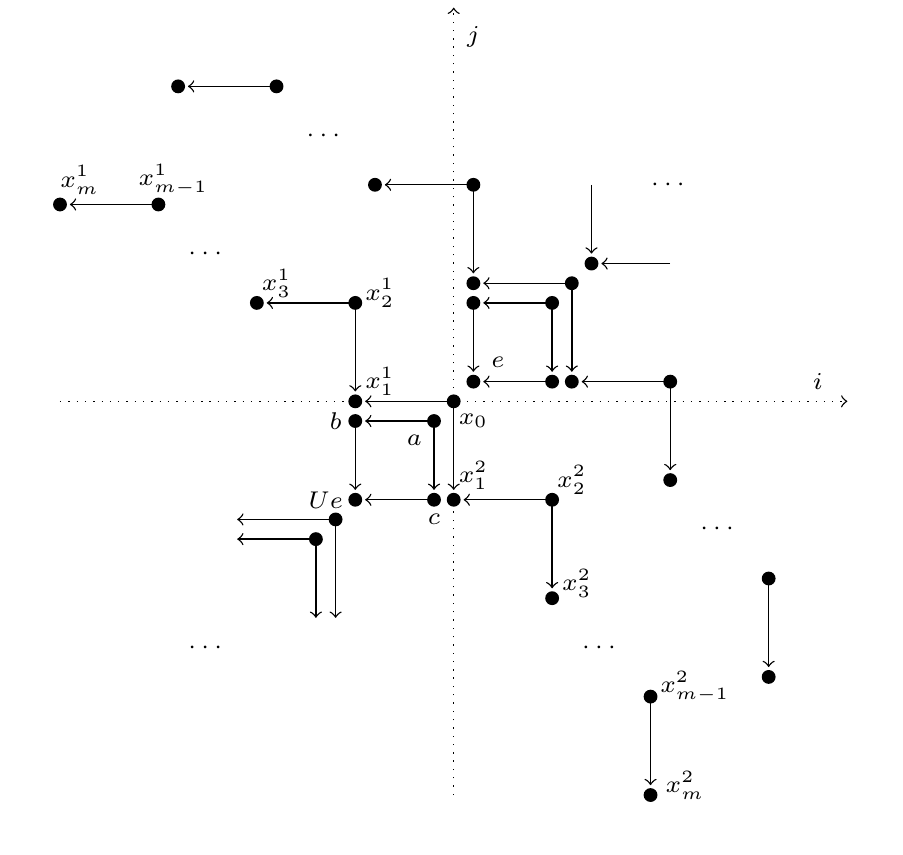}
\caption{A model complex for a thin knot with $\tau(K)>0$, and $\tau(K)$ and $r_0$ both odd.}
\label{fig:thincomplex1}
\end {center}
\end {figure}

The proofs of Cases (2b)-(2d) are computational, and are all variations on the same strategy. We will do Cases (2b) and (2c) in detail.

Suppose that $\tau(K)>0$ and $\tau(K)$ is odd. This is Case (2b). The relevant portion of the knot Floer complex is illustrated in Figure~\ref{fig:thincomplex1}. Observe that for this knot complex, $H_*(A_0^+) = \Z_2[U^{-1}]\langle [x_1^1]\rangle \oplus \Z_2\langle[b]\rangle$ and $V_0 = n(K)$. Therefore we will compare $\Vl_0$ and $\Vu_0$ to $n(K)$ by computing the difference in gradings between the generators of towers in $H_*(\AI_0^+)$ and the gradings of the elements  $x_{1}^1$ and $Qx_{1}^1$.
From the description of $\inv_K$ in the proof of Proposition~\ref{propn:thin} we see that the map $1+\inv_K$ is given by
\begin{align*}
x_0&\mapsto e, & 
b&\mapsto (b+c)+x_1^2, &
x_s^t&\mapsto x_s^1+x_s^2, \\
a&\mapsto x_0, &
c &\mapsto (b+c)+x_2^2, &
e&\mapsto 0.
\end{align*}
We start by breaking the full involutive complex 
$$\bigl( \CFKinfty(K)[-1] \oplus Q \ccdot \CFKinfty(K)[-1],\ \del + Q \ccdot (1+\inv_K) \bigr)$$
into the direct sum of three subcomplexes over $\Z_2[U^{-1}, U]$. The first of these is generated by a model $\Z_2$-complex $C_1$ with basis consisting of the elements $x_s^1+x_s^2$, $x_0$, $Qe$, and $U^{-1}Qb$. The nonzero differentials in this subcomplex are as follows:
\begin{align*}
&\delinv(U^{-1}Qb) = Qe, \\
&\delinv(x_0) = (x_1^1+x_1^2) + Qe, \\
&\delinv(x_s^1+x_s^2) = (x_{s+1}^1 + x_{s+1}^2) + (x_{s-1}^1+x_{s-1}^2) \text{ for } s \text{ even. }
\end{align*}
This subcomplex is acyclic.

 Our second subcomplex is generated by a model $\Z_2$-complex $C_2$ with generators $x_s^1$, $Q(x_s^1+x_s^2)$, $Qx_0$, $a$, and $b+c$. Here are the nonzero differentials:
\begin{align*}
&\delinv(a) = (b+c) + Qx_0, \\
&\delinv(x_s^1)=  Q(x_s^1+x_s^2)  \text{ for } s \text{ odd,} \\
&\delinv(b+c) = \delinv(Qx_0) = Q(x_1^1+x_1^2),\\
&\delinv(x_s^1)= x_{s+1}^1 + x_{s-1}^1 + Q(x_s^1+x_s^2) \text{ for } s \text{ even,} \\
&\delinv(Q(x_s^1+x_s^2))=Q(x_{s+1}^1+x_{s+1}^2) + Q(x_{s-1}^1+x_{s-1}^2) \text{ for } s \text{ even.} 
\end{align*} 

The homology of this subcomplex is $\Z_2[U,U^{-1}]\langle [x_1^1 + b + c]\rangle$.

Our third and last subcomplex is generated by a model $\Z_2$-complex $C_3$ generated by the elements  $b$, $Qa$, $Ue$, $Q(b+c)$, and $Qx_s^2$. It has nonzero differentials as follows:
\begin{align*}
&\delinv(Qa) = Q(b+c), \\
&\delinv(b) = Ue+ Q(b+c) + Qx_1^2, \\
&\delinv(Qx_s^2) = Qx_{s+1}^2 + Qx_{s-1}^2 \text{ for } s \text{ even.}
\end{align*}
The homology of this subcomplex is $\Z_2[U,U^{-1}]\langle [Qx_1^2]\rangle$. 

Now we can investigate the homology of the intersection of each of these three complexes with $\AI_0^+$, i.e., we will look at those terms in the region of the $(i, j)$-plane given by $\max (i, j) \geq 0$. The direct sum of these intersections will give $\AI_0^+$, and we are interested in the two infinite $U$-towers in its homology.

First, consider the intersection of the subcomplex generated by $C_1$ with $\AI_0^+$. For $k\leq 0$, the complex $U^k C_1$ is contained in $\AI_0^+$. Since $C_1$ is acyclic, it must be that neither of the two towers in $H_*(\AI_0^+)$ comes from $C_1 \otimes \Z_2[U^{-1}, U].$

Let us move to our second subcomplex $C_2$. All nonpositive $U$-powers of $C_2$ are contained in ${\AI_0^+}$, and have homology generated over $\Z_2$ by $U^{-k} [x_1^1 + b + c]$. However, the intersection of $U \ccdot C_2$ with ${\AI_0^+}$ is generated over $\Z_2$ by $Ux_s^1$ and $UQ(x_s^1+x_s^2)$ for $s\geq 2$, with nonzero differentials as follows:
\begin{align*}
&\delinv(Ux_{2}^1) = Ux_{3}^1 + UQ(x_{2}^1+x_{2}^2),\\
&\delinv(Ux_s^1)=  UQ(x_s^1+x_s^2)  \text{ for } s \text{ odd}, \\
&\delinv(UQ(x_{2}^1 + x_{2}^2))=UQ(x_{3}^1+x_{3}^2),\\
&\delinv(Ux_s^1)= Ux_{s+1}^1 + Ux_{s-1}^1 + UQ(x_s^1+x_s^2) \text{ for } s \text{ even, }s>2 ,\\
&\delinv(UQ(x_s^1+x_s^2))=UQ(x_{s+1}^1+x_{s+1}^2) + UQ(x_{s-1}^1+x_{s-1}^2) \text{ for } s \text{ even, }s>2.
\end{align*} 
This complex is acyclic. We conclude that the first infinite tower in $H_*(\AI_0^+)$ has $[x_1^1 + b + c]$ as the lowest degree element.

We finally turn our attention to the subcomplex generated by $C_3$. All negative $U$-powers of $C_3$ are contained in $\AI_0^+$ and have homology generated over $\Z_2$ by $[U^{-k} Qx_1^2]$. However, the intersection of $C_3$ itself with ${\AI_0^+}$ loses $Ue$. That is, $C_3 \cap {\AI_0^+}$ is generated by $b$, $Qa$, $Qx^2_2$, $Q(b+c)$, and $Qx_s^2$ with nonzero differentials as follows:
\begin{align*}
&\delinv(Qa) = Q(b+c), \\
&\delinv(b) = Q(b+c) + Qx_1^2, \\
&\delinv(Qx_s^2) = Qx_{s+1}^2 + Qx_{s-1}^2 \text{ for } s \text{ even.}
\end{align*}
This is acyclic. This means that the second infinite tower in $H_*(\AI_0^+)$ has $[U^{-1}Qx_1^2]$ as the lowest degree element.

Recall that the tower in $H_*(A_0^+)$ has $[x_1^1]$ as the lowest degree element, and $V_0=n(K)$.  Thus, the grading of $[x_1^1 + b + c]$ in $H_*(\AI_0^+)$ is one more than the grading of $[x_1^1]$ in $H_*(A_0^+)$. Further, $[x_1^1]$ and $[x_1^2]$ have the same grading, so the grading of $[U^{-1}Qx_1^2]$ in $H_*(\AI_0^+)$ is two more than the grading of $[x_1^1]$ in $H_*(A_0^+)$. We conclude that $\Vl(K) = n(K)$ and $\Vu(K) = n(K)-1$. 

\begin {figure}
\begin {center}
\includegraphics{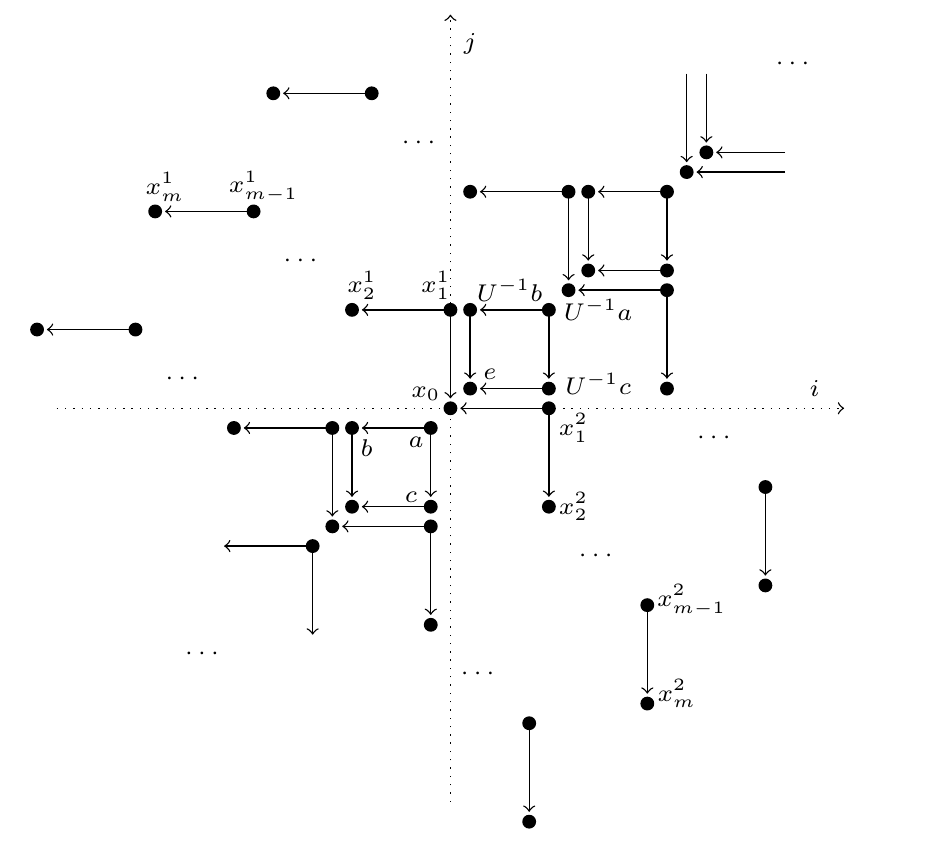}
\caption{A model complex for a thin knot with $\tau(K)>0$, $\tau(K)$ even, and $r_0$ odd.}
\label{fig:thincomplex2}
\end {center}
\end {figure}

Now consider Case (2c), with $\tau(K)>0$ and even, and $r_0$ odd. The complex $\CFKinfty(K)$ is illustrated in Figure~\ref{fig:thincomplex2}. Observe that for this knot complex, $H_*(A_0^+) = \Z_2[U^{-1}]\langle [x_0]\rangle \oplus \Z_2\langle[b]\rangle$ and $V_0 = n(K)$. Therefore we will compare $\Vl_0$ and $\Vu_0$ to $n(K)$ by computing the difference in gradings between the generators of towers in $H_*(\AI_0^+)$ and the gradings of the elements  $x_{0}$ and $Qx_{0}$.
The map $1+\inv_K$ on this complex is given by
\begin{align*}
a&\mapsto x_0, &
x_1^1&\mapsto x_1^1+x_1^2+U^{-1}c, &
x_s^1&\mapsto x_s^1+x_s^2\ \text{ for } s>1,  \\
x_0&\mapsto e, & 
x_1^2&\mapsto x_1^1+x_1^2+U^{-1}b & 
x_s^2&\mapsto x_s^1+x_s^2\ \text{ for } s>1,\\
e&\mapsto 0, &
b&\mapsto b+c, &
c&\mapsto b+c.
\end{align*}

We start, as in the previous case, by breaking the full involutive complex
$$\bigl( \CFKinfty(K)[-1] \oplus Q \ccdot \CFKinfty(K)[-1],\ \del + Q \ccdot (1+\inv_K) \bigr)$$
into the direct sum of four subcomplexes over $\Z_2[U^{-1}, U]$. The first of these is generated by a model $\Z_2$-complex $C_1$ with basis consisting of the elements $x_s^1+x_s^2$, $U^{-1}Q(b+c)$, and $U^{-1}Qa$. The nonzero differentials in this subcomplex are as follows:
\begin{align*}
&\delinv(U^{-1}Qa) = U^{-1}Q(b+c), \\
&\delinv(x_1^1+x^2_1) = U^{-1}Q(b+c) + (x_2^1+x_2^2), \\
&\delinv(x_s^1+x_s^2)= (x_{s+1}^1 + x_{s+1}^2) + (x_{s-1}^1+x_{s-1}^2) \ \text{ for } s \text{ odd, } s>1.
\end{align*}
This subcomplex is acyclic.
 
Our second subcomplex is generated by a model $\mathbb Z_2$-complex $C_2$ with basis consisting of the elements $Qa+c$ and $e$. The only nonzero differential is the following:
\begin{align*}
\delinv(Qa+c)=e.
\end{align*}
This subcomplex is also acyclic. 

Our third subcomplex is generated by a model $\mathbb Z_2$-complex $C_3$ with basis consisting of the elements $x_s^1, Q(x_s^1+x_s^2), x_0, U^{-1}Qc$ and $Qe$. The nonzero differentials are as follows:
\begin{align*}
&\delinv(x_0) = \delinv(U^{-1}Qc) = Qe, \\
&\delinv(Q(x_1^1 + x_1^2)) = Q(x_2^1+x_2^2), \\
&\delinv(x_s^1) = Q(x_s^1+x_s^2)  \text{ for } s  \text{ even, } \\
&\delinv(x_1^1) = x_2^1 + Q(x_1^1+x_1^2) + x_0 + U^{-1}Qc, \\
&\delinv(Q(x_{s}^1+ x_s^2)) = Q(x_{s+1}^1 + x_{s+1}^2) + Q(x_{s-1}^1 + x_{s-1}^2)  \text{ for } s  \text{ odd, } s>1,\\
&\delinv(x_s^1) = x_{s+1}^1+x_{s-1}^1 + Q(x_s^1+x_s^2)  \text{ for } s  \text{ odd, } s>1.
\end{align*}

The homology of this complex is $\Z_2[U^{-1},U]\langle [x_0 + U^{-1}Qc] \rangle$.

Our final subcomplex is generated by a model complex $C_4$ with generators $Qx_s^1, Qx_0, a$, and $b+c$. Here are the nonzero differentials:
\begin{align*}
&\delinv(a) = (b+c) + Qx_0 \\
&\delinv(Qx_1^1) = Qx_2^1+Qx_0 \\
&\delinv(Qx_s^1) = Qx_{s+1}^1+Qx_{s-1}^1 \text{ for } s \text{ odd, } s>1.
\end{align*}
The homology of this complex is $\Z_2[U^{-1}, U]\langle [Qx_0] \rangle$.

We now turn our attention to the homology of the intersection of each of these four complexes with $\AI_0^{+}$, to wit, the region of the plane with $\max(i,j) \geq 0$. The direct sum of these intersections gives $\AI_0^{+}$. As before, we are interested in finding the two infinite $U$-towers in its homology. 

First, consider the intersection of $C_1$ with $\AI_0^{+}$. For $k\geq -1$, $U^kC_1$ is contained in $\AI_0^{+}$. Since $C_1$ is acyclic, neither of the two towers in $H_*(\AI_0^{+})$ comes from $C_1 \otimes \Z_2[U^{-1},U]$. By similar logic, neither of the two towers comes from $C_2 \otimes \mathbb Z_2[U^{-1},U]$.

Now let us consider our third subcomplex, generated by $C_3$. All nonpositive $U$-powers of $C_2$ are contained in $\AI_0^{+}$, and have homology generated by $U^{-k}[x_0 + U^{-1}Qc]$. However, the intersection of $U \cdot C_2$ with $\AI_0^{+}$ is generated over $\mathbb Z_2$ by $Ux_s^1, UQ(x_s^1+x_s^2)$, and $Qc$, with nonzero differentials as follows:
\begin{align*}
&\delinv(UQ(x_1^1 + x_1^2)) = UQ(x_2^1+x_2^2), \\
&\delinv(Ux_s^1) = UQ(x_s^1+x_s^2)  \text{ for } s  \text{ even, } \\
&\delinv(Ux_1^1) = Ux_2^1 + UQ(x_1^1+x_1^2) + Qc, \\
&\delinv(UQ(x_{s}^1+ x_s^2)) = UQ(x_{s+1}^1 + Ux_{s+1}^2) +  UQ(x_{s-1}^1 + Ux_{s-1}^2)\text{ for } s  \text{ odd, } s>1,\\
&\delinv(Ux_s^1) = Ux_{s+1}^1+Ux_{s-1}^1 + UQ(x_s^1+x_s^2)  \text{ for } s  \text{ odd, } s>1.
\end{align*}
The homology of this complex is generated over $\mathbb Z_2$ by $[Qc]$. Since $U[x_0 + QU^{-1}c] = [Qc]$, we conclude this element is part of the first infinite tower. Finally, the intersection of $U^2 \cdot C_2$ with $\AI_0^{+}$ is generated over $\mathbb Z_2$ by $U^2x_s^1, QU^2(x_s^1+x_s^2)$ with $s>2$, and has nonzero differentials as follows:
\begin{align*}
&\delinv(U^2Q(x_3^1 + x_3^2)) = U^2Q(x_4^1+x_4^2), \\
&\delinv(U^2x_s^1) = U^2Q(x_s^1+x_s^2)  \text{ for } s  \text{ even, } \\
&\delinv(U^2x_3^1) = U^2x_4^1 + U^2Q(x_3^1+x_3^2) \\
&\delinv(U^2Q(x_{s}^1+ x_s^2)) = U^2Q(x_{s+1}^1 + x_{s+1}^2) +  U^2Q(x_{s-1}^1 + x_{s-1}^2) \text{ for } s  \text{ odd, } s>3,\\
&\delinv(U^2x_s^1) = U^2x_{s+1}^1+U^2x_{s-1}^1 + U^2Q(x_s^1+x_s^2)  \text{ for } s  \text{ odd, } s>3.
\end{align*}
This complex is acyclic. We conclude that the bottom element of the first infinite tower in $H_*(\AI_0^+)$ is $[Qc]$.

Finally, we turn our attention to the fourth subcomplex, generated by $C_4$. Every nonpositive $U$-power of $C_4$ is contained in $\AI_0^{+}$ and has homology generated over $\Z_2$ by $U^{-k}[Qx_0]$. However, $U \cdot C_4$ is generated over $\Z_2$ by the elements $UQx_s^1$, with nonzero differentials as follows: 
\begin{align*}
&\delinv(Qx_1^1) = Qx_2^1 \\
&\delinv(Qx_s^1) = Qx_{s+1}^1+Qx_{s-1}^1 \text{ for } s \text{ odd, } s>1.
\end{align*}
This complex is acyclic. We conclude that the bottom element in the second infinite tower in $H_*(\AI_0^{+})$ is $[Qx_0]$. 

Now, recall that the tower in $H_*(A_0)$ has $[x_0]$ as its lowest degree element, and $V_0 =n(K)$. The grading of $[Qc]$ in $H_*(\AI_0^{+})$ is one less than the grading of $[x_0]$ in $H_*(A_0^{+})$, and the grading of $[Qx_0]$ in $H_*(\AI_0^{+})$ is equal to the grading of $[x_0]$ in $H_*(A_0^{+})$.
We conclude that $\Vl(K) = n(K)+1$ and $\Vu(K)=n(K)$.
\medskip

\begin {figure}
\begin {center}
\includegraphics{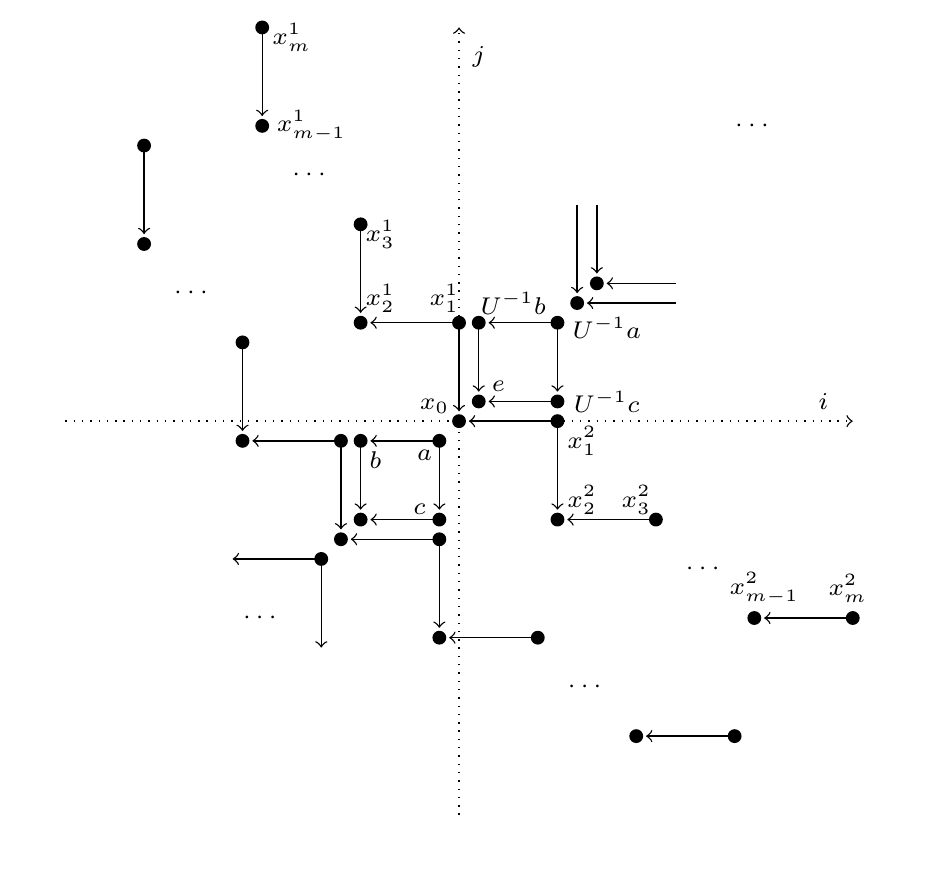}
\caption{A model complex for a thin knot with $\tau(K)<0$, $\tau(K)$ odd, and $r_0$ odd.}
\label{fig:thincomplex4}
\end {center}
\end {figure}

\begin {figure}
\begin {center}
\includegraphics{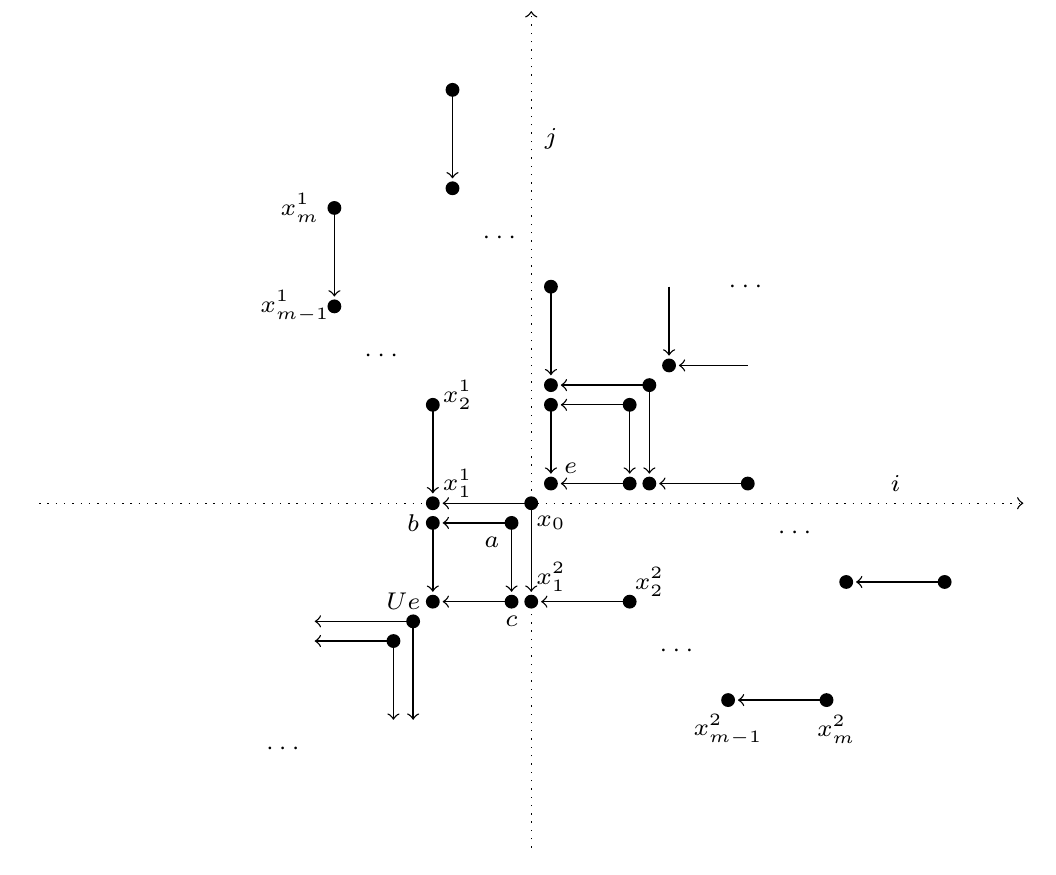}
\caption{A model complex for a thin knot with $\tau(K)<0$, $\tau(K)$ even, and $r_0$ odd.}
\label{fig:thincomplex3}
\end {center}
\end {figure}

Cases (2d) and (2e) are pictured in Figures~\ref{fig:thincomplex4} and ~\ref{fig:thincomplex3}, respectively. We leave the computations in these cases to the interested reader.\end{proof}

\begin{proof}[Proof of Theorem~\ref{thm:Alt}] This is now a simple consequence of Proposition~\ref{propn:lowervthin}, taking into account the relation $\tau = -\sigma(K)/2$ for alternating knots, the identities \eqref{eq:mtau}, \eqref{eq:det}, \eqref{eq:rminus} and \eqref{eq:nk}, as well as the fact that a knot has Arf invariant $0$ if and only if $D \equiv \pm 1 \pmod 8.$ 
\end{proof}

\begin{proof}[Proof of Corollary~\ref{cor:altHC}] We use the following $\Z_2$-homology cobordism invariants of $\Z_2$-homology spheres $Y$: $d$, $\dl$, $\du$, and the (generalized) Rokhlin invariant 
$$\mu(Y) \in \frac{1}{8} \Z \pmod{2\Z}$$ 
from \cite{EellsKuiper}. If $Y=S^3_p(K)$ for $p > 0$ odd, Proposition 3 in \cite{BohrLee} says that
$$ \mu(S^3_p(K)) = \frac{1}{8}(1-p) + \Arf(K) \pmod{2\Z}.$$

In our case, if $S^3_p(K)$ and $S^3_p(K')$ are $\Z_2$-homology cobordant, we see that $K$ and $K'$ have the same Arf invariant. Further, in view of Equation~\eqref{eq:dspk} and Theorem~\ref{thm:dLarge}, the values of $d$, $\dl$ and $\du$ for $p$-surgery on a knot are determined by the values of $(\Vl_0, V_0, \Vu_0)$ for that knot. It follows that $K$ and $K'$ have the same triple of values $(\Vl_0, V_0, \Vu_0)$. By inspecting the tables in Theorem~\ref{thm:Alt}, we conclude that $K$ and $K'$ have the same signature.
\end{proof}

\begin{remark}
The results in Theorem~\ref{thm:Alt} and Corollary~\ref{cor:altHC} apply equally well to quasi-alternating knots.
\end{remark}

\bibliographystyle{custom} 
\bibliography{biblio} 

\end{document}

%% file: triple.pdf_t
\begin{picture}(0,0)%
\includegraphics{triple.pdf}%
\end{picture}%
\setlength{\unitlength}{2763sp}%
\begingroup\makeatletter\ifx\SetFigFont\undefined%
\gdef\SetFigFont#1#2#3#4#5{%
  \reset@font\fontsize{#1}{#2pt}%
  \fontfamily{#3}\fontseries{#4}\fontshape{#5}%
  \selectfont}%
\fi\endgroup%
\begin{picture}(4252,1516)(1042,-6076)
\put(3065,-5243){\makebox(0,0)[lb]{\smash{{\SetFigFont{8}{9.6}{\rmdefault}{\mddefault}{\updefault}{\color[rgb]{0,0,0}$w$}%
}}}}
\put(5279,-5794){\makebox(0,0)[lb]{\smash{{\SetFigFont{8}{9.6}{\rmdefault}{\mddefault}{\updefault}{\color[rgb]{1,0,0}$\alpha_1$}%
}}}}
\put(5238,-5986){\makebox(0,0)[lb]{\smash{{\SetFigFont{8}{9.6}{\rmdefault}{\mddefault}{\updefault}{\color[rgb]{1,0,0}$\alpha_2$}%
}}}}
\put(5279,-5794){\makebox(0,0)[lb]{\smash{{\SetFigFont{8}{9.6}{\rmdefault}{\mddefault}{\updefault}{\color[rgb]{1,0,0}$\alpha_1$}%
}}}}
\put(1585,-5361){\makebox(0,0)[lb]{\smash{{\SetFigFont{8}{9.6}{\rmdefault}{\mddefault}{\updefault}{\color[rgb]{0,.56,0}$\gamma_g$}%
}}}}
\put(3334,-5242){\makebox(0,0)[lb]{\smash{{\SetFigFont{8}{9.6}{\rmdefault}{\mddefault}{\updefault}{\color[rgb]{0,0,0}$z$}%
}}}}
\put(3050,-4731){\makebox(0,0)[lb]{\smash{{\SetFigFont{8}{9.6}{\rmdefault}{\mddefault}{\updefault}{\color[rgb]{0,0,1}$\beta_g$}%
}}}}
\end{picture}%

%% file: plus1.pdf_t
\begin{picture}(0,0)%
\includegraphics{plus1.pdf}%
\end{picture}%
\setlength{\unitlength}{2960sp}%
\begingroup\makeatletter\ifx\SetFigFont\undefined%
\gdef\SetFigFont#1#2#3#4#5{%
  \reset@font\fontsize{#1}{#2pt}%
  \fontfamily{#3}\fontseries{#4}\fontshape{#5}%
  \selectfont}%
\fi\endgroup%
\begin{picture}(5628,1534)(449,-3373)
\put(601,-2611){\makebox(0,0)[lb]{\smash{{\SetFigFont{11}{13.2}{\rmdefault}{\mddefault}{\updefault}{\color[rgb]{0,0,0}$\dots$}%
}}}}
\put(5551,-2611){\makebox(0,0)[lb]{\smash{{\SetFigFont{11}{13.2}{\rmdefault}{\mddefault}{\updefault}{\color[rgb]{0,0,0}$\dots$}%
}}}}
\put(1126,-3211){\makebox(0,0)[lb]{\smash{{\SetFigFont{9}{10.8}{\rmdefault}{\mddefault}{\updefault}{\color[rgb]{0,0,0}$B^+$}%
}}}}
\put(2926,-3211){\makebox(0,0)[lb]{\smash{{\SetFigFont{9}{10.8}{\rmdefault}{\mddefault}{\updefault}{\color[rgb]{0,0,0}$B^+$}%
}}}}
\put(2026,-3211){\makebox(0,0)[lb]{\smash{{\SetFigFont{9}{10.8}{\rmdefault}{\mddefault}{\updefault}{\color[rgb]{0,0,0}$B^+$}%
}}}}
\put(3826,-3211){\makebox(0,0)[lb]{\smash{{\SetFigFont{9}{10.8}{\rmdefault}{\mddefault}{\updefault}{\color[rgb]{0,0,0}$B^+$}%
}}}}
\put(4651,-3211){\makebox(0,0)[lb]{\smash{{\SetFigFont{9}{10.8}{\rmdefault}{\mddefault}{\updefault}{\color[rgb]{0,0,0}$B^+$}%
}}}}
\put(2851,-2086){\makebox(0,0)[lb]{\smash{{\SetFigFont{9}{10.8}{\rmdefault}{\mddefault}{\updefault}{\color[rgb]{0,0,0}$A_s^+$}%
}}}}
\put(3751,-2086){\makebox(0,0)[lb]{\smash{{\SetFigFont{9}{10.8}{\rmdefault}{\mddefault}{\updefault}{\color[rgb]{0,0,0}$A_{s+p}^+$}%
}}}}
\put(4576,-2086){\makebox(0,0)[lb]{\smash{{\SetFigFont{9}{10.8}{\rmdefault}{\mddefault}{\updefault}{\color[rgb]{0,0,0}$A_{s+2p}^+$}%
}}}}
\put(1951,-2086){\makebox(0,0)[lb]{\smash{{\SetFigFont{9}{10.8}{\rmdefault}{\mddefault}{\updefault}{\color[rgb]{0,0,0}$A_{s-p}^+$}%
}}}}
\put(1051,-2086){\makebox(0,0)[lb]{\smash{{\SetFigFont{9}{10.8}{\rmdefault}{\mddefault}{\updefault}{\color[rgb]{0,0,0}$A_{s-2p}^+$}%
}}}}
\end{picture}%

%% file: compound.pdf_t
\begin{picture}(0,0)%
\includegraphics{compound.pdf}%
\end{picture}%
\setlength{\unitlength}{2605sp}%
\begingroup\makeatletter\ifx\SetFigFont\undefined%
\gdef\SetFigFont#1#2#3#4#5{%
  \reset@font\fontsize{#1}{#2pt}%
  \fontfamily{#3}\fontseries{#4}\fontshape{#5}%
  \selectfont}%
\fi\endgroup%
\begin{picture}(8995,5088)(-1036,-5923)
\put(6151,-1006){\makebox(0,0)[lb]{\smash{{\SetFigFont{8}{9.6}{\rmdefault}{\mddefault}{\updefault}{\color[rgb]{1,0,0}$\alpha_1$}%
}}}}
\put(1641,-5226){\makebox(0,0)[lb]{\smash{{\SetFigFont{8}{9.6}{\rmdefault}{\mddefault}{\updefault}{\color[rgb]{0,0,0}$\zeta$}%
}}}}
\put(763,-1006){\makebox(0,0)[lb]{\smash{{\SetFigFont{8}{9.6}{\rmdefault}{\mddefault}{\updefault}{\color[rgb]{1,0,0}$\alpha_2$}%
}}}}
\put(7241,-2211){\makebox(0,0)[lb]{\smash{{\SetFigFont{8}{9.6}{\rmdefault}{\mddefault}{\updefault}{\color[rgb]{0,0,1}$\beta_n$}%
}}}}
\put(6976,-1006){\makebox(0,0)[lb]{\smash{{\SetFigFont{8}{9.6}{\rmdefault}{\mddefault}{\updefault}{\color[rgb]{1,0,0}$\alpha_3$}%
}}}}
\put(5401,-1606){\makebox(0,0)[lb]{\smash{{\SetFigFont{8}{9.6}{\rmdefault}{\mddefault}{\updefault}{\color[rgb]{1,0,0}$\alpha_n$}%
}}}}
\put(1201,-1006){\makebox(0,0)[lb]{\smash{{\SetFigFont{8}{9.6}{\rmdefault}{\mddefault}{\updefault}{\color[rgb]{1,0,0}$\alpha_3$}%
}}}}
\put(6516,-1006){\makebox(0,0)[lb]{\smash{{\SetFigFont{8}{9.6}{\rmdefault}{\mddefault}{\updefault}{\color[rgb]{1,0,0}$\alpha_2$}%
}}}}
\put(376,-1006){\makebox(0,0)[lb]{\smash{{\SetFigFont{8}{9.6}{\rmdefault}{\mddefault}{\updefault}{\color[rgb]{1,0,0}$\alpha_1$}%
}}}}
\put(1651,-2216){\makebox(0,0)[lb]{\smash{{\SetFigFont{8}{9.6}{\rmdefault}{\mddefault}{\updefault}{\color[rgb]{0,0,0}$\zeta$}%
}}}}
\end{picture}%

%% file: KnotStabilized.pdf_t
\begin{picture}(0,0)%
\includegraphics{KnotStabilized.pdf}%
\end{picture}%
\setlength{\unitlength}{2763sp}%
\begingroup\makeatletter\ifx\SetFigFont\undefined%
\gdef\SetFigFont#1#2#3#4#5{%
  \reset@font\fontsize{#1}{#2pt}%
  \fontfamily{#3}\fontseries{#4}\fontshape{#5}%
  \selectfont}%
\fi\endgroup%
\begin{picture}(4570,2931)(1054,-8964)
\put(3394,-8012){\makebox(0,0)[lb]{\smash{{\SetFigFont{8}{9.6}{\rmdefault}{\mddefault}{\updefault}{\color[rgb]{0,0,0}$z$}%
}}}}
\put(3055,-8013){\makebox(0,0)[lb]{\smash{{\SetFigFont{8}{9.6}{\rmdefault}{\mddefault}{\updefault}{\color[rgb]{0,0,0}$w$}%
}}}}
\put(5609,-7694){\makebox(0,0)[lb]{\smash{{\SetFigFont{8}{9.6}{\rmdefault}{\mddefault}{\updefault}{\color[rgb]{1,0,0}$\alpha_1$}%
}}}}
\put(5548,-7916){\makebox(0,0)[lb]{\smash{{\SetFigFont{8}{9.6}{\rmdefault}{\mddefault}{\updefault}{\color[rgb]{1,0,0}$\alpha_2$}%
}}}}
\put(3250,-8891){\makebox(0,0)[lb]{\smash{{\SetFigFont{8}{9.6}{\rmdefault}{\mddefault}{\updefault}{\color[rgb]{0,0,1}$\beta_g$}%
}}}}
\put(5550,-7161){\makebox(0,0)[lb]{\smash{{\SetFigFont{8}{9.6}{\rmdefault}{\mddefault}{\updefault}{\color[rgb]{0,0,1}$\beta_{g+1}$}%
}}}}
\put(3189,-6204){\makebox(0,0)[lb]{\smash{{\SetFigFont{8}{9.6}{\rmdefault}{\mddefault}{\updefault}{\color[rgb]{1,0,0}$\alpha_{g+1}$}%
}}}}
\end{picture}%

%% file: Base.pdf_t
\begin{picture}(0,0)%
\includegraphics{Base.pdf}%
\end{picture}%
\setlength{\unitlength}{2763sp}%
\begingroup\makeatletter\ifx\SetFigFont\undefined%
\gdef\SetFigFont#1#2#3#4#5{%
  \reset@font\fontsize{#1}{#2pt}%
  \fontfamily{#3}\fontseries{#4}\fontshape{#5}%
  \selectfont}%
\fi\endgroup%
\begin{picture}(9193,2698)(-4489,-8802)
\put(-2630,-7524){\makebox(0,0)[lb]{\smash{{\SetFigFont{8}{9.6}{\rmdefault}{\mddefault}{\updefault}{\color[rgb]{0,0,0}$w'$}%
}}}}
\put(-2754,-8221){\makebox(0,0)[lb]{\smash{{\SetFigFont{8}{9.6}{\rmdefault}{\mddefault}{\updefault}{\color[rgb]{0,0,0}$\xi$}%
}}}}
\put(-2764,-6861){\makebox(0,0)[lb]{\smash{{\SetFigFont{8}{9.6}{\rmdefault}{\mddefault}{\updefault}{\color[rgb]{0,0,0}$\xi'$}%
}}}}
\put(376,-7411){\makebox(0,0)[lb]{\smash{{\SetFigFont{8}{9.6}{\rmdefault}{\mddefault}{\updefault}{\color[rgb]{0,0,0}$\mm$}%
}}}}
\put(-318,-7214){\makebox(0,0)[lb]{\smash{{\SetFigFont{8}{9.6}{\rmdefault}{\mddefault}{\updefault}{\color[rgb]{0,0,0}$180^\circ$}%
}}}}
\put(-2896,-7519){\makebox(0,0)[lb]{\smash{{\SetFigFont{8}{9.6}{\rmdefault}{\mddefault}{\updefault}{\color[rgb]{0,0,0}$z$}%
}}}}
\put(4689,-7774){\makebox(0,0)[lb]{\smash{{\SetFigFont{8}{9.6}{\rmdefault}{\mddefault}{\updefault}{\color[rgb]{1,0,0}$\alpha_1$}%
}}}}
\put(4638,-7986){\makebox(0,0)[lb]{\smash{{\SetFigFont{8}{9.6}{\rmdefault}{\mddefault}{\updefault}{\color[rgb]{1,0,0}$\alpha_2$}%
}}}}
\put(4640,-7131){\makebox(0,0)[lb]{\smash{{\SetFigFont{8}{9.6}{\rmdefault}{\mddefault}{\updefault}{\color[rgb]{0,0,1}$\beta'_g$}%
}}}}
\put(2789,-8234){\makebox(0,0)[lb]{\smash{{\SetFigFont{8}{9.6}{\rmdefault}{\mddefault}{\updefault}{\color[rgb]{0,0,0}$\xi'$}%
}}}}
\put(2779,-6874){\makebox(0,0)[lb]{\smash{{\SetFigFont{8}{9.6}{\rmdefault}{\mddefault}{\updefault}{\color[rgb]{0,0,0}$\xi$}%
}}}}
\put(2913,-7536){\makebox(0,0)[lb]{\smash{{\SetFigFont{8}{9.6}{\rmdefault}{\mddefault}{\updefault}{\color[rgb]{0,0,0}$z$}%
}}}}
\put(-936,-7084){\makebox(0,0)[lb]{\smash{{\SetFigFont{8}{9.6}{\rmdefault}{\mddefault}{\updefault}{\color[rgb]{1,0,0}$\alpha$}%
}}}}
\put(-865,-7746){\makebox(0,0)[lb]{\smash{{\SetFigFont{8}{9.6}{\rmdefault}{\mddefault}{\updefault}{\color[rgb]{0,0,1}$\beta$}%
}}}}
\put(-925,-8001){\makebox(0,0)[lb]{\smash{{\SetFigFont{8}{9.6}{\rmdefault}{\mddefault}{\updefault}{\color[rgb]{0,0,1}$\beta$}%
}}}}
\put(2640,-7538){\makebox(0,0)[lb]{\smash{{\SetFigFont{8}{9.6}{\rmdefault}{\mddefault}{\updefault}{\color[rgb]{0,0,0}$w'$}%
}}}}
\end{picture}%

%% file: nbhd.pdf_t
\begin{picture}(0,0)%
\includegraphics{nbhd.pdf}%
\end{picture}%
\setlength{\unitlength}{1579sp}%
\begingroup\makeatletter\ifx\SetFigFont\undefined%
\gdef\SetFigFont#1#2#3#4#5{%
  \reset@font\fontsize{#1}{#2pt}%
  \fontfamily{#3}\fontseries{#4}\fontshape{#5}%
  \selectfont}%
\fi\endgroup%
\begin{picture}(8478,9078)(-41,-8201)
\put(181,505){\makebox(0,0)[lb]{\smash{{\SetFigFont{10}{12.0}{\rmdefault}{\mddefault}{\updefault}{\color[rgb]{0,0,0}$\pi$}%
}}}}
\put(6297,-1128){\makebox(0,0)[lb]{\smash{{\SetFigFont{10}{12.0}{\rmdefault}{\mddefault}{\updefault}{\color[rgb]{0,0,0}$U$}%
}}}}
\put(2476,-3736){\makebox(0,0)[lb]{\smash{{\SetFigFont{10}{12.0}{\rmdefault}{\mddefault}{\updefault}{\color[rgb]{0,0,0}$w'$}%
}}}}
\put(6076,-3661){\makebox(0,0)[lb]{\smash{{\SetFigFont{10}{12.0}{\rmdefault}{\mddefault}{\updefault}{\color[rgb]{0,0,0}$z$}%
}}}}
\put(5476,-1561){\makebox(0,0)[lb]{\smash{{\SetFigFont{10}{12.0}{\rmdefault}{\mddefault}{\updefault}{\color[rgb]{0,0,0}$\xi$}%
}}}}
\put(4876,-5911){\makebox(0,0)[lb]{\smash{{\SetFigFont{10}{12.0}{\rmdefault}{\mddefault}{\updefault}{\color[rgb]{0,0,0}$\xi'$}%
}}}}
\put(6376,-4636){\makebox(0,0)[lb]{\smash{{\SetFigFont{10}{12.0}{\rmdefault}{\mddefault}{\updefault}{\color[rgb]{0,0,0}$K$}%
}}}}
\end{picture}%

%% file: MovesKnot.pdf_t
\begin{picture}(0,0)%
\includegraphics{MovesKnot.pdf}%
\end{picture}%
\setlength{\unitlength}{2565sp}%
\begingroup\makeatletter\ifx\SetFigFont\undefined%
\gdef\SetFigFont#1#2#3#4#5{%
  \reset@font\fontsize{#1}{#2pt}%
  \fontfamily{#3}\fontseries{#4}\fontshape{#5}%
  \selectfont}%
\fi\endgroup%
\begin{picture}(10570,2941)(-5396,-8964)
\put(-921,-7149){\makebox(0,0)[lb]{\smash{{\SetFigFont{8}{9.6}{\rmdefault}{\mddefault}{\updefault}{\color[rgb]{1,0,0}$\alpha$}%
}}}}
\put(-3200,-8891){\makebox(0,0)[lb]{\smash{{\SetFigFont{8}{9.6}{\rmdefault}{\mddefault}{\updefault}{\color[rgb]{1,0,0}$\alpha_{g+1}$}%
}}}}
\put(-224,-7411){\makebox(0,0)[lb]{\smash{{\SetFigFont{8}{9.6}{\rmdefault}{\mddefault}{\updefault}{\color[rgb]{0,0,0}$\mm$}%
}}}}
\put(2605,-8013){\makebox(0,0)[lb]{\smash{{\SetFigFont{8}{9.6}{\rmdefault}{\mddefault}{\updefault}{\color[rgb]{0,0,0}$w$}%
}}}}
\put(5159,-7694){\makebox(0,0)[lb]{\smash{{\SetFigFont{8}{9.6}{\rmdefault}{\mddefault}{\updefault}{\color[rgb]{1,0,0}$\alpha_1$}%
}}}}
\put(5098,-7916){\makebox(0,0)[lb]{\smash{{\SetFigFont{8}{9.6}{\rmdefault}{\mddefault}{\updefault}{\color[rgb]{1,0,0}$\alpha_2$}%
}}}}
\put(2800,-8891){\makebox(0,0)[lb]{\smash{{\SetFigFont{8}{9.6}{\rmdefault}{\mddefault}{\updefault}{\color[rgb]{0,0,1}$\beta_g$}%
}}}}
\put(5100,-7161){\makebox(0,0)[lb]{\smash{{\SetFigFont{8}{9.6}{\rmdefault}{\mddefault}{\updefault}{\color[rgb]{0,0,1}$\beta_{g+1}$}%
}}}}
\put(2956,-8011){\makebox(0,0)[lb]{\smash{{\SetFigFont{8}{9.6}{\rmdefault}{\mddefault}{\updefault}{\color[rgb]{0,0,0}$z$}%
}}}}
\put(-3261,-6194){\makebox(0,0)[lb]{\smash{{\SetFigFont{8}{9.6}{\rmdefault}{\mddefault}{\updefault}{\color[rgb]{0,0,1}$\beta_{g}$}%
}}}}
\put(2739,-6204){\makebox(0,0)[lb]{\smash{{\SetFigFont{8}{9.6}{\rmdefault}{\mddefault}{\updefault}{\color[rgb]{1,0,0}$\alpha_{g+1}$}%
}}}}
\put(-3385,-7063){\makebox(0,0)[lb]{\smash{{\SetFigFont{8}{9.6}{\rmdefault}{\mddefault}{\updefault}{\color[rgb]{0,0,0}$z$}%
}}}}
\put(-3036,-7062){\makebox(0,0)[lb]{\smash{{\SetFigFont{8}{9.6}{\rmdefault}{\mddefault}{\updefault}{\color[rgb]{0,0,0}$w$}%
}}}}
\put(-914,-7966){\makebox(0,0)[lb]{\smash{{\SetFigFont{8}{9.6}{\rmdefault}{\mddefault}{\updefault}{\color[rgb]{0,0,1}$\beta$}%
}}}}
\put(-869,-7741){\makebox(0,0)[lb]{\smash{{\SetFigFont{8}{9.6}{\rmdefault}{\mddefault}{\updefault}{\color[rgb]{0,0,1}$\beta$}%
}}}}
\end{picture}%

%% file: BaseTriple.pdf_t
\begin{picture}(0,0)%
\includegraphics{BaseTriple.pdf}%
\end{picture}%
\setlength{\unitlength}{2565sp}%
\begingroup\makeatletter\ifx\SetFigFont\undefined%
\gdef\SetFigFont#1#2#3#4#5{%
  \reset@font\fontsize{#1}{#2pt}%
  \fontfamily{#3}\fontseries{#4}\fontshape{#5}%
  \selectfont}%
\fi\endgroup%
\begin{picture}(10570,2941)(-5396,-8964)
\put(-2324,-7336){\makebox(0,0)[lb]{\smash{{\SetFigFont{8}{9.6}{\rmdefault}{\mddefault}{\updefault}{\color[rgb]{0,.56,0}$\gamma$}%
}}}}
\put(-3200,-8891){\makebox(0,0)[lb]{\smash{{\SetFigFont{8}{9.6}{\rmdefault}{\mddefault}{\updefault}{\color[rgb]{1,0,0}$\alpha_{g+1}$}%
}}}}
\put(-224,-7411){\makebox(0,0)[lb]{\smash{{\SetFigFont{8}{9.6}{\rmdefault}{\mddefault}{\updefault}{\color[rgb]{0,0,0}$\mm$}%
}}}}
\put(2605,-8013){\makebox(0,0)[lb]{\smash{{\SetFigFont{8}{9.6}{\rmdefault}{\mddefault}{\updefault}{\color[rgb]{0,0,0}$w$}%
}}}}
\put(5159,-7694){\makebox(0,0)[lb]{\smash{{\SetFigFont{8}{9.6}{\rmdefault}{\mddefault}{\updefault}{\color[rgb]{1,0,0}$\alpha_1$}%
}}}}
\put(5098,-7916){\makebox(0,0)[lb]{\smash{{\SetFigFont{8}{9.6}{\rmdefault}{\mddefault}{\updefault}{\color[rgb]{1,0,0}$\alpha_2$}%
}}}}
\put(2800,-8891){\makebox(0,0)[lb]{\smash{{\SetFigFont{8}{9.6}{\rmdefault}{\mddefault}{\updefault}{\color[rgb]{0,0,1}$\beta_g$}%
}}}}
\put(2956,-8011){\makebox(0,0)[lb]{\smash{{\SetFigFont{8}{9.6}{\rmdefault}{\mddefault}{\updefault}{\color[rgb]{0,0,0}$z$}%
}}}}
\put(-3261,-6194){\makebox(0,0)[lb]{\smash{{\SetFigFont{8}{9.6}{\rmdefault}{\mddefault}{\updefault}{\color[rgb]{0,0,1}$\beta_{g}$}%
}}}}
\put(2739,-6204){\makebox(0,0)[lb]{\smash{{\SetFigFont{8}{9.6}{\rmdefault}{\mddefault}{\updefault}{\color[rgb]{1,0,0}$\alpha_{g+1}$}%
}}}}
\put(-3385,-7063){\makebox(0,0)[lb]{\smash{{\SetFigFont{8}{9.6}{\rmdefault}{\mddefault}{\updefault}{\color[rgb]{0,0,0}$z$}%
}}}}
\put(-3036,-7062){\makebox(0,0)[lb]{\smash{{\SetFigFont{8}{9.6}{\rmdefault}{\mddefault}{\updefault}{\color[rgb]{0,0,0}$w$}%
}}}}
\put(-921,-7149){\makebox(0,0)[lb]{\smash{{\SetFigFont{8}{9.6}{\rmdefault}{\mddefault}{\updefault}{\color[rgb]{1,0,0}$\alpha$}%
}}}}
\put(3730,-7531){\makebox(0,0)[lb]{\smash{{\SetFigFont{8}{9.6}{\rmdefault}{\mddefault}{\updefault}{\color[rgb]{0,.56,0}$\gamma_g$}%
}}}}
\put(5077,-7098){\makebox(0,0)[lb]{\smash{{\SetFigFont{8}{9.6}{\rmdefault}{\mddefault}{\updefault}{\color[rgb]{0,0,1}$\beta_{g+1}$}%
}}}}
\put(5127,-7334){\makebox(0,0)[lb]{\smash{{\SetFigFont{8}{9.6}{\rmdefault}{\mddefault}{\updefault}{\color[rgb]{0,.56,0}$\gamma_{g+1}$}%
}}}}
\put(-875,-7701){\makebox(0,0)[lb]{\smash{{\SetFigFont{8}{9.6}{\rmdefault}{\mddefault}{\updefault}{\color[rgb]{0,0,1}$\beta$}%
}}}}
\put(-887,-7873){\makebox(0,0)[lb]{\smash{{\SetFigFont{8}{9.6}{\rmdefault}{\mddefault}{\updefault}{\color[rgb]{0,.56,0}$\gamma$}%
}}}}
\put(-988,-8052){\makebox(0,0)[lb]{\smash{{\SetFigFont{8}{9.6}{\rmdefault}{\mddefault}{\updefault}{\color[rgb]{0,0,1}$\beta$}%
}}}}
\put(-875,-7494){\makebox(0,0)[lb]{\smash{{\SetFigFont{8}{9.6}{\rmdefault}{\mddefault}{\updefault}{\color[rgb]{0,.56,0}$\gamma$}%
}}}}
\end{picture}%

%% file: Standard.pdf_t
\begin{picture}(0,0)%
\includegraphics{Standard.pdf}%
\end{picture}%
\setlength{\unitlength}{2565sp}%
\begingroup\makeatletter\ifx\SetFigFont\undefined%
\gdef\SetFigFont#1#2#3#4#5{%
  \reset@font\fontsize{#1}{#2pt}%
  \fontfamily{#3}\fontseries{#4}\fontshape{#5}%
  \selectfont}%
\fi\endgroup%
\begin{picture}(10559,9834)(-5396,-19059)
\put(-749,-10786){\makebox(0,0)[lb]{\smash{{\SetFigFont{8}{9.6}{\rmdefault}{\mddefault}{\updefault}{\color[rgb]{0,0,0}$\alpha$-handleslide}%
}}}}
\put(-921,-17274){\makebox(0,0)[lb]{\smash{{\SetFigFont{8}{9.6}{\rmdefault}{\mddefault}{\updefault}{\color[rgb]{1,0,0}$\alpha$}%
}}}}
\put(-887,-17998){\makebox(0,0)[lb]{\smash{{\SetFigFont{8}{9.6}{\rmdefault}{\mddefault}{\updefault}{\color[rgb]{0,.56,0}$\gamma$}%
}}}}
\put(5079,-13824){\makebox(0,0)[lb]{\smash{{\SetFigFont{8}{9.6}{\rmdefault}{\mddefault}{\updefault}{\color[rgb]{1,0,0}$\alpha$}%
}}}}
\put(5113,-14548){\makebox(0,0)[lb]{\smash{{\SetFigFont{8}{9.6}{\rmdefault}{\mddefault}{\updefault}{\color[rgb]{0,.56,0}$\gamma$}%
}}}}
\put(5079,-17274){\makebox(0,0)[lb]{\smash{{\SetFigFont{8}{9.6}{\rmdefault}{\mddefault}{\updefault}{\color[rgb]{1,0,0}$\alpha$}%
}}}}
\put(-3181,-12123){\makebox(0,0)[lb]{\smash{{\SetFigFont{8}{9.6}{\rmdefault}{\mddefault}{\updefault}{\color[rgb]{.69,0,.69}$\delta$}%
}}}}
\put(-904,-10330){\makebox(0,0)[lb]{\smash{{\SetFigFont{8}{9.6}{\rmdefault}{\mddefault}{\updefault}{\color[rgb]{1,0,0}$\alpha$}%
}}}}
\put(5148,-10926){\makebox(0,0)[lb]{\smash{{\SetFigFont{8}{9.6}{\rmdefault}{\mddefault}{\updefault}{\color[rgb]{0,.56,0}$\gamma$}%
}}}}
\put(5087,-11148){\makebox(0,0)[lb]{\smash{{\SetFigFont{8}{9.6}{\rmdefault}{\mddefault}{\updefault}{\color[rgb]{0,.56,0}$\gamma$}%
}}}}
\put(2789,-12123){\makebox(0,0)[lb]{\smash{{\SetFigFont{8}{9.6}{\rmdefault}{\mddefault}{\updefault}{\color[rgb]{.69,0,.69}$\delta$}%
}}}}
\put(5116,-10566){\makebox(0,0)[lb]{\smash{{\SetFigFont{8}{9.6}{\rmdefault}{\mddefault}{\updefault}{\color[rgb]{.69,0,.69}$\delta$}%
}}}}
\put(-854,-10566){\makebox(0,0)[lb]{\smash{{\SetFigFont{8}{9.6}{\rmdefault}{\mddefault}{\updefault}{\color[rgb]{.69,0,.69}$\delta$}%
}}}}
\put(3046,-18976){\makebox(0,0)[lb]{\smash{{\SetFigFont{8}{9.6}{\rmdefault}{\mddefault}{\updefault}{\color[rgb]{.69,0,.69}$\delta$}%
}}}}
\put(5145,-17732){\makebox(0,0)[lb]{\smash{{\SetFigFont{8}{9.6}{\rmdefault}{\mddefault}{\updefault}{\color[rgb]{0,.56,0}$\gamma$}%
}}}}
\put(5100,-18005){\makebox(0,0)[lb]{\smash{{\SetFigFont{8}{9.6}{\rmdefault}{\mddefault}{\updefault}{\color[rgb]{0,.56,0}$\gamma$}%
}}}}
\put(-861,-17745){\makebox(0,0)[lb]{\smash{{\SetFigFont{8}{9.6}{\rmdefault}{\mddefault}{\updefault}{\color[rgb]{0,.56,0}$\gamma$}%
}}}}
\put(-2251,-10763){\makebox(0,0)[lb]{\smash{{\SetFigFont{8}{9.6}{\rmdefault}{\mddefault}{\updefault}{\color[rgb]{1,0,0}$\alpha$}%
}}}}
\put(3028,-15524){\makebox(0,0)[lb]{\smash{{\SetFigFont{8}{9.6}{\rmdefault}{\mddefault}{\updefault}{\color[rgb]{.69,0,.69}$\delta$}%
}}}}
\put(-2957,-18982){\makebox(0,0)[lb]{\smash{{\SetFigFont{8}{9.6}{\rmdefault}{\mddefault}{\updefault}{\color[rgb]{.69,0,.69}$\delta$}%
}}}}
\put(5138,-14309){\makebox(0,0)[lb]{\smash{{\SetFigFont{8}{9.6}{\rmdefault}{\mddefault}{\updefault}{\color[rgb]{0,.56,0}$\gamma$}%
}}}}
\put(-3047,-13851){\makebox(0,0)[lb]{\smash{{\SetFigFont{8}{9.6}{\rmdefault}{\mddefault}{\updefault}{\color[rgb]{0,0,0}$w$}%
}}}}
\put(-718,-16089){\rotatebox{24.0}{\makebox(0,0)[lb]{\smash{{\SetFigFont{8}{9.6}{\rmdefault}{\mddefault}{\updefault}{\color[rgb]{0,0,0}$\gamma$-handleslides}%
}}}}}
\put(-862,-10980){\makebox(0,0)[lb]{\smash{{\SetFigFont{8}{9.6}{\rmdefault}{\mddefault}{\updefault}{\color[rgb]{0,.56,0}$\gamma$}%
}}}}
\put(5056,-10340){\makebox(0,0)[lb]{\smash{{\SetFigFont{8}{9.6}{\rmdefault}{\mddefault}{\updefault}{\color[rgb]{1,0,0}$\alpha$}%
}}}}
\put(2607,-11245){\makebox(0,0)[lb]{\smash{{\SetFigFont{8}{9.6}{\rmdefault}{\mddefault}{\updefault}{\color[rgb]{0,0,0}$z$}%
}}}}
\put(2912,-11243){\makebox(0,0)[lb]{\smash{{\SetFigFont{8}{9.6}{\rmdefault}{\mddefault}{\updefault}{\color[rgb]{0,0,0}$w$}%
}}}}
\put(2653,-13151){\makebox(0,0)[lb]{\smash{{\SetFigFont{8}{9.6}{\rmdefault}{\mddefault}{\updefault}{\color[rgb]{0,0,0}$z$}%
}}}}
\put(2956,-13164){\makebox(0,0)[lb]{\smash{{\SetFigFont{8}{9.6}{\rmdefault}{\mddefault}{\updefault}{\color[rgb]{0,0,0}$w$}%
}}}}
\put(2621,-17188){\makebox(0,0)[lb]{\smash{{\SetFigFont{8}{9.6}{\rmdefault}{\mddefault}{\updefault}{\color[rgb]{0,0,0}$z$}%
}}}}
\put(2957,-17187){\makebox(0,0)[lb]{\smash{{\SetFigFont{8}{9.6}{\rmdefault}{\mddefault}{\updefault}{\color[rgb]{0,0,0}$w$}%
}}}}
\put(-3359,-16608){\makebox(0,0)[lb]{\smash{{\SetFigFont{8}{9.6}{\rmdefault}{\mddefault}{\updefault}{\color[rgb]{0,0,0}$z$}%
}}}}
\put(-3063,-16614){\makebox(0,0)[lb]{\smash{{\SetFigFont{8}{9.6}{\rmdefault}{\mddefault}{\updefault}{\color[rgb]{0,0,0}$w$}%
}}}}
\put(-3382,-13851){\makebox(0,0)[lb]{\smash{{\SetFigFont{8}{9.6}{\rmdefault}{\mddefault}{\updefault}{\color[rgb]{0,0,0}$z$}%
}}}}
\put(-3363,-11245){\makebox(0,0)[lb]{\smash{{\SetFigFont{8}{9.6}{\rmdefault}{\mddefault}{\updefault}{\color[rgb]{0,0,0}$z$}%
}}}}
\put(-3045,-11243){\makebox(0,0)[lb]{\smash{{\SetFigFont{8}{9.6}{\rmdefault}{\mddefault}{\updefault}{\color[rgb]{0,0,0}$w$}%
}}}}
\put(-957,-14655){\makebox(0,0)[lb]{\smash{{\SetFigFont{8}{9.6}{\rmdefault}{\mddefault}{\updefault}{\color[rgb]{0,.56,0}$\gamma$}%
}}}}
\put(-3188,-14870){\makebox(0,0)[lb]{\smash{{\SetFigFont{8}{9.6}{\rmdefault}{\mddefault}{\updefault}{\color[rgb]{.69,0,.69}$\delta$}%
}}}}
\put(-917,-13877){\makebox(0,0)[lb]{\smash{{\SetFigFont{8}{9.6}{\rmdefault}{\mddefault}{\updefault}{\color[rgb]{1,0,0}$\alpha$}%
}}}}
\put(-930,-11222){\makebox(0,0)[lb]{\smash{{\SetFigFont{8}{9.6}{\rmdefault}{\mddefault}{\updefault}{\color[rgb]{0,.56,0}$\gamma$}%
}}}}
\put(-2277,-17748){\makebox(0,0)[lb]{\smash{{\SetFigFont{8}{9.6}{\rmdefault}{\mddefault}{\updefault}{\color[rgb]{0,.56,0}$\gamma$}%
}}}}
\put(3773,-17562){\makebox(0,0)[lb]{\smash{{\SetFigFont{8}{9.6}{\rmdefault}{\mddefault}{\updefault}{\color[rgb]{0,.56,0}$\gamma$}%
}}}}
\put(3676,-14386){\makebox(0,0)[lb]{\smash{{\SetFigFont{8}{9.6}{\rmdefault}{\mddefault}{\updefault}{\color[rgb]{0,.56,0}$\gamma$}%
}}}}
\put(-870,-14402){\makebox(0,0)[lb]{\smash{{\SetFigFont{8}{9.6}{\rmdefault}{\mddefault}{\updefault}{\color[rgb]{0,.56,0}$\gamma$}%
}}}}
\put(2778,-9396){\makebox(0,0)[lb]{\smash{{\SetFigFont{8}{9.6}{\rmdefault}{\mddefault}{\updefault}{\color[rgb]{0,.56,0}$\gamma$}%
}}}}
\put(-3202,-9406){\makebox(0,0)[lb]{\smash{{\SetFigFont{8}{9.6}{\rmdefault}{\mddefault}{\updefault}{\color[rgb]{0,.56,0}$\gamma$}%
}}}}
\put(3982,-10810){\makebox(0,0)[lb]{\smash{{\SetFigFont{8}{9.6}{\rmdefault}{\mddefault}{\updefault}{\color[rgb]{1,0,0}$\alpha$}%
}}}}
\put(2767,-18968){\makebox(0,0)[lb]{\smash{{\SetFigFont{8}{9.6}{\rmdefault}{\mddefault}{\updefault}{\color[rgb]{1,0,0}$\alpha$}%
}}}}
\put(-3224,-18986){\makebox(0,0)[lb]{\smash{{\SetFigFont{8}{9.6}{\rmdefault}{\mddefault}{\updefault}{\color[rgb]{1,0,0}$\alpha$}%
}}}}
\put(-3191,-16359){\makebox(0,0)[lb]{\smash{{\SetFigFont{8}{9.6}{\rmdefault}{\mddefault}{\updefault}{\color[rgb]{.69,0,.69}$\delta$}%
}}}}
\put(2799,-16359){\makebox(0,0)[lb]{\smash{{\SetFigFont{8}{9.6}{\rmdefault}{\mddefault}{\updefault}{\color[rgb]{.69,0,.69}$\delta$}%
}}}}
\put(2821,-12892){\makebox(0,0)[lb]{\smash{{\SetFigFont{8}{9.6}{\rmdefault}{\mddefault}{\updefault}{\color[rgb]{.69,0,.69}$\delta$}%
}}}}
\put(2790,-15534){\makebox(0,0)[lb]{\smash{{\SetFigFont{8}{9.6}{\rmdefault}{\mddefault}{\updefault}{\color[rgb]{1,0,0}$\alpha$}%
}}}}
\end{picture}%

%% file: Quadruple.pdf_t
\begin{picture}(0,0)%
\includegraphics{Quadruple.pdf}%
\end{picture}%
\setlength{\unitlength}{3355sp}%
\begingroup\makeatletter\ifx\SetFigFont\undefined%
\gdef\SetFigFont#1#2#3#4#5{%
  \reset@font\fontsize{#1}{#2pt}%
  \fontfamily{#3}\fontseries{#4}\fontshape{#5}%
  \selectfont}%
\fi\endgroup%
\begin{picture}(4536,2897)(-5396,-8939)
\put(-921,-7149){\makebox(0,0)[lb]{\smash{{\SetFigFont{9}{10.8}{\rmdefault}{\mddefault}{\updefault}{\color[rgb]{1,0,0}$\alpha$}%
}}}}
\put(-875,-7701){\makebox(0,0)[lb]{\smash{{\SetFigFont{9}{10.8}{\rmdefault}{\mddefault}{\updefault}{\color[rgb]{0,0,1}$\beta$}%
}}}}
\put(-887,-7873){\makebox(0,0)[lb]{\smash{{\SetFigFont{9}{10.8}{\rmdefault}{\mddefault}{\updefault}{\color[rgb]{0,.56,0}$\gamma$}%
}}}}
\put(-988,-8052){\makebox(0,0)[lb]{\smash{{\SetFigFont{9}{10.8}{\rmdefault}{\mddefault}{\updefault}{\color[rgb]{0,0,1}$\beta$}%
}}}}
\put(-875,-7494){\makebox(0,0)[lb]{\smash{{\SetFigFont{9}{10.8}{\rmdefault}{\mddefault}{\updefault}{\color[rgb]{0,.56,0}$\gamma$}%
}}}}
\put(-3015,-8866){\makebox(0,0)[lb]{\smash{{\SetFigFont{9}{10.8}{\rmdefault}{\mddefault}{\updefault}{\color[rgb]{.69,0,.69}$\delta_{g+1}$}%
}}}}
\put(-3066,-6241){\makebox(0,0)[lb]{\smash{{\SetFigFont{9}{10.8}{\rmdefault}{\mddefault}{\updefault}{\color[rgb]{.69,0,.69}$\delta_{g}$}%
}}}}
\put(-3362,-6213){\makebox(0,0)[lb]{\smash{{\SetFigFont{9}{10.8}{\rmdefault}{\mddefault}{\updefault}{\color[rgb]{0,0,1}$\beta_{g}$}%
}}}}
\put(-3365,-7063){\makebox(0,0)[lb]{\smash{{\SetFigFont{9}{10.8}{\rmdefault}{\mddefault}{\updefault}{\color[rgb]{0,0,0}$z$}%
}}}}
\put(-3518,-8852){\makebox(0,0)[lb]{\smash{{\SetFigFont{9}{10.8}{\rmdefault}{\mddefault}{\updefault}{\color[rgb]{1,0,0}$\alpha_{g+1}$}%
}}}}
\put(-2998,-7062){\makebox(0,0)[lb]{\smash{{\SetFigFont{9}{10.8}{\rmdefault}{\mddefault}{\updefault}{\color[rgb]{0,0,0}$w$}%
}}}}
\put(-2249,-7411){\makebox(0,0)[lb]{\smash{{\SetFigFont{9}{10.8}{\rmdefault}{\mddefault}{\updefault}{\color[rgb]{0,.56,0}$\gamma$}%
}}}}
\end{picture}%

%% file: ChainHomotopies.pdf_t
\begin{picture}(0,0)%
\includegraphics{ChainHomotopies.pdf}%
\end{picture}%
\setlength{\unitlength}{3158sp}%
\begingroup\makeatletter\ifx\SetFigFont\undefined%
\gdef\SetFigFont#1#2#3#4#5{%
  \reset@font\fontsize{#1}{#2pt}%
  \fontfamily{#3}\fontseries{#4}\fontshape{#5}%
  \selectfont}%
\fi\endgroup%
\begin{picture}(7155,3619)(-6464,-10334)
\put(-3674,-7111){\makebox(0,0)[lb]{\smash{{\SetFigFont{9}{10.8}{\rmdefault}{\mddefault}{\updefault}{\color[rgb]{0,.56,0}$\gamma$}%
}}}}
\put(-3674,-7486){\makebox(0,0)[lb]{\smash{{\SetFigFont{9}{10.8}{\rmdefault}{\mddefault}{\updefault}{\color[rgb]{.69,0,.69}$\delta$}%
}}}}
\put(-2174,-8761){\makebox(0,0)[lb]{\smash{{\SetFigFont{9}{10.8}{\rmdefault}{\mddefault}{\updefault}{\color[rgb]{.69,0,.69}$\delta$}%
}}}}
\put(-5099,-8386){\makebox(0,0)[lb]{\smash{{\SetFigFont{9}{10.8}{\rmdefault}{\mddefault}{\updefault}{\color[rgb]{0,.56,0}$\gamma$}%
}}}}
\put(-5099,-8761){\makebox(0,0)[lb]{\smash{{\SetFigFont{9}{10.8}{\rmdefault}{\mddefault}{\updefault}{\color[rgb]{1,0,0}$\alpha$}%
}}}}
\put(-3674,-10036){\makebox(0,0)[lb]{\smash{{\SetFigFont{9}{10.8}{\rmdefault}{\mddefault}{\updefault}{\color[rgb]{1,0,0}$\alpha$}%
}}}}
\put(-3674,-9661){\makebox(0,0)[lb]{\smash{{\SetFigFont{9}{10.8}{\rmdefault}{\mddefault}{\updefault}{\color[rgb]{0,0,1}$\beta$}%
}}}}
\put(-2174,-8386){\makebox(0,0)[lb]{\smash{{\SetFigFont{9}{10.8}{\rmdefault}{\mddefault}{\updefault}{\color[rgb]{0,0,1}$\beta$}%
}}}}
\put(-2474,-9211){\makebox(0,0)[lb]{\smash{{\SetFigFont{9}{10.8}{\rmdefault}{\mddefault}{\updefault}{\color[rgb]{0,0,1}$\beta$}%
}}}}
\put(-2399,-9511){\makebox(0,0)[lb]{\smash{{\SetFigFont{9}{10.8}{\rmdefault}{\mddefault}{\updefault}{\color[rgb]{.69,0,.69}$\delta$}%
}}}}
\put(-2474,-7936){\makebox(0,0)[lb]{\smash{{\SetFigFont{9}{10.8}{\rmdefault}{\mddefault}{\updefault}{\color[rgb]{.69,0,.69}$\delta$}%
}}}}
\put(-2699,-9511){\makebox(0,0)[lb]{\smash{{\SetFigFont{9}{10.8}{\rmdefault}{\mddefault}{\updefault}{\color[rgb]{1,0,0}$\alpha$}%
}}}}
\put(-4799,-7486){\makebox(0,0)[lb]{\smash{{\SetFigFont{9}{10.8}{\rmdefault}{\mddefault}{\updefault}{\color[rgb]{0,.56,0}$\gamma$}%
}}}}
\put(-4649,-7786){\makebox(0,0)[lb]{\smash{{\SetFigFont{9}{10.8}{\rmdefault}{\mddefault}{\updefault}{\color[rgb]{.69,0,.69}$\delta$}%
}}}}
\put(-4949,-7786){\makebox(0,0)[lb]{\smash{{\SetFigFont{9}{10.8}{\rmdefault}{\mddefault}{\updefault}{\color[rgb]{1,0,0}$\alpha$}%
}}}}
\put(-3749,-10261){\makebox(0,0)[lb]{\smash{{\SetFigFont{9}{10.8}{\rmdefault}{\mddefault}{\updefault}{\color[rgb]{0,0,0}$A^+_0$}%
}}}}
\put(-3674,-6886){\makebox(0,0)[lb]{\smash{{\SetFigFont{9}{10.8}{\rmdefault}{\mddefault}{\updefault}{\color[rgb]{0,0,0}$A^+_0$}%
}}}}
\put(676,-8611){\makebox(0,0)[lb]{\smash{{\SetFigFont{9}{10.8}{\rmdefault}{\mddefault}{\updefault}{\color[rgb]{0,0,0}$A^+_0(\bullet)$}%
}}}}
\put(-1349,-9511){\makebox(0,0)[lb]{\smash{{\SetFigFont{9}{10.8}{\rmdefault}{\mddefault}{\updefault}{\color[rgb]{0,0,0}$\simeq$}%
}}}}
\put(-1799,-8611){\makebox(0,0)[lb]{\smash{{\SetFigFont{9}{10.8}{\rmdefault}{\mddefault}{\updefault}{\color[rgb]{0,0,0}$A^+_0(\bullet) \oplus A^+_{0}(\circ)$}%
}}}}
\put(-2614,-7678){\makebox(0,0)[lb]{\smash{{\SetFigFont{9}{10.8}{\rmdefault}{\mddefault}{\updefault}{\color[rgb]{0,.56,0}$\gamma$}%
}}}}
\put(-4953,-9407){\makebox(0,0)[lb]{\smash{{\SetFigFont{9}{10.8}{\rmdefault}{\mddefault}{\updefault}{\color[rgb]{0,.56,0}$\gamma$}%
}}}}
\put(-3899,-8686){\makebox(0,0)[lb]{\smash{{\SetFigFont{9}{10.8}{\rmdefault}{\mddefault}{\updefault}{\color[rgb]{1,0,0}$\alpha$}%
}}}}
\put(-3449,-8686){\makebox(0,0)[lb]{\smash{{\SetFigFont{9}{10.8}{\rmdefault}{\mddefault}{\updefault}{\color[rgb]{.69,0,.69}$\delta$}%
}}}}
\put(-3449,-8311){\makebox(0,0)[lb]{\smash{{\SetFigFont{9}{10.8}{\rmdefault}{\mddefault}{\updefault}{\color[rgb]{0,0,1}$\beta$}%
}}}}
\put(-3892,-8311){\makebox(0,0)[lb]{\smash{{\SetFigFont{9}{10.8}{\rmdefault}{\mddefault}{\updefault}{\color[rgb]{0,.56,0}$\gamma$}%
}}}}
\put(-4782,-9668){\makebox(0,0)[lb]{\smash{{\SetFigFont{9}{10.8}{\rmdefault}{\mddefault}{\updefault}{\color[rgb]{1,0,0}$\alpha$}%
}}}}
\put(-4656,-9395){\makebox(0,0)[lb]{\smash{{\SetFigFont{9}{10.8}{\rmdefault}{\mddefault}{\updefault}{\color[rgb]{0,0,1}$\beta$}%
}}}}
\put(-2271,-7691){\makebox(0,0)[lb]{\smash{{\SetFigFont{9}{10.8}{\rmdefault}{\mddefault}{\updefault}{\color[rgb]{0,0,1}$\beta$}%
}}}}
\put(-6449,-8536){\makebox(0,0)[lb]{\smash{{\SetFigFont{9}{10.8}{\rmdefault}{\mddefault}{\updefault}{\color[rgb]{0,0,0}$\CFIp(Y_p(K))$}%
}}}}
\put(-1349,-7691){\makebox(0,0)[lb]{\smash{{\SetFigFont{9}{10.8}{\rmdefault}{\mddefault}{\updefault}{\color[rgb]{0,0,0}$\simeq$}%
}}}}
\end{picture}%